\documentclass[twoside,11pt]{article}

\usepackage{blindtext}

\usepackage[abbrvbib, preprint]{jmlr2e}
\usepackage{lastpage}
\usepackage{booktabs}       
\usepackage{amsmath}
\usepackage{colortbl,hhline} 
\usepackage{pifont}
\usepackage{algorithm}
\usepackage{algpseudocode}
\usepackage{graphicx}
\usepackage{tikz}
\usepackage{subfigure}
\usetikzlibrary{positioning}
\usetikzlibrary{calc} 
\usepackage{ifthen}
\usepackage{xparse} 

 \ExplSyntaxOn
\NewDocumentCommand{\aspectratio}{smo}
 {
  \hbox_set:Nn \l_tmpa_box {\includegraphics{#2}}
  \IfNoValueTF{#3}
   {
    \__student_aspectratio:nn { \box_wd:N \l_tmpa_box } { \box_ht:N \l_tmpa_box }
   }
   {
    \IfBooleanTF{#1}{ \tl_gset:Nx } { \tl_set:Nx } #3
     {
      \__student_aspectratio:nn { \box_wd:N \l_tmpa_box } { \box_ht:N \l_tmpa_box }
     }
   }
 }

\cs_new:Nn \__student_aspectratio:nn
 {
  \fp_eval:n {round( #1 / #2 , 5)}
 }
\ExplSyntaxOff

\newcommand{\neworrenewcommand}[1]{\providecommand{#1}{}\renewcommand{#1}}

\newcommand{\zoomincludgraphic}[9]{
    \neworrenewcommand{\ffoo}[5]{
\begin{tikzpicture}[x=#1, y=#1, font=\footnotesize]
\aspectratio{#2}[\imsizeratio] 

  \node[anchor = south east, inner sep=0] (image) at (1,0) {\includegraphics[width=#1]{#2}};
	    \coordinate (viewport lower left) at (#3,#4/\imsizeratio);  
	    \coordinate(viewport upper right) at (#5,#6/\imsizeratio);  
        \draw[##5, line width = ##4 pt] (viewport lower left) rectangle (viewport upper right);
 
     \pgfmathsetmacro{\multone}{#5-#3}
     \pgfmathsetmacro{\multtwo}{#6/\imsizeratio-#4/\imsizeratio}
     
     \ifthenelse{\equal{#9}{bottom_left} }{ 
	      \node[anchor=north, draw= ##3, inner sep=0pt, line width = ##2 pt,outer sep=0pt] (zoomPart) at (\multone*#7/2+##2/345*1.333, \multtwo*#7+##2/345*1.333) {
	       \scalebox{#7}{\tikz{
	         \clip (#3,#4/\imsizeratio) rectangle (#5,#6/\imsizeratio);
	           
	         \node[anchor=south east, inner sep=0] at (1,0) {\includegraphics[width=#1]{#2}}; 
	         }}};
	   \ifthenelse{\equal{##1}{line_connection_on} }{ 
		  \draw[red, dashed] (viewport upper right|-viewport lower left) -- (zoomPart.north east); 
		  \draw[red, dashed] (viewport lower left) -- (zoomPart.north west);
		   }{}
	       
	 }{}

     \ifthenelse{\equal{#9}{bottom_right} }{ 
	      \node[anchor=north, draw= ##3, inner sep=0pt, line width = ##2 pt,outer sep=0pt] (zoomPart) at (1-\multone*#7/2-##2/345*1.333, \multtwo*#7 + ##2/345*1.333) {
	       \scalebox{#7}{\tikz{
			 \clip (#3,#4/\imsizeratio) rectangle (#5,#6/\imsizeratio);
	         \node[anchor=south east, inner sep=0] at (1,0) {\includegraphics[width=#1]{#2}}; 
	         }}%
	       };
		\ifthenelse{\equal{##1}{line_connection_on} }{ 
			  \draw[red, dashed] (viewport upper right|-viewport lower left) -- (zoomPart.south west); 
			  \draw[red, dashed] (viewport upper right) -- (zoomPart.north west);
		   }{}
     }{}  
       
    \ifthenelse{\equal{#9}{up_right} }{  
	      \node[anchor=north, draw= ##3, inner sep=0pt, line width = ##2pt, outer sep=0pt] (zoomPart) at (1-\multone*#7/2-##2/345*1.333,1/\imsizeratio-##2/345*1.333) {
	       \scalebox{#7}{\tikz{
	          \clip (#3,#4/\imsizeratio) rectangle (#5,#6/\imsizeratio);
	          \node[anchor=south east, inner sep=0] at (1,0) {\includegraphics[width=#1]{#2}}; 
	         }}%
	         
	       };
	   \ifthenelse{\equal{##1}{line_connection_on} }{ 
		  \draw[red, dashed] (viewport lower left|-viewport upper right) -- (zoomPart.south west);
		  \draw[red, dashed] (viewport upper right) -- (zoomPart.south east);
		   }{}
     }{}

     \ifthenelse{\equal{#9}{up_left} }{ 
	      \node[anchor=north, draw= ##3, inner sep=0pt, line width = ##2pt,outer sep=0pt] (zoomPart) at (\multone*#7/2+##2/345*1.333, 1/\imsizeratio-##2/345*1.333) {
	       \scalebox{#7}{\tikz{
	         \clip (#3,#4/\imsizeratio) rectangle (#5,#6/\imsizeratio);
	           
	          \node[anchor=south east, inner sep=0] at (1,0) {\includegraphics[width=#1]{#2}}; 
	         }}};
		   \ifthenelse{\equal{##1}{line_connection_on} }{ 
			  \draw[red, dashed] (viewport lower left|-viewport upper right) -- (zoomPart.south west);
			  \draw[red, dashed] (viewport upper right) -- (zoomPart.south east);
			   }{}
	     }{}

  	\ifthenelse{\equal{#8}{help_grid_on} }{ 
           \begin{scope}[
                x={(image.south east)},
                y={(image.north west)},
                font=\footnotesize,
                help lines,
                overlay
            ]
            
            \draw[help lines, xstep=.1,ystep=.1,overlay] (0,0) grid (1,1);
            \foreach \x in {0,1,...,9} { 
                \node[anchor=north] at (\x/10,0) {0.\x}; 
            }
            \foreach \y in {0,1,...,9} {
                \node[anchor=east] at (0,\y/10) {0.\y};
            }
        \end{scope}    
	}{}  
   
\end{tikzpicture}

    }
    \ffoo
}
\usepackage{diagbox}

\usepackage{xcolor}  

\newcommand{\cmark}{\ding{51}}%
\newcommand{\xmark}{\ding{55}}%
\newcommand{\ie}{i.e.,}

\usepackage{multirow, makecell}
\DeclareMathOperator*{\argmin}{\arg\!\min}

\newtheorem{assumption}{Assumption}
\newsavebox\CBox
\def\tBF#1{\sbox\CBox{#1}\resizebox{\wd\CBox}{\ht\CBox}{\textbf{#1}}}
\hypersetup{hidelinks}

\definecolor{olive}{rgb}{0.6, 0.6, 0.2}
\definecolor{sand}{rgb}{0.8666666666666667, 0.8, 0.4666666666666667}
\definecolor{wine}{rgb}{0.5333333333333333, 0.13333333333333333, 0.3333333333333333}
\definecolor{deblue}{RGB}{11,132,147}
\definecolor{ocra}{RGB}{204, 119, 34}
\usepackage{enumitem}
\usepackage{tikz}
\newcommand{\fcircle}[2][red,fill=red]{\tikz[baseline=-0.5ex]\draw[#1,radius=#2] (0,0.03) circle ;}

\usepackage{lastpage}
\jmlrheading{25}{2024}{1-\pageref{LastPage}}{8/31; Revised xx/xx}{xx/xx}{21-0000}{Tsz Ching Chow$^\dagger$, Chaoyan Huang$^\dagger$, Zhongming Wu$^*$, Tieyong Zeng and Angelica I. Aviles-Rivero;  $^*$Correspondence. $^\dagger$Equal Contribution}

\ShortHeadings{Inertial Proximal DCA with Convergent Bregman Plug-and-Play}{Chow, Huang, Wu, Zeng and Aviles-Rivero}
\firstpageno{1}

\begin{document}

\title{Inertial Proximal Difference-of-Convex Algorithm with Convergent Bregman Plug-and-Play for Nonconvex Imaging}

\author{\name Tsz Ching Chow$^\dagger$ \email tcchow@math.cuhk.edu.hk \\
       \addr Department of  Mathematics\\
       The  Chinese  University of  Hong  Kong, Shatin,  Hong  Kong, China
       \AND
       \name Chaoyan Huang$^\dagger$ \email cyhuang@math.cuhk.edu.hk \\
       \addr Department of  Mathematics\\
       The  Chinese  University of  Hong  Kong, Shatin,  Hong  Kong, China
       \AND
       \name Zhongming Wu$^*$ \email wuzm@nuist.edu.cn \\
       \addr School of Management Science and Engineering\\
       Nanjing University of Information Science and Technology, Nanjing, China
        \AND
        \name Tieyong Zeng \email zeng@math.cuhk.edu.hk \\
       \addr Department of  Mathematics\\
       The  Chinese  University of  Hong  Kong, Shatin,  Hong  Kong, China
       \AND
       \name Angelica I. Aviles-Rivero
       \email  ai323@cam.ac.uk\\ \addr Department of Applied Mathematics and Theoretical Physics\\ University of Cambridge, Cambridge, UK  
       }


\maketitle

\begin{abstract}
Imaging tasks are typically tackled using a structured optimization framework. This paper delves into a class of algorithms for difference-of-convex (DC) structured optimization, focusing on minimizing a DC function along with a possibly nonconvex function. Existing DC algorithm (DCA) versions often fail to effectively handle nonconvex functions or exhibit slow convergence rates. We propose a novel inertial proximal DC algorithm in Bregman geometry, named iBPDCA, designed to address nonconvex terms and enhance convergence speed through inertial techniques. We provide a detailed theoretical analysis, establishing both subsequential and global convergence of iBPDCA via the Kurdyka-{\L}ojasiewicz property. Additionally, we introduce a Plug-and-Play variant, PnP-iBPDCA, which employs a deep neural network-based prior for greater flexibility and robustness while ensuring theoretical convergence. We also establish that the Gaussian gradient step denoiser used in our method is equivalent to evaluating the Bregman proximal operator for an implicitly weakly convex functional. We extensively validate our method on Rician noise and phase retrieval. We demonstrate that iBPDCA surpasses existing state-of-the-art methods.
\end{abstract}


\section{Introduction}
In this paper, we consider the following type of difference-of-convex (DC) composite optimization problem:
\begin{equation} \label{eq:DCA_model}
    \min_{{\bf x} \in \mathcal{X}} ~\left\{ \Psi({\bf x}):= f_1({\bf x}) - f_2({\bf x}) +g({\bf x}) \right\},
\end{equation}
where $\mathcal{X}\subseteq \mathbb{R}^n$ is a closed convex set with nonempty interior denoted by ${\rm int}(\mathcal{X})$, and $\mathbb{R}^n$ is a real finite dimensional Euclidean space.
$f_1:\mathbb{R}^n\rightarrow\mathbb{R}$ and $f_2:\mathbb{R}^n\rightarrow\mathbb{R}$ are both convex functions, and $g: \mathbb{R}^n\rightarrow (-\infty,\infty]$ is a proper closed (possibly nonconvex) function.
 We assume ${\rm dom}(g) \cap  \mathcal{X}$ is a nonempty closed set and $\Psi$ is bounded from below.

Some notable applications of Problem \eqref{eq:DCA_model} includes large-scale molecular optimization in computational biology \citep{le2014dca,ying2009enhanced}, machine learning \citep{bradley1998feature, yuille2003concave}, data mining \citep{gasso2009recovering,yin2015minimization}, signal and image processing  \citep{xiao2015convergence, lou2015weighted}. In this paper, we focus on designing efficient algorithms for Problem \eqref{eq:DCA_model} and their applications to the following nonconvex imaging problems.

\paragraph{Rician noise removal.}
The degradation model for the Rician noise removal problem \citep{chen2019variational,wu2022efficient} can be formulated as 
\begin{equation} \label{eq:Rician_noise_model}
    {\bf b} = \sqrt{(\mathcal{A}{\bf x}+\eta_1^2)+\eta_2^2}, \quad \eta_1 \sim \mathcal{N}(0,\sigma^2), \quad \eta_2 \sim \mathcal{N} (0,\sigma^2),
\end{equation}
where $\bf{x}$ represents ground truth, $\mathcal{A}$ is a forward operator, ${\bf b}$ represents image corrupted with Rician noise. The variables \(\eta_1\) and \(\eta_2\) are independent Gaussian noise, each following a normal distribution \(\mathcal{N}(0, \sigma^2)\) with zero-mean and variance \(\sigma^2\). Rician noise is more challenging to address than additive Gaussian noise because it is signal-dependent, whereas Gaussian noise is signal-independent.
Based on maximum a posterior (MAP) estimation, the nonconvex regularization model for Rician noise removal to recover ${\bf x}$ from noisy measurement ${\bf b}$ from Problem \eqref{eq:Rician_noise_model} reads
\begin{equation} \label{eq:map_Rician_model}
    \inf _{{\bf x}}\left\{F({\bf x})=\frac{1}{2 \sigma^2}\|\mathcal{A} {\bf x}\|^2-\left\langle\log \left(I_0\left(\frac{{\bf b}\mathcal{A} {\bf x}}{\sigma^2}\right)\right), \mathbf{1}\right\rangle+ \mu \phi({\bf {\bf x}})\right\},
\end{equation}
where $I_0$ is the modified Bessel function of the first kind with zeroth order \citep{gray1895treatise}, $\phi$ is the regularizing term (or prior term), and $\mu$ is the trade-off parameter between the data-fidelity term and prior term. 
The model of that~\eqref{eq:map_Rician_model} is an instance of Problem \eqref{eq:DCA_model} with 
\begin{equation} \label{eq:Rician_obj_split}
        f_1({\bf x}) := \frac{1}{2 \sigma^2}\|\mathcal{A} {\bf x}\|^2, ~ f_2({\bf x}):=\left\langle\log \left(I_0\left(\frac{{\bf b} \mathcal{A} {\bf x}}{\sigma^2}\right)\right), \mathbf{1}\right\rangle, ~ {\rm and}~ g({\bf x}) := \mu \phi({\bf x}).
\end{equation}
While the incorporation of a nonconvex plug-and-play learned prior $\phi({\bf x})$ may seem advantageous for enhancing recovery quality, existing DC-based methods face challenges in applying this model with guaranteed convergence or within a reasonable timeframe.

\paragraph{Phase retrieval.}
The phase retrieval process in imaging~\citep{fannjiang2020numerics} seeks to reconstruct the original image from the squared magnitude of its Fourier transformation~\citep{gerchberg1972practical}. This technique is utilized across diverse fields including X-ray crystallography~\citep{harrison1993phase, millane1990phase}, optics~\citep{walther1963question}, and diffraction imaging \citep{miao1999extending}. The degradation model applicable to the phase retrieval problem can be described as follows:
\begin{equation} \label{eq:phase_retrieval_model}
    {\bf d} = |\mathcal{K}{\bf x}|^2 + \omega,
\end{equation}
where  $ {\bf d} $ represents the noisy phaseless measurement. The measurement operator, $ \mathcal{K} \in \mathbb{C}^{m \times n} $, is defined by $ (\mathcal{K}{\bf x})[r] = \mathcal{F}(\mathcal{M}_r \odot {\bf x}) $, where $ \mathcal{M}_1, \mathcal{M}_2, \ldots, \mathcal{M}_r $ are diagonal matrices in $ \mathbb{R}^{m \times n} $ with modulation patterns on their diagonals. $ \mathcal{F} $ denotes the 2D Fourier transform, $ [r] $ indicates the $r$-th element (row) of its corresponding vector or matrix, and $ m $ specifies the number of measurements. Additionally, $ \omega $ signifies the additive noise, which may be either Gaussian or Poisson.
For additive white Gaussian noise $\omega \sim \mathcal{N}(0, 10^{-\frac{{\rm SNR}}{10}})$, the noise level is controlled by the signal-to-noise ratio (SNR) which is defined as ${\rm SNR} = 10 \log_{10} \left[{\left \Vert |\mathcal{K}{\bf x}|^2 \right\Vert^2} / {\left \Vert |\mathcal{K}{\bf x}|^2 - {\bf d} \right\Vert^2} \right ]$. For shot noise $\omega \sim \mathcal{N}\left(0,\alpha^2 |\mathcal{K} {\bf x}|^2\right)$, the sigma-to-noise ratio is controlled by $\alpha$. Following the setting from \citep{metzler2018prdeep,wei2022tfpnp}, the phase retrieval Problem~\eqref{eq:phase_retrieval_model} for both noise models can be expressed as 
\begin{equation} \label{eq:phase_retrieval}
    \inf_{{\bf x}} \left\{G({\bf x}) = \frac{1}{4}\Vert |\mathcal{K}{\bf x}|^2 -{\bf d} \Vert^2 + \varsigma \vartheta({\bf x}) \right\},
\end{equation}
where $ \vartheta({\bf x})$ is the regularizer term (or prior term), and $\varsigma$ is the trade-off parameter. 
We can observe that the model of that~\eqref{eq:phase_retrieval} exemplifies another instance of Problem \eqref{eq:DCA_model} with 
\begin{equation}\label{decompostion}
    f_1({\bf x}):= \frac{1}{4} \left\Vert |\mathcal{K}{\bf x}|^2\right \Vert^2 + \frac{1}{4} \Vert {\bf d} \Vert^2,~
    f_2({\bf x}):= \frac{1}{2} \left\langle {\bf d},|\mathcal{K}{\bf x}|^2\right\rangle, ~{\rm and}~ 
    g({\bf x}):= \varsigma \vartheta({\bf x}). 
\end{equation}
Note that $f_1$ does not have a globally Lipschitz continuous gradient, and $g$ may be nonconvex in this setting. It is essential to develop efficient and convergent methods to tackle these challenges.

Some existing studies concentrate on the solution methods for Problem \eqref{eq:DCA_model} or its specific instances, including the classical difference-of-convex algorithm (DCA) \citep{tao1997convex, an2005dc}, proximal DCA \citep{gotoh2018dc, wen2018proximal}, and the proximal gradient-based methods \citep{fukushima1981generalized,bolte2018first}. Nonetheless, these approaches typically require the global Lipschitz continuity of the gradient of $f_1$ or $f_1-f_2$, and/or the convexity of $g$.
On the other hand, although DCA encapsulates a wide range of applications in various fields, little attention has been paid to incorporating neural networks with DCA. Traditionally, model-based priors like variational priors are convex. However, when utilizing data-driven deep priors, such as pre-trained convolutional neural networks, the resulting deep priors often exhibit weak convexity \citep{hurault2022proximal, goujon2024learning}. Some might contend that a weakly convex prior can be integrated into the $f_1-f_2$ terms in Problem \eqref{eq:DCA_model}. However, this approach could undermine the intrinsic structure of the practical model. 
Additionally, this modification may not be viable within the framework of the Bregman distance-based methods. Above all, bundling deep prior with data fidelity term to enforce the use of DC-based methods
may compromise the denoising performance of off-the-shelf denoisers. Therefore, it is necessary to design some efficient methods to solve Problem \eqref{eq:DCA_model}, especially for the nonconvex Rician noise removal and phase retrieval problems in imaging.

\paragraph{Contributions.} In this paper, we propose a Bregman proximal DCA with inertial acceleration to address Problem \eqref{eq:DCA_model}, particularly in the setting of nonconvex $g$ and the absence of a globally Lipschitz continuous gradient for $f_1$. \textit{To the best of our knowledge, no previous studies have integrated DCA with a deep prior through the plug-and-play (PnP) framework while achieving fast convergence with inertial techniques.} Our work aims to address this research gap by conducting rigorous theoretical analysis and demonstrating practical applications of the proposed algorithms. Our main contributions are summarized next.

\fcircle[fill=wine]{2.5pt} We propose a novel inertial Bregman proximal DCA (iBPDCA) to solve Problem \eqref{eq:DCA_model}, which minimizes the sum of a DC function and a weakly convex function. This method incorporates a Bregman distance-based proximal term to broaden its applicability and employs an inertial step to accelerate practical convergence. Notably, the inertial step size is adaptively selected using a line search strategy based on two Bregman distances of the iterates. Theoretically, we establish the subsequential convergence of the method under $L$-smooth adaptable property, which is a less stringent condition than the global smoothness of $f_1$. We further establish the global convergence by leveraging the Kurdyka-{\L}ojasiewicz property. 
\vspace{0.2cm}

\fcircle[fill=wine]{2.5pt}  We extend the proposed iBPDCA by integrating deep priors as regularisers, resulting in PnP-iBPDCA.  This scheme can be seamlessly applied to a wide range of DC problems while leveraging the advantages of neural networks. Specifically, we perform the PnP method by replacing the Bregman proximal operator of $g$ with a common Gaussian denoiser. As a result, we eliminate the necessity of retraining the Bregman denoiser, which is usually tailored to each Legendre kernel. The adopted denoiser is established to be equivalent to evaluating the Bregman proximal operator of an implicit functional with weak convexity. The theoretical convergence of PnP-iBPDCA thus can be guaranteed.
\vspace{0.2cm}

\fcircle[fill=wine]{2.5pt} We apply the proposed iBPDCA and PnP-iBPDCA to address nonconvex Rician noise removal and phase retrieval problems in imaging. For Rician noise removal, the classical Euclidean norm is employed as the kernel function. In phase retrieval, we select a quartic kernel function. Additionally, we adopted a Gaussian denoiser corresponding to the Bregman proximal operator of a weakly convex implicit functional, applicable to both kernel functions and beyond. We also rigorously verify the theoretical convergence of PnP-iBPDCA in both applications.
\vspace{0.2cm}

\fcircle[fill=wine]{2.5pt} In our Rician noise removal experiments, we demonstrate the advantages of incorporating inertial techniques into DC methods. Notably, PnP-iBPDCA operates up to twice as fast as PnP-BPDCA while maintaining high computational efficiency. Moreover, it outperforms existing state-of-the-art methods in both visual and quantitative metrics for phase retrieval and Rician noise removal. Our proposed approaches leverage mathematical models to tackle these issues, providing high explainability and theoretical convergence. They can handle a broad spectrum of DC problems and forward operators without the need for extensive training. In contrast, specialized neural networks often lack transparency, achieve only empirical convergence, and require problem-specific training.
\vspace{0.2cm}

\paragraph{Organization.} The remainder of this paper is organized as follows: Section~\ref{sec:related_work} reviews related work in DCA, plug-and-play methods, and inertial acceleration in proximal algorithms. Section~\ref{sec:alg} outlines our proposed iBPDCA and its convergence results, with detailed proofs provided in Appendix~\ref{app:proof}. In Section~\ref{sec:pnp_alg}, we introduce a novel plug-and-play approach in Bregman geometry using a common Gaussian denoiser. Alongside this, we employ our proposed Bregman PnP framework and iBPDCA, leading to PnP-iBPDCA. Section~\ref{sec:experiments} describes the application of PnP-iBPDCA to Rician noise removal and phase retrieval in nonconvex imaging problems. Conclusions follow in Section~\ref{sec:conclusions}. 

\newpage
\section{Related Work} \label{sec:related_work}
In this section, we review the foundational and recent developments in Difference-of-Convex (DCA) algorithms, Plug-and-Play methods, and inertial acceleration in proximal algorithms. We explore how these areas influence our proposed iBPDCA approach, contextualizing our work within the broader optimization field.

\subsection{Difference-of-Convex Algorithm}
The classical difference-of-convex algorithm (DCA) \citep{tao1997convex, tao1998dc, an2005dc} has been extensively studied in the context of nonconvex optimization. DCA is originally developed for a specific instance of Problem \eqref{eq:DCA_model} with \( g= 0 \), which compute a subgradient \( \xi^k \in \partial f_2({\bf x}^k) \) and updates the next iterate \( {\bf x}^{k+1} \) as follows
\begin{equation} \nonumber
    {\bf x}^{k+1} = \argmin_{{\bf x} \in \mathcal{X}} \left\{ f_1({\bf x}) - \left\langle \xi^k, {\bf x}^k \right\rangle\right\}.
\end{equation}
Since then, numerous efforts have been made to broaden its applications and enhance its efficiency. We refer readers to \cite{le2018dc} for a comprehensive review of DCA. Notably, \cite{gotoh2018dc} introduced a proximal variant of DCA (PDCA), which is for Problem~\eqref{eq:DCA_model} with a convex $g$. The iterative scheme of PDCA can be read as  
\begin{equation} \nonumber
    {\bf x}^{k+1} = \argmin_{{\bf x} \in \mathcal{X}} \left\{ g({\bf x}) - \left\langle \nabla f_1({\bf x}^k) - \xi^k, {\bf x} \right\rangle +\frac{L}{2} \Vert {\bf x} - {\bf x}^k\Vert^2\right\},
\end{equation}
where $\xi^k \in \partial f_2({\bf x}^k)$, and $L>0$ is the global Lipschitz modulus of $\nabla f_1$. The PDCA can be seen as a generalization of classic DCA and proximal gradient (PG) method \citep{fukushima1981generalized}. There are many efforts have been made to enhance the performance of PDCA, including nonmonotone-enhanced PDCA \citep{lu2019nonmonotone, lu2019enhanced} and inexact PDCA \citep{yang2021inexact, yang2024inexact, nakayama2024inexact}.

Classical DCA often restricts the convex $f_1$ to the global Lipschitz gradient continuity assumption, similar to the classical PG method. \cite{bauschke2017descent} relaxed the requirement of global Lipschitz continuity of $\nabla f_1$ by introducing the concept of $L$-smooth adaptable property through Bregman distances. This lays the groundwork for the development of the two-sided extended descent lemma (known as extended descent lemma), and the Bregman proximal gradient (BPG) algorithm \citep{bolte2018first}. Building upon this, \cite{phan2023difference} developed DCAe to accommodate nonconvex function without Lipschitz continuous gradient. 
More specifically, they reformulated the objective function \eqref{eq:DCA_model} as $\Psi({\bf x}) = J({\bf x}) - H({\bf x})$ by introducing $J({\bf x}) = L h({\bf x}) + f_1({\bf x})$ and $H({\bf x}) = L h({\bf x}) - g({\bf x}) +f_2({\bf x})$, where $h$ is a Legendre kernel and $L>0$ is the adaptive Lipschitz smoothness modulus of pair $(g,h)$. Hence, their algorithm can be written as
\begin{equation} \label{eq:DCAe_iter_scheme}
    {\bf x}^{k+1} = \argmin_{{\bf x} \in \mathcal{X}} \left \{ J({\bf x}) - \left \langle L h({\bf y}^k) - \nabla g({\bf y}^k) + { \xi}^k , {\bf x}\right \rangle\right\},
\end{equation}
where ${ \xi}^k \in \partial f_2({\bf x}^k)$, and ${\bf y}^k = {\bf x}^k + \beta_k ({\bf x}^k -{\bf x}^{k-1})$ with $\beta_k$ being an extrapolation step-size. 
Since the scheme \eqref{eq:DCAe_iter_scheme} cannot be reformulated as a Bregman proximal operator or a classical proximal operator, this hinders the application of the deep PnP method. Later, \cite{takahashi2022new} introduced the Bregman proximal DCA (BPDCA), which extended PDCA \citep{wen2018proximal} to Bregman geometry. However, their proposed algorithm has slow convergence without the extrapolation technique while the extrapolation requires $g$ to be convex.

\subsection{Plug-and-Play Method}
Plug-and-Play (PnP) method is an effective learning-to-optimize approach \citep{chen2022learning}, which integrates a pre-trained neural network into part of the iteration of an optimization algorithm. 
In particular, the PnP method replaces the proximal operator with implicit denoising prior (e.g. pretrained neural network) in proximal splitting algorithms. For instance, the iterative schemes of the PG method and PnP-PG method for $\min_{{\bf x}\in\mathbb{R}^n} f({\bf x})+g({\bf x})$ can be read as
\begin{equation*}
   {\bf x}^{k+1} = {\rm prox}_{\lambda g}({\bf x}^k-\lambda \nabla f({\bf x}^k)),   
\end{equation*}
and 
\begin{equation*}
   {\bf x}^{k+1} = \mathcal{D}_{\gamma}({\bf x}^k-\lambda \nabla f({\bf x}^k)),   
\end{equation*}
respectively. The proximal operator of $\lambda g$ is defined as ${\rm prox}_{\lambda g}({\bf y})=\argmin_{{\bf x}\in\mathbb{R}^n}\{g({\bf x})+\frac{1}{2\lambda}\|{\bf x}-{\bf y}\|^2\}$ for any ${\bf y}\in\mathbb{R}^n$ and $\lambda>0$. Additionally, $\mathcal{D}_{\gamma}$ represents a pre-trained deep denoiser corresponding to ${\rm prox}_{\lambda g}$ with $\gamma=\sqrt{\lambda}$ be the input noise level of off-the-shelves Gaussian denoiser.

PnP method was first introduced by \cite{venkatakrishnan2013plug} and was based on the alternating direction method of the multipliers (ADMM) algorithm. Since then, PnP framework has been widely used in various splitting algorithms such as half-quadratic splitting \citep{zhang2017learning,zhang2021plug}, forward-backward splitting  \citep{sreehari2016plug,tirer2018image}, Douglas-Rachford splitting \citep{buzzard2018plug}, and proximal gradient descent  \citep{terris2020building}. At that time, these algorithms achieved state-of-the-art results in some practical problems such as ill-posed image restoration problems \citep[see][]{buzzard2018plug, zhang2021plug, wei2022tfpnp, wu2024extrapolated}. However, most of the aforementioned methods have no theoretical convergence guarantee.
The problem of lack of theoretical convergence in PnP method was first addressed by \cite{chan2016plug}. They have attempted to impose a bounded denoiser assumption; however, there is no assurance that the solution is a minimum or critical point of any function. Later, \cite{ryu2019plug} proposed to apply real spectral normalization for each convolutional layer to encourage non-expansiveness and convergence; but their result did not apply to the lack of strong-convexity data-fidelity terms. Others have attempted to establish convergence by restricting to nonexpansive denoisers \citep{sun2019online, reehorst2018regularization}, but its performance deteriorated empirically \citep{romano2017little, zhang2021plug} due to a lack of $1$-Lipschitz continuity in off-the-shelf denoisers. 

Most recently, \cite{hurault2022gradient} tackled the above issues by training a deep gradient-based denoiser. The provable convergence of PnP-HQS was demonstrated without compromising denoising performance. Subsequently, they found that the proposed prior was associated with a proximal operator of a weakly-convex function \citep{hurault2022proximal} and further extended their convergence result to PnP-FBS, PnP-ADMM, and PnP-DRS. Later, \cite{hurault2024convergent} extended the framework to Bregman geometry and applied it to the Poisson inverse problem. 
However, several challenges hinder us from applying the PnP framework to the existing DC algorithms. First, the DC algorithm usually requires function $g$ to be convex; while the gradient-based denoiser associates the proximal operator with a weakly-convex function. Second, some DC algorithms cannot formulate their subproblems as the evaluation of the proximal operator, which is incompatible with the PnP framework.

\begin{table}[t!] 
    \centering
    \begin{tabular}{c|ccccc}
    \hline 
        \rowcolor[HTML]{EFEFEF} \textsc{Methods} & $g$ & Inertial & Bregman & PnP & Condition for $f_1$\\ \hline
        iDCA & Vanish & \cmark & \xmark & \xmark & Global smooth \\ 
        ADCA  & Vanish & \cmark & \xmark & \xmark & Global smooth \\ 
        Boosted DCA &  Vanish & \cmark & \xmark & \xmark & Global smooth \\ 
        PDCA &   Vanish & \xmark & \xmark & \xmark  & Global smooth\\ 
        PDCAe  &   Convex & \cmark & \xmark & \xmark  & Global smooth\\ 
        DCAe  & $L$-smad & \cmark & \cmark & \xmark & Convex\\ 
        BPDCA   &  Nonconvex & \xmark & \cmark & \xmark & $L$-smad\\ 
        BPDCAe  &   Convex & \cmark & \cmark & \xmark  & $L$-smad\\ 
        \rowcolor[HTML]{D4FFD3}
        iBPDCA (Ours) &  Nonconvex & \cmark & \cmark & \cmark& $L$-smad \\ \hline
    \end{tabular}
    \caption{Comparison with existing DC-based methods regarding the property of $g$, the use of inertial acceleration techniques (inertial), the incorporation of Bregman distance-based proximal terms (Bregman), the integration of deep plug-and-play approach (PnP), and the convergence conditions for $f_1$.}
    \label{tab:comparsion_DCA}
\end{table}

\subsection{Acceleration in Proximal Algorithm}
Accelerating convergence without a significant increase in computational costs is highly desirable in both convex and nonconvex optimization. One popular strategy is to incorporate an inertial force, also known as extrapolation, into the iterative scheme.
This strategy combines iterates from the previous two iterations and updates the current iterate.
Specifically, one can add an inertial force with an extrapolation parameter $\alpha_k$ to the current iterate ${\bf x}^k$ using $\alpha_k({\bf x}^k-{\bf x}^{k-1})$ for the update of the next iterate ${\bf x}^{k+1}$, where ${\bf x}^{k-1}$ is the previous iterate.
Some well-known examples belonging to this type of acceleration method include the heavy-ball method  \citep{polyak1964some} and Nesterov’s acceleration techniques \citep{nesterov1983method, nesterov2007dual, nesterov2013gradient,nesterov2013introductory}.

In terms of accelerating DCA, \cite{aragon2018accelerating} proposed the boosted DCA, which incorporates the algorithm proposed by \cite{fukushima1981generalized}, also known as backtracking line search with Armijo condition. Later, \cite{aragon2020boosted} proved the global convergence of boosted DCA and extended their framework to the difference between two possibly nonsmooth functions. Moreover, \cite{wen2018proximal} and \cite{takahashi2022new} both adapted extrapolation parameters from FISTA \citep{beck2009fast} with a restart scheme. Their algorithms are referred to as PDCAe and BPDCAe, respectively. Even though BPDCAe tends to have faster convergence than its counterpart without extrapolation \citep[see][Tables 1 and 2]{takahashi2022new}, BPDCAe requires $g$ to be convex. To the best of our knowledge, integrating acceleration
strategies and PnP approach to the general DCA framework have not been studied in the literature.

Motivated by the aforementioned, we propose an inertial proximal DCA and integrate the Plug-and-Play (PnP) approach.  Our method offers a Bregman PnP solution that guarantees convergence for the general DC problem. 
To contextualize our approach, we compare existing methods to solve the DC problem with ours in Table~\ref{tab:comparsion_DCA}. The comparison methods include iDCA \citep{de2019inertial},  ADCA \citep{nhat2018accelerated,le2021novel}, Boosted DCA \citep{aragon2018accelerating,aragon2020boosted}, PDCA \citep{gotoh2018dc}, PDCAe \citep{wen2018proximal}, DCAe \citep{phan2023difference}, BPDCA and BPDCAe \citep{takahashi2022new}.

\section{Inertial Bregman Proximal DC Algorithm} \label{sec:alg}
In this section, we first review the notation and key results in Bregman geometry. Next, we introduce the proposed inertial Bregman proximal DC algorithm (iBPDCA), followed by a theoretical convergence analysis towards a stationary point of Problem \eqref{eq:DCA_model}. 

\subsection{Notations and Preliminaries}
Throughout the paper, we represent scalars, vectors, and matrices using lowercase letters, bold lowercase letters, and uppercase letters, respectively. We use $\mathbb{R},~\mathbb{R}_+,~\mathbb{R}^n$ and $\mathbb{R}^{m \times n}$ to represent the set of real numbers, non-negative real numbers, $n$-dimensional real vectors, and $m\times n$ real matrices, respectively. For a real matrix $M \in \mathbb{R}^{m \times n}$, we denote $M^{\top}$ as the transpose of $M$, and $\lambda_{\min}(M)$ and $\lambda_{\max}(M)$ as the minimal and maximal eigenvalues of $M$, respectively. For the set of complex numbers, n-dimensional complex vectors, and $m \times n$ complex matrices, we denote them as $\mathbb{C}, ~\mathbb{C}^n$ and $\mathbb{C}^{m \times n}$, respectively. For a given matrix $\mathcal{Z} \in \mathbb{C}^{m \times n}$, $\mathcal{Z}^{\dagger}$ represents its conjugate transpose matrix. 
We use $I_d$ to denote the identity mapping. We use $\left \langle \cdot, \cdot \right \rangle$ and $\Vert \cdot \Vert$ to denote the inner product and norm induced from the inner product. We use $\circ$ to stand for composition operator, and $\cdot$ to represent dot product.

For an extended real-valued function $f$, the domain of $f$ is defined as ${\rm dom}(f):=\{{\bf x}\in\mathbb{R}^n\;|\;f({\bf x})<\infty\}$. A function $f$ is proper if $ {\rm dom}  f \neq \emptyset$ and $f({\bf x}) > -\infty$ for any $x \in  {\rm dom} (f)$, and is closed if it is lower semicontinuous. For any subset $S \subseteq \mathbb{R}^{n}$ and any point ${\bf x}\in \mathbb{R}^{n}$, the distance from ${\bf x}$ to $S$ is defined by ${\rm dist}({\bf x},S):= \inf\left\{\|{\bf y}-{\bf x}\|\; \big| \; {\bf y}\in S\right\}$, and ${\rm dist}({\bf x},S)=\infty$ when $S=\emptyset$.

\paragraph{Bregman geometry.} 
The Bregman distance \citep{bregman1967relaxation} is a proximity measure commonly used to relax the global Lipschitz continuity assumption in first-order methods \citep{bolte2018first}. First, we review the proximity measure that is commonly referred to as Bregman distance \citep{bregman1967relaxation}.

\begin{definition} {\rm ({\bf Bregman distance})} \label{def:breg_dist} For a proper lower semicontinuous convex function $h:\mathbb{R}^n \rightarrow (-\infty, +\infty]$, which is known as kernel function, the Bregman distance $D_h: {\rm dom}(h) \times  {\rm int}{\rm dom}(h) \rightarrow \mathbb{R}_+$ is defined by 
\[ D_h({\bf x},{\bf y}) := h({\bf x}) - h({\bf y}) - \langle \nabla h({\bf y}), {\bf x} - {\bf y} \rangle.\]
\end{definition}

Through gradient inequality of the convex function, we can deduce that $h$ is convex if and only if $D_h(\mathbf{x}, \mathbf{y}) \geq 0,~\forall ~ \mathbf{x} \in  {\rm dom}(h),~ \mathbf{y} \in  {\rm int}{\rm dom}(h)$. Equality holds if and only if $\mathbf{x}=\mathbf{y}$ from strict convexity of $h$. When $h$ is $\kappa$-strongly convex (\ie\  $\nabla^2 h({\bf x}) \succeq \kappa I_d \succ 0, \forall {\bf x} \in  {\rm dom}(h)$), we have $D_h({\bf x},{\bf y}) \geq \frac{\kappa}{2}\Vert {\bf x} - {\bf y} \Vert ^2$. An indispensable property in Bregman geometry is the three-point identity \cite[Lemma 3.1]{chen1993convergence} outlined in the following lemma.
\begin{lemma} \label{lem:three-point-identity}{\rm{({\bf Three-point identity})}} 
    For any ${\bf y}, {\bf z} \in  {\rm int}{\rm dom}(h)$, ${\bf x} \in  {\rm dom}(h)$, we have
    \[
    D_h({\bf x}, {\bf z}) - D_h({\bf x}, {\bf y}) - D_h({\bf y}, {\bf z}) = \langle \nabla h({\bf y}) - \nabla h({\bf z}), {\bf x} - {\bf y} \rangle.
    \]
\end{lemma}

Unlike \cite{bolte2018first, bauschke2017descent}, we employ a relaxed notion of relative smoothness by imposing the condition solely on a restricted set $\mathcal{X}$. As discussed in \cite{yang2021inexact}, this definition of relative smoothness allows a broader range of choices for the pairs $(f, h)$.

\begin{definition}
 {\rm ({\bf Restricted $L$-smooth adaptable on $\mathcal{X}$})} \label{lem:full_extended_descent_lemma}
Let $f,h: \mathbb{R}^n \rightarrow (-\infty,\infty]$ be proper lower semicontinuous convex functions with ${\rm dom}(h)\subseteq {\rm dom}(f)$, and $f, h$ are differentiable on ${\rm int} {\rm dom}(h)$. Given a closed convex set $\mathcal{X} \subseteq \mathbb{R}^n$ with $\mathcal{X} \cap {\rm int}{\rm dom}(h)\neq \emptyset$, we say that $(f,h)$ is $L$-smooth adaptable ($L$-smad) restricted on $\mathcal{X}$ if there exists $L \geq 0$ such that
    \[
    \left| f({\bf x}) - f({\bf y}) - \langle \nabla f({\bf y}), {\bf x} - {\bf y} \rangle \right| \leq LD_h({\bf x}, {\bf y}), ~\forall {\bf x}\in \mathcal{X}\cap\textnormal{intdom}(h),~{\bf y}\in \mathcal{X}\cap\textnormal{dom}(h). \]   
\end{definition}

From Definition \ref{lem:full_extended_descent_lemma}, the property that $(f,h)$ is $L$-smooth adaptable ($L$-smad) restricted on $\mathcal{X}$ is equivalent to 
\[
    \exists\ L\geq0 \quad \text{such that }  Lh +f \text{ and } Lh -f \text{ are convex on } \mathcal{X}\cap{\rm int}{\rm dom}(h).
\]
Since $f$ and $h$ are assumed to be convex, $Lh+f$ holds trivially, and hence only $Lh-f$ needs to be verified, which is often referred to as \textit{NoLips} condition \citep{bauschke2017descent}.

Additionally, if the kernel-generating distance function $h$ is of Legendre type, as discussed in \cite[Chapter 26]{rockafellar2015convex}, 
it is referred to as a Legendre kernel. Legendre function assumption on $h$ is imposed in Section~\ref{sec:pnp_alg} to apply the Plug-and-Play framework.

\begin{definition} {\rm{({\bf Legendre functions})}}
    Let $h: X \rightarrow (-\infty, \infty]$ be the lower semicontinuous, proper, and convex function. It is called:
     \begin{itemize}
        \item[(i)] \textbf{Essentially smooth:} if $h$ is differentiable on $ {\rm int}{\rm dom}(h)$, and $\Vert \nabla h({\bf x}^k) \Vert \rightarrow \infty $ for every sequence $\{ {\bf x}^k \}_{k \in \mathbb{N}} \subset  {\rm int}{\rm dom}(h)$ converging to a boundary point of ${\rm dom}(h)$ as $k\rightarrow +\infty$.
        \item[(ii)] \textbf{Legendre Type:} if $h$ is  essentially smooth and strictly convex on $ {\rm int}{\rm dom}(h)$.
    \end{itemize}
\end{definition}

Similar to the classical proximal mapping \citep{teboulle1992entropic}, a Bregman proximal mapping  \citep{censor1992proximal,gribonval2020characterization} associated with kernel-generating distance function $h$ is defined as follows.
\begin{definition} {\rm{({\bf Bregman proximal operator})}} \label{bregman_proximal_mapping} 
Suppose the kernel function $h$ is of Legendre type, and let \( f: \mathbb{R}^n \rightarrow (-\infty, \infty] \). The Bregman proximal mapping associated with \( f \) and \( h \) is defined as
\[
    {\rm prox}_f^h({\bf y}) \in \argmin_{{\bf x}} \{ f({\bf x}) + D_h({\bf x}, {\bf y}) \}, \quad \forall {\bf y} \in  {\rm int}{\rm dom}(h).
\]
This mapping is referred to as a Bregman proximal operator. If $h\equiv \frac{1}{2} \Vert \cdot \Vert^2$, ${\rm prox}_f^h({\bf x})$ readily reduce to Moreau proximal mapping \citep{moreau1965proximite}.
\end{definition}

For other well-known but essential preliminaries of nonconvex nonsmooth optimization, including the subdifferential and the Kurdyka-{\L}ojasiewicz property, we direct readers to Appendix~\ref{app:preliminaries}.

\subsection{The proposed iBPDCA}
To present our algorithm and establish its theoretical convergence, we  establish the following foundational technical assumptions.

\begin{assumption} \label{asm:assumption_on_h}
Problem \eqref{eq:DCA_model} and the kernel function $h$ satisfy the following assumptions:
\begin{itemize}
    \item[(i)] $h:\mathbb{R}^n\rightarrow (-\infty,\infty]$ is  $\kappa$-strongly convex and $\nabla h$ is $L_h$-Lipschitz continuous on any bounded subset of $\mathbb{R}^n$, and $\overline{\textnormal{dom}(h)}=\mathcal{X}$, where $\overline{\textnormal{dom}(h)}$ denotes the closure of $\textnormal{dom}(h)$.
    \item[(ii)]  $f_1: \mathbb{R}^n \rightarrow (-\infty,+\infty]$ is a proper closed convex function with $\textnormal{dom}(h) \subseteq \textnormal{dom}(f_1)$ and $f_1$ is continuously differentiable on ${\rm int dom}(f_1)$. Moreover, $(f_1, h)$ is $L$-smad restricted on $\mathcal{X}$.
    \item[(iii)] $f_2: \mathbb{R}^n \rightarrow (-\infty,+\infty]$ is proper and convex. $g: \mathbb{R}^n \rightarrow (-\infty,+\infty]$ is proper, $\eta$-weakly convex and lower semicontinuous, with $\textnormal{dom}(g) \cap {\rm int}(\mathcal{X}) \neq \emptyset$.
    \item[(iv)] The objective function $\Psi$ is level-bounded on $\mathcal{X}$, which means for any $r \in \mathbb{R}$, the lower level sets $\{{\bf x} \in \mathcal{X} \mid \Psi({\bf x}) \leq r\}$ are bounded.
    \item[(v)] For any $\lambda > 0$, $\lambda g + h$ is supercoercive, that is,
  $\underset{\|{\bf u}\| \to \infty}{\lim} \frac{\lambda g({\bf u}) + h({\bf u})}{\|{\bf u}\|} = \infty$.
\end{itemize}
\end{assumption}

\begin{remark}
The above assumptions are commonly made for analyzing the convergence of the Bregman-type
algorithms. Note that since $\Psi$ is bounded from below, we know that
\begin{equation} \label{v(p)}
    \Psi^* = \inf \left \{\Psi({\bf x}) | x \in \mathcal{X} \right \} > -\infty
\end{equation}
follows from Assumption~\ref{asm:assumption_on_h}(ii) and (iii) that $ {\rm dom} (\Psi) \cap {\rm int}{\rm dom} (h) =  {\rm dom} (g) \cap {\rm int}{\rm dom} (h) \neq \emptyset$. Also, Assumption~\ref{asm:assumption_on_h}(v) is automatically satisfied when $\mathcal{X}$ is compact.   
\end{remark}

Before presenting the algorithm to tackle Problem \eqref{eq:DCA_model}, we define the following Bregman proximal mapping.
For ${\bf y} \in {\rm int}{\rm dom}(h)$, ${\bf z} \in {\rm dom}(f_2)$ and $\lambda >0$, we define the proximal mapping corresponding to the subproblem of iBPDCA as 
\begin{equation}\label{T_lambda}
    \begin{split}
        \mathcal{T}_{\lambda}({\bf y},{\bf z}) :&= \argmin_{{\bf u} \in \mathcal{X}} \left \{ g({\bf u}) + \left \langle \nabla f_1({\bf y }) - \xi, {\bf u-y} \right \rangle + \frac{1}{\lambda} D_h({\bf u,y}) \right\} \\
        &= \argmin_{{\bf u} \in \mathbb{R}^n} \left \{ g({\bf u}) + \left \langle \nabla f_1({\bf y }) - \xi, {\bf u-y} \right \rangle + \frac{1}{\lambda} D_h({\bf u,y}) \right\},
    \end{split}
\end{equation}
where $\xi\in\partial f_2({\bf z})$, and the second equality follows from the $\overline{{\rm dom}(h)}=\mathcal{X}$ in Assumption \ref{asm:assumption_on_h}(i). Additionally, we put the following assumption to guarantee the well-definedness of the subproblem.
\begin{assumption} \label{asm:proximal_mapping_subset_C}
    For the functions $f_1$, $f_2$, $g$ and $h$ satisfying Assumption \ref{asm:assumption_on_h}, and $\lambda >0$, we assume
 $\mathcal{T}_{\lambda} \in {\rm int} {\rm dom}(h)$ for any $ {\bf y} \in {\rm dom}(h)$, where $\mathcal{T}_{\lambda}$ is defined in \eqref{T_lambda}.   
\end{assumption}

The above assumption trivially holds when ${\rm dom}(h)=\mathbb{R}^n$ or $g$ is convex. Otherwise, it requires a classical constraint qualification condition as mentioned in \cite{bolte2018first}. 
Now we are ready to propose iBPDCA for solving Problem~\eqref{eq:DCA_model}, described in Algorithm~\ref{alg:iBPDCA}. This method mainly includes an inertial step and a Bregman proximal step, where the inertial step size is chosen adaptively by a line search strategy. 

\begin{algorithm}[t!]
\caption{Inertial Bregman Proximal DC Algorithm (iBPDCA)}
  \label{alg:iBPDCA}  
  \begin{algorithmic}[1]
    \Require{Choose a $\kappa$-strongly convex kernel function $h$ in accordance to Assumption~\ref{asm:assumption_on_h} such that $(f_1,h)$ is $L$-smad. Choose $\delta,\epsilon$ with $1> \delta \geq\epsilon>0$ and ${\rm tol}>0$.}
    \State{\textbf{Initialization.} $ {\bf x}^0={\bf x}^{-1} \in {\rm intdom}(h)\text{, and } \frac{1}{\lambda} >  \max \left \{\delta+\frac{\eta}{\kappa} , L \right\}.$} 
     \For{$k = 0, 1, 2,\ldots,$}
        \State{Compute}
        \begin{equation} \label{extrapolation_step}
            {\bf y}^{k} = {\bf x}^{k} + \beta_k \left({\bf x}^k - {\bf x}^{k-1}\right),
        \end{equation}
        
        \State{where $\beta_k \in [0,1)$ is chosen such that
        \begin{equation} \label{bregman_distance_condition_on_extrapolation}
            \lambda\left(\delta-\epsilon\right)D_h({\bf x}^{k-1},{\bf x}^k) \geq D_h({\bf x}^k,{\bf y}^k).
        \end{equation}} 
        
        \If {${\bf y}^k \notin {\rm int}{\rm dom}(h)$,} \label{if}
            \State{${\bf y}^k = {\bf x}^k$}.
        \EndIf
        
        \State{Take $\xi^k \in \partial f_2({\bf x}^k)$, and compute}
        \begin{equation} \label{subproblem:main_subproblem_of_iBPDCA}
            {\bf x}^{k+1} = \argmin_{{\bf x}\in \mathcal{X}} {\left\{ g({\bf x}) +\left \langle \nabla f_1({\bf y}^k)-\xi^k, {\bf x}-{\bf y}^k \right \rangle +\frac{1}{\lambda} D_h({\bf x},{\bf y}^k) \right\}}.
        \end{equation}
            \If{$\frac{\|{\bf x}_{k+1} - {\bf x}_{k}\|}{\|{\bf x}^k\|} < {\rm tol}$}
            break
        \EndIf
    \EndFor
  \end{algorithmic}
\end{algorithm}

\begin{remark} \label{remark:range_of_beta_k}
    It is easy to verify that $\beta_k = 0$ always satisfies \eqref{bregman_distance_condition_on_extrapolation}. Hence, iBPDCA will reduce to BPDCA \citep{takahashi2022new}. If $h = \frac{1}{2} \Vert \cdot \Vert^2$, the line search condition \eqref{bregman_distance_condition_on_extrapolation} can be reduced to $\lambda (\delta - \epsilon) \Vert x^{k-1} - x^k \Vert^2 \geq \Vert \beta_k (x^k -x^{k-1} )\Vert^2$, or simply $\beta_k\leq \sqrt{\lambda (\delta - \epsilon)} $. 
    As shown in \cite{mukkamala2020convex}, there exists a value $\gamma_k$ such that Inequality~\eqref{bregman_distance_condition_on_extrapolation} is satisfied for all $\beta_k \in [0,\gamma_k]$. A backtracking line search is employed when the choice of $\beta_k$ is not deterministic, as in the case where $h = \frac{1}{2} \| \cdot \|^2$. At each iteration, we initialize $\beta_k = \frac{\mu_k - 1}{\mu_k}$\vspace{0.05in} and update it by decreasing $\beta_k$ to $c \beta_k$ for some constant $c \in (0, 1)$ until the condition~\eqref{bregman_distance_condition_on_extrapolation} is satisfied. The value of $\mu_k$ is updated every iteration using the formula $\mu_k = \frac{1 + \sqrt{1 + 4\mu_k^2}}{2}$, with an initial value of $\mu_0 = 1$.    
    \end{remark}
    
Later on, we demonstrate that $\epsilon$ quantifies the reduction in the objective function relative to the optimal solution at each iteration.
To ensure the well-definedness of subproblem \eqref{subproblem:main_subproblem_of_iBPDCA}, we need to ensure the auxiliary variable ${\bf y}^k$ defined in \eqref{extrapolation_step} belongs to $ {\rm int}{\rm dom}(h)$. 
As shown in \cite{bolte2018first,mukkamala2020convex,takahashi2022new}, when $ {\rm dom}(h) =\mathbb{R}^n$, ${\bf y}^k$  always stay in ${\rm int}{\rm dom} (h)$. If $ {\rm dom}(h) \neq \mathbb{R}^n$, we set ${\bf y}^k ={\bf x}^k$ if ${\bf y}^k \notin  {\rm int}{\rm dom}(h)$ to guarantee well-definedness of ${\bf y}^k$. This leads us to line~\ref{if} in Algorithm~\ref{alg:iBPDCA}.

\subsection{Convergence of iBPDCA} \label{sec:convergence_iBPDCA}
 In this subsection, we are devoted to establishing the subsequential and global convergence of iBPDCA (Algorithm \ref{alg:iBPDCA}), based on Assumption~\ref{asm:assumption_on_h},~\ref{asm:proximal_mapping_subset_C} and an auxiliary Lyapunov function.
Let $\{{\bf x}^k\}_{k=0}^\infty$ be a sequence generated by iBPDCA. We define, at iterate $k \in \mathbb{N}$, the following auxiliary Lyapunov function for $\delta>0$,
\begin{equation} \label{eq:definition_Hk}
    H_\delta({\bf x},{\bf y})= \Psi({\bf x}) + \delta D_h({\bf y},{\bf x}),\quad \forall {\bf x} \in  {\rm dom}(h), ~{\bf y}\in  {\rm int}{\rm dom}(h).
\end{equation}
We first show the sufficient decrease property of $H_\delta$ in the following lemma, whose proof can be found in Appendix~\ref{proof:descending_property_auxiliary_function}.

\begin{lemma} {\rm{({\bf Sufficient decrease property of $H_\delta$})}}
\label{lem:descending_property_auxiliary_function}
    Suppose Assumptions~\ref{asm:assumption_on_h} and \ref{asm:proximal_mapping_subset_C} hold. Let $\{{\bf x}^k\}_{k=0}^\infty$ be a sequence generated by iBPDCA presented in Algorithm~\ref{alg:iBPDCA}. Then, it holds that
      \begin{equation} \label{inequality_for_auxiliary_function}
    \begin{aligned}
        H_\delta ({\bf x}^k,{\bf x}^{k-1}) &\geq H_\delta({\bf x}^{k+1},{\bf x}^k) + \left(\frac{1}{\lambda} -\frac{\eta}{\kappa}-\delta \right)D_h({\bf x}^k,{\bf x}^{k+1}) \\
        &\quad + \left(\frac{1}{\lambda} -L \right) D_h({\bf x}^{k+1},{\bf y}^k) + \epsilon D_h({\bf x}^{k-1},{\bf x}^k).
    \end{aligned}
    \end{equation} 
    Moreover, the sequence $\{H_\delta\}_{k=0}^\infty$ is non-increasing.
\end{lemma}

The above fact yields the following result, whose proof can be found in Appendix~\ref{proof:convergence_property_of_Dh}.

\begin{proposition} \label{prop:convergence_property_of_Dh}
   Suppose Assumptions~\ref{asm:assumption_on_h} and \ref{asm:proximal_mapping_subset_C} hold. Let $\{{\bf x}^k\}_{k=0}^\infty$ be a sequence generated by iBPDCA presented in Algorithm~\ref{alg:iBPDCA}. Then, the following statements hold:
    \begin{itemize} 
    \item[(i)] $\sum_{k=1}^\infty D_h({\bf x}^{k-1},{\bf x}^k)<\infty$; hence, the sequence $\{D_h({\bf x}^{k-1},{\bf x}^k)\}^\infty_{k=0}$ converges to zero.
    \item[(ii)] $\min_{1\leq k \leq n}D_h({\bf x}^{k-1},{\bf x}^k)\leq \frac{1}{n\epsilon} (\Psi({\bf x}^0) - \Psi^*)$, where $\Psi^*= \inf_{\bf x} \Psi({\bf x}) >-\infty$.
    \end{itemize}
\end{proposition}

Now we present the result of subsequential convergence for the proposed iBPDCA, and its proof is presented in Appendix~\ref{proof:thm:subsquential_convergence}.

\begin{theorem}{\rm ({\bf Subsequential convergence of iBPDCA})} \label{thm:subsquential_convergence} 
Suppose Assumptions~\ref{asm:assumption_on_h} and \ref{asm:proximal_mapping_subset_C} hold. Let $\{{\bf x}^k\}_{k=0}^\infty$ be a sequence generated by iBPDCA presented in Algorithm~\ref{alg:iBPDCA}. Then the following statements hold:
\begin{itemize}
\item[(i)] The sequence $\{{\bf x}^k\}^\infty_{k=0}$ is bounded.
\item[(ii)]  The sequence $\{\xi^k\}^\infty_{k=0}$ is bounded.
\item[(iii)] $\lim_{k \rightarrow \infty} \Vert {\bf x}^{k+1}-{\bf x}^k \Vert =0.$
\item[(iv)] Any accumulation point of $\{{\bf x}^k\}^\infty_{k=0}$ is a limiting critical point of Problem \eqref{eq:DCA_model}.
\end{itemize}
\end{theorem}

Next, we study the behavior of the sequence $\{\Psi({\bf x}^k)\}_{k=0}^\infty$ for a sequence $\{{\bf x}^k\}_{k=0}^\infty$ generated by iBPDCA. The result will be used in establishing the global convergence of the whole sequence $\{{\bf x}^k\}_{k=0}^\infty$  under additional assumptions. The proof can be found in Appendix~\ref{proof:proporty_accumulation_point}.

\begin{proposition} \label{prop:proporty_accumulation_point}
Suppose Assumptions~\ref{asm:assumption_on_h} and \ref{asm:proximal_mapping_subset_C} hold. Let $\{{\bf x}^k\}_{k=0}^\infty$ be a sequence generated by iBPDCA presented in Algorithm~\ref{alg:iBPDCA}. Then, we obtain the following conclusions:
\begin{itemize} 
    \item[(i)] $\zeta := \lim_{k\rightarrow\infty} \Psi({\bf x}^k) =\lim_{k\rightarrow \infty}H_\delta({\bf x}^k,{\bf x}^{k-1})$ exists.
    \item[(ii)] $\Psi \equiv \zeta$ on $\Omega$, where $\Omega$ is the set of accumulation points of $\{{\bf x}^k\}^\infty_{k=0}$.
\end{itemize}
\end{proposition}

To analyze the global convergence property of the sequence $\left\{{\bf x}^k\right\}^\infty_{k=0}$ generated by iBPDCA, we need to make the following additional assumption.

\begin{assumption} \label{asm:f2_local_continuity}
The functions $f_2$ and $h$ satisfy the following assumptions:
    \begin{enumerate}
    \item[(i)] $f_2$ is continuously differentiable on an open set $\mathcal{N}_0\subseteq \mathbb{R}^n$ that contains all the limiting critical points of $\Psi$ (\ie\  ${\rm crit} \Psi$), and $\nabla f_2$ is locally Lipschitz continuous on $\mathcal{N}_0$.
    \item[(ii)] $h$ has bounded second derivative on any compact subset $B \subset  {\rm int}{\rm dom}(h)$.
    \end{enumerate}
\end{assumption}

The above, together with the Kurdyka-{\L}ojasiewicz (KL) property (see \cite{bolte2007lojasiewicz} and Appendix \ref{app:preliminaries} for details) allows us to establish global convergence of iBPDCA. The proof can be found in Appendix~\ref{proof:global_convergence}.

\begin{theorem} {\rm{({\bf Global convergence of iBPDCA})} }\label{thm:global_convergence}
Suppose Assumptions~\ref{asm:assumption_on_h},~\ref{asm:proximal_mapping_subset_C}, and~\ref{asm:f2_local_continuity} hold, and the auxiliary function $H_\delta$ is a KL function. Let $\{ {\bf x}^k\}^\infty_{k=0}$ be the sequence generated by iBPDCA presented in Algorithm \ref{alg:iBPDCA}. Then the following statements hold:
\begin{enumerate}
    \item[(i)] $\lim_{k\rightarrow\infty} {\rm dist}\left({\bf (0, 0)},\partial H_\delta({\bf x}^k,{\bf x}^{k-1})\right)=0$. 
    \item[(ii)] The set of accumulation points of $\left\{({\bf x}^k,{\bf x}^{k-1}) \right\}^\infty_{k=0}$ is $\Upsilon:= \{({\bf x},{\bf x})\mid {\bf x} \in \Omega \}$ and $H_\delta \equiv \zeta$ on $\Upsilon$, where $\Omega$ is the set of accumulation point of $\{{\bf x}^k\}^\infty_{k=0}$.
    \item[(iii)]  The sequence $\left\{{\bf x}^k \right\}^\infty_{k=0}$ converges to a limiting critical point of Problem~\eqref{eq:DCA_model}, and $\sum_{k=1}^\infty \Vert {\bf x}^k - {\bf x}^{k-1}\Vert < \infty$.
\end{enumerate}
\end{theorem}

\section{Plug-and-Play iBPDCA}
\label{sec:pnp_alg}
In this section, we focus on developing the iBPDCA (Algorithm~\ref{alg:iBPDCA}) in conjunction with a deep prior, namely plug-and-play inertial proximal difference-of-convex algorithm (PnP-iBPDCA). 
Similar to the PnP concept in Euclidean space, our approach seeks to substitute the evaluation of the Bregman proximal operator with a PnP deep prior in the iterative scheme. Specifically, we opt for the Gaussian gradient step denoiser and investigate its theoretical connection with the Bregman proximal operator of a specific implicit functional. While previous works concentrate on adapting the Bregman denoiser to a particular $h$, our method offers a new perspective for a variety of kernel functions by utilizing the same Gaussian denoiser. Our approach eliminates the need to retrain the Bregman denoiser when handling Rician noise removal and phase retrieval tasks (more details can be found in Section \ref{sec:experiments}).
In the following, we first explore the relationship between the gradient step denoiser and Bregman proximal operator to establish a Bregman PnP denoiser; then we propose a novel PnP-iBPDCA framework with convergence guarantee.

\subsection{Bregman PnP Denoiser}\label{subsec4-1}
To ensure $h$ is of Legendre type, we have to make the following additional assumption, which is crucial for reformulating the subproblem \eqref{subproblem:main_subproblem_of_iBPDCA} as a Bregman proximal operator. 
\begin{assumption} \label{asm:essentially_smooth}
    For a given kernel function $h$ satisfying Assumption \ref{asm:assumption_on_h}, we further suppose that $h$ is essentially smooth.
    \end{assumption}
With the above assumption, we now have $h$, which is of Legendre type, and its properties are showcased in the following proposition.

\begin{proposition} {\rm{({\bf Legendre Type})}} \label{prop:Legendre}\cite[Theorem 26.5]{rockafellar2015convex} 
        A convex kernel function $h$ is of Legendre type, if and only if its conjugate $h^*$ is also of Legendre type. Moreover, $\nabla h:  {\rm int}{\rm dom}(h) \rightarrow  {\rm int}{\rm dom}(h^*)$ is bijective map and have following properties:
        \[
            (\nabla h)^{-1} = \nabla h^*,~~ h^*(\nabla h({\bf x})) = \left \langle {\bf x}, \nabla h({\bf x})\right \rangle -h({\bf x}),
        \] 
        and
        \[
             {\rm dom}(\partial h) =  {\rm int}{\rm dom}(h) \text{ with } \partial h({\bf x}) =\{\nabla h({\bf x})\}, \quad \forall {\bf x} \in  {\rm int}{\rm dom}(h).
        \]
\end{proposition}

By leveraging Assumption \ref{asm:assumption_on_h} and Proposition \ref{prop:Legendre}, we can reformulate the subproblem \eqref{subproblem:main_subproblem_of_iBPDCA} into the evaluation of a Bregman proximal operator.
\begin{lemma}
Let $h$ be a kernel function satisfying Assumption~\ref{asm:assumption_on_h},  Assumption~\ref{asm:essentially_smooth} and \[\nabla h({\bf y}^k) - \lambda \left(\nabla f_1({\bf y}^k) -\xi^k\right) \in  {\rm dom}(h^*),\] then the subproblem~\eqref{subproblem:main_subproblem_of_iBPDCA} can be reformulated as
\begin{equation} \nonumber
     {\bf x}^{k+1} = {\rm prox}_{\lambda g}^h \circ \nabla h^* \left(\nabla h({\bf y}^k) - \lambda \left(\nabla f_1({\bf y}^k) -\xi^k\right)\right).
\end{equation}
\end{lemma}
\begin{proof}
From the first-order optimality condition of the subproblem~\eqref{subproblem:main_subproblem_of_iBPDCA}, we obtain
\[ 
    0 \in \partial g({\bf x}) + \nabla f_1({\bf y}^k) -\xi^k +\frac{1}{\lambda} \left(\nabla h({\bf x}) - \nabla h({\bf y}^k)\right) . 
\]
With the condition that  $\nabla h({\bf y}^k) - \lambda \left(\nabla f_1({\bf y}^k) -\xi^k\right) \in  {\rm dom}(h^*)$, we can define
\[p_\lambda({\bf y}^k) = \nabla h^*\left( \nabla h({\bf y}^k) -\lambda \left(\nabla f_1({\bf y}^k)-\xi^k\right)\right).\]
According to Proposition~\ref{prop:Legendre}, we have $\nabla h^* = (\nabla h)^{-1}$ for all Legendre type function $h$. Therefore, $\nabla h \circ \nabla h^* = I$ and we can simplify the first-order optimality condition as 
\[
    0 \in \partial g({\bf x}) +\frac{1}{\lambda} \left(\nabla h({\bf x}) - \nabla h\left(p_\lambda({\bf y}^k)\right) \right).  
\]
Then, with the definition of Bregman proximal mapping \eqref{bregman_proximal_mapping}, we can rewrite ${\bf x}^{k+1}$ as 
\begin{equation} \label{eq:subproblem_breg_prox_oper}
    \begin{aligned}
    {\bf x}^{k+1} &= \argmin_{{\bf x} \in \mathcal{X}} { \left\{ g({\bf x})  +\frac{1}{\lambda} D_h(x,p_\lambda({\bf y}^k)) \right \}} \\
    &= {\rm prox}_{\lambda g}^h \left(p_\lambda({\bf y}^k) \right)\\
    &= {\rm prox}_{\lambda g}^h \circ \nabla h^* \left(\nabla h({\bf y}^k) - \lambda \left(\nabla f_1({\bf y}^k) -\xi^k\right)\right).
    \end{aligned}
\end{equation}
This completes the proof.
\end{proof}

As we can see, the subproblem ${\bf x}^{k+1}$ can be seen as the evaluation of a Bregman proximal operator. In Euclidean geometry, an off-the-shelf denoiser can replace the proximal operator, transforming the original minimization problem into a weakly-convex potential \citep{hurault2022gradient, hurault2022proximal}. However, a generalized PnP framework in Bregman geometry has not yet been established. \cite{hurault2024convergent} has recently proposed a novel approach that replaces the Bregman proximal operator with a Bregman score denoiser $\mathcal{B}_\gamma$ to solve the Poisson inverse problem with Burg entropy (\ie\ $h({\bf x}) = -\sum_i x_i$). \textit{The exploration of generalizing their proposed method to other kernel functions $h$ remains an open question.} For instance, the kernel function $h= \frac{1}{4} \Vert \cdot \Vert^4 +\frac{1}{2} \Vert \cdot \Vert^2$ for phase retrieval problem cannot be associated with a probability distribution in the Bregman noise model as follows:
\[
    \text{for } {\bf x , y} \in  {\rm dom}({\bf x}) \times  {\rm int}{\rm dom}(h), \quad p({\bf y|x}) := \exp\{-\gamma D_h({\bf x,y}) + \rho({\bf x})\},
\]
since there is no probability distribution can be associated with 
\[
    p({\bf y|x}) := \exp \left[ -\gamma\left(\frac{1}{2}\Vert {\bf x}- {\bf y} \Vert^2 + \frac{1}{4}\Vert {\bf x} \Vert^4 - \frac{1}{4} \Vert {\bf y} \Vert^4 -\left \langle \Vert {\bf y} \Vert^2 {\bf y, x-y} \right \rangle \right) + \rho({\bf x}) \right].
\]

Motivated by the above example, we correlate the gradient step denoiser \citep{hurault2022proximal,hurault2022gradient} with the Bregman proximal operator to perform PnP. Our method accommodates commonly used Legendre kernels in solving DC problems.

First, we have to review some properties for Bregman proximal mapping. Previous work from \cite{hurault2024convergent} has extended the work from \cite{gribonval2020characterization}, which characterize the Bregman proximal mapping $\varphi({\bf y}) \in \argmin_{{\bf x}} \left\{ D_h({\bf y}, {\bf x}) + \phi({\bf x}) \right\}
$. However, Bregman distance is asymmetric \citep{bauschke2017descent} and hence it is not suitable for solving problems like~\eqref{eq:subproblem_breg_prox_oper}. \cite{hurault2024convergent} modified the relation as:

\begin{proposition}\label{prop:charac_bregman} \citep{hurault2024convergent}
Let $h$ be a function of Legendre type on $\mathbb{R}^n$. 
Let $\psi : \mathbb{R}^n \to \mathbb{R}\cup \{ +\infty \}$ be a proper lower-semicontinuous convex function on ${\rm int} {\rm dom}(h^*)$, which satisfy ${\varphi(\nabla h^*({\bf z})) \in \partial \psi({\bf z})}$, ${\bf z} \in {\rm int} {\rm dom}(h^*)$, for a certain function $\varphi: {\rm int} {\rm dom}(h) \to \mathbb{R}^n$. 
Then, for the function $\phi$ defined by
\begin{equation} \nonumber
    \phi({\bf x}) := \left\{\begin{array}{ll}
        \langle \varphi({\bf y}), \nabla h({\bf y}) \rangle - h(\varphi({\bf y})) - \psi(\nabla h({\bf y})) \ \ \text{for} \ {\bf y} \in \varphi^{-1}({\bf x}), \ \ &\text{if} \ {\bf x} \in {\rm Im}(\varphi), \\
        \hspace{-3pt}+\infty, \ \ &\text{otherwise}.
    \end{array}\right.
\end{equation}
It holds that ${\rm Im}(\varphi) \subset {\rm dom}(\phi)$, and for each ${\bf y} \in {\rm  int  dom}(h)$,
\begin{equation} \nonumber
    \varphi({\bf y}) \in \argmin_{{\bf x} \in \mathbb{R}^n} \left\{ D_h({\bf x},{\bf y}) + \phi({\bf x}) \right\}.
\end{equation}
\end{proposition}

Employing a similar methodology as \cite{hurault2022gradient, hurault2022proximal, hurault2024convergent}, we define a Gaussian gradient step (GS) denoiser as 
\begin{equation}\label{gaussian_denoiser}
    \mathcal{D}_\gamma := I - \nabla g_\gamma, 
\end{equation}
where $g_\gamma$ is the deep differentially parameterized potential and $\nabla g_\gamma$ should be $L_{g_\gamma}$-Lipschitz continuous with $L_{g_\gamma}<1$, defined by 
\begin{equation}\label{g_gamma}
    g_\gamma({\bf x}) = \frac{1}{2} \Vert {\bf x} - N_\gamma({\bf x}) \Vert^2,
\end{equation}
with $N_\gamma$ the deep convolutional neural network. In practice, we use lightweight DRUNet \citep{zhang2021plug,hurault2022gradient} to construct $N_\gamma$. 
As discussed in \cite{hurault2022proximal}, $\mathcal{D}_\gamma$ defined in \eqref{gaussian_denoiser} is injective, and $\forall x \in \mathcal{X}$, $\mathcal{D}_\gamma({\bf x})=\operatorname{Prox}_{\theta_\gamma}({\bf x})$, with $\theta_\gamma: \mathcal{X} \rightarrow \mathbb{R} \cup\{+\infty\}$ defined by
\begin{equation}\label{theta_gamma}
   \theta_\gamma({\bf x}):=\left\{\begin{array}{lc}
 g_\gamma (\mathcal{D}_\gamma^{-1}({\bf x}))-\frac{1}{2}\left\|\mathcal{D}_\gamma^{-1}({\bf x})-{\bf x}\right\|^2, & \text { if } {\bf x} \in \operatorname{Im}\left(\mathcal{D}_\gamma\right), \\
+\infty, 
& \text {otherwise}.
\end{array}\right.
\end{equation}
Moreover, it follows from Propostion 3.1 of \cite{hurault2022proximal} that $\theta_\gamma$ is $\frac{L_{g_\gamma}}{1+g_\gamma}$-weakly convex. To ensure the differentiability of the network, we follow the approach as \cite{hurault2022proximal} and replace the ReLU activation functions with softplus activation. Next, we discuss the relationship between GS denoiser $\mathcal{D}_\gamma$ and the Bregman proximal operator.

\begin{proposition} \label{prop:bregman_denoiser_defined}
    Let $h$ be $\mathcal{C}^2$ and of Legendre type on $\mathbb{R}^n$ and $\psi_\gamma : \mathbb{R}^n \rightarrow \mathbb{R} \cup \{+ \infty \}$ be a function with first-order derivatives as
    \begin{equation}\label{add1}
         \nabla \psi_\gamma({\bf y}):= \begin{cases} \nabla^2 h({\bf y}) \cdot ({\bf y}-\nabla g_\gamma ({\bf y})), & \text { if } {\bf y} \in  {\rm int}{\rm dom}(h),\\
        +\infty, & \text { otherwise,}\end{cases}       
    \end{equation}
    where $g_\gamma:\mathbb{R}^n \rightarrow \mathbb{R} \cup \{+ \infty \}$ be a proper and differentiable potential defined in \eqref{g_gamma}.
   Suppose $\psi_\gamma \circ \nabla h^*$ is convex on $ {\rm int}{\rm dom}(h^*)$ and ${\rm Im}(\mathcal{B}_\gamma )\subset {\rm dom}(\phi_\gamma) = {\rm int}{\rm dom}(h)$. Then,  for the functional $\phi_\gamma: \mathbb{R}^n \rightarrow \mathbb{R} \cup \{+\infty\}$ defined by
\begin{equation}\label{add2}
   \phi_\gamma({\bf x}) := \begin{cases}
       \left \langle {\bf x} , \nabla h({\bf y}) \right \rangle - h({\bf x}) -\psi_\gamma({\bf y}) \text{ for } {\bf y} \in\mathcal{D}_\gamma^{-1}({\bf x}), & \text { if } {\bf x}\in{\rm Im}(\mathcal{D}_\gamma),
       \\ +\infty, & \text { otherwise,}
   \end{cases}
\end{equation}
we have  
\begin{equation} \nonumber
    \mathcal{D}_\gamma ({\bf y}) \in {\rm prox}_{\phi_\gamma}^h ({\bf y}),  
\end{equation} 
for any $~ {\bf y} \in  {\rm int}{\rm dom}(h)$, where $\mathcal{D}_\gamma$ is the GS denoiser in \eqref{gaussian_denoiser}.
\end{proposition}
 \begin{proof}
    We first define $ \mathcal{B}_\gamma : {\rm int} {\rm dom} (h) \rightarrow \mathbb{R}^n$ for all $ {\bf y} \in {\rm int} {\rm dom}(h)$ by
    \begin{equation} \label{eq:breg_denoiser_def}
        \mathcal{B}_\gamma({\bf y}) =\nabla (\psi_\gamma \circ \nabla h^*) \circ \nabla h({\bf y}).
    \end{equation}
    Then, we can verify that our given $\mathcal{B}_\gamma$ satisfy requirements in Proposition~\ref{prop:charac_bregman} with $\varphi \equiv \mathcal{B}_\gamma$. Since $h$ is of Legendre type, we have $\nabla h^* : {\rm int} {\rm dom}(h^*) \rightarrow {\rm int} {\rm dom}(h)$ is bijective map from Proposition~\ref{prop:Legendre}. For each ${\bf z} \in {\rm int} {\rm dom}(h^*)$, we have $\nabla h^* ({\bf z}) \in {\rm int} {\rm dom}(h)$. Hence,
    \[
        \mathcal{B}_\gamma(\nabla h^*({\bf z})) =\nabla (\psi_\gamma \circ \nabla h^*) ({\bf z}),
    \]
    where $\psi_\gamma \circ \nabla h^*$ is the $\psi$ in Proposition~\ref{prop:charac_bregman}. Since we have already assumed that  $\psi_\gamma \circ \nabla h^*$ is convex on ${\rm int} {\rm dom} (h^*)$, we can guarantee the existence of $\phi_\gamma: \mathbb{R}^n \rightarrow \mathbb{R} \cup \{ +\infty \}$ through Proposition~\ref{prop:charac_bregman}. For any ${\bf y} \in {\rm int} {\rm dom} (h)$ we have
    \begin{equation}\label{add3}\begin{split}
     \mathcal{B}_\gamma({\bf y}) &\in \argmin_{{\bf x} \in \mathbb{R}^n} \{ D_h({\bf x},{\bf y}) + \phi_\gamma ({\bf x})\} \\
        &= {\rm prox}_{\phi_\gamma}^h ({\bf y}).   
    \end{split}       
    \end{equation}
    Furthermore, with Proposition~\ref{prop:charac_bregman}, $\phi_\gamma$ can be expressed explicitly as
    \[
        \phi_\gamma({\bf x}) := \begin{cases} 
            \left \langle \mathcal{B}_\gamma({\bf y}), \nabla h({\bf y}) \right \rangle - h(\mathcal{B}_\gamma({\bf y})) - \psi_\gamma \circ \nabla h^*(\nabla h({\bf y})) ~ {\rm for} ~ {\bf y} \in \mathcal{B}_\gamma^{-1}({\bf x}),\!\! & {\rm if} ~ {\bf x} \in {\rm Im}(\mathcal{B}_\gamma), \\
            +\infty, & {\rm otherwise.}
        \end{cases}
    \]
    It follows from ${\bf y} \in \mathcal{B}_\gamma^{-1} ({\bf x})$ that ${\bf x} = \mathcal{B}_\gamma({\bf y})$. This implies
    \begin{align*}
        \phi_\gamma(\mathcal{B}_\gamma({\bf y})) &= \left \langle\mathcal{B}_\gamma({\bf y}) , \nabla h({\bf y}) \right \rangle - h(\mathcal{B}_\gamma({\bf y})) -\psi_\gamma \circ \nabla h^*(\nabla h({\bf y})) \\
        &= \left \langle\mathcal{B}_\gamma({\bf y}) , \nabla h({\bf y}) \right \rangle - h(\mathcal{B}_\gamma({\bf y})) -\psi_\gamma({\bf y}),
    \end{align*}
    which leads us to
    \[
        \phi_\gamma({\bf x})= \left \langle {\bf x} , \nabla h({\bf y}) \right \rangle - h({\bf x}) -\psi_\gamma({\bf y}).
   \]  
    On the other hand, $\mathcal{B}_\gamma$ in Equation~\eqref{eq:breg_denoiser_def} can be simplified as 
    \begin{align*}
        \mathcal{B}_\gamma({\bf y}) & =\nabla (\psi_\gamma \circ \nabla h^*) \circ \nabla h({\bf y}) \\
        &= \nabla^2 h^* (\nabla h({\bf y})) \cdot \nabla \psi_\gamma \circ \nabla h^* \circ \nabla h({\bf y}) \\
        &= \nabla^2 h^*(\nabla h({\bf y})) \cdot \nabla \psi_\gamma({\bf y}).
    \end{align*}
    Since $h$ is assumed to be strictly convex from Assumption~\ref{asm:assumption_on_h} on ${\rm int} {\rm dom} (h)$, Hessian of $h$ (denoted as $\nabla^2 h({\bf y})$) is invertible. Since $h$ is of Legendre type from Assumption~\ref{asm:essentially_smooth}, we have $\nabla h^*(\nabla h({\bf y})) = {\bf y}$. By taking the first derivative of this relationship, we have $\nabla^2 h^*(\nabla h({\bf y})) = (\nabla^2 h({\bf y}))^{-1}
    $. This leads to 
    \begin{align*}
        \mathcal{B}_\gamma ({\bf y}) &= (\nabla^2 h({\bf y}))^{-1} \cdot \nabla \psi_\gamma ({\bf y}) = {\bf y} - \nabla g_\gamma ({\bf y}),
    \end{align*}
    where the last equality comes from \eqref{add1}. Therefore we obtain $\mathcal{B}_\gamma=\mathcal{D}_\gamma$, and thus $\mathcal{D}_\gamma  
    \in {\rm prox}_{\phi_\gamma}^h $ from \eqref{add3}. This completes the proof.
    \end{proof}

\subsection{The proposed PnP-iBPDCA}
With the discussions in Subsection \ref{subsec4-1}, we turn our attention to specifying $\lambda g: = \phi_\gamma$ in Problem~\eqref{eq:DCA_model}, which is associating the weakly-convex term $g$ with a data-driven implicit objective function.
Specifically, the DC optimization problem with Bregman-based deep prior can be formulated as 
\begin{equation} \label{eq:DCA_with_deep_prior}
   \min_{{\bf x} \in \mathbb{R}^n}~\left\{\Psi_\lambda({\bf x}) := f_1({\bf x}) -f_2({\bf x}) +\frac{1}{\lambda} \phi_\gamma({\bf x})\right\},
\end{equation}
where $\phi_\gamma({\bf x})$ is defined in \eqref{add2}. 
We are now outlining the plug-and-play inertial Bregman proximal DC algorithm for addressing Problem \eqref{eq:DCA_with_deep_prior} in Algorithm~\ref{alg:PnP-iBPDCA} (PnP-iBPDCA). The key distinction between Algorithms~\ref{alg:iBPDCA} and \ref{alg:PnP-iBPDCA} is the scheme of updating ${\bf x}^{k+1}$, where Algorithm~\ref{alg:PnP-iBPDCA} utilizes an off-the-shelve denoiser instead of evaluating the Bregman proximal operator of a specific function $g$. 

\begin{algorithm}[!h]
   \caption{Plug-and-play Inertial Bregman Proximal DC Algorithm (PnP-iBPDCA)}
  \label{alg:PnP-iBPDCA}  
  \begin{algorithmic}[1]
    \Require{Choose a $\kappa$-strongly convex kernel function $h$ in accordance to Assumption~\ref{asm:assumption_on_h} such that $(f_1,h)$ is $L$-smad. Choose $\delta,\epsilon$ with $1> \delta \geq\epsilon>0$ and ${\rm tol}>0$.}
    \State{\textbf{Initialization:} $ {\bf x}^0={\bf x}^{-1} \in {\rm intdom}(h)\text{, and } \frac{1}{\lambda} >  \max \left \{\delta+\frac{\eta}{\kappa} , L \right\}.$} 
     \For{$k = 0, 1, 2,\ldots,$}
        \State{Compute}
        \begin{equation} 
            {\bf y}^{k} = {\bf x}^{k} + \beta_k \left({\bf x}^k - {\bf x}^{k-1}\right),
        \end{equation}
        
        \State{where $\beta_k \in [0,1)$ is chosen such that
        \begin{equation} \label{eq:bregman_distance_condition_on_extrapolation_Rician}
            \lambda\left(\delta-\epsilon\right)D_h({\bf x}^{k-1},{\bf x}^k) \geq D_h({\bf x}^k,{\bf y}^k).
        \end{equation}} 
        
        \If {${\bf y}^k \notin
         {\rm int}{\rm dom}(h)$,}
            \State{${\bf y}^k = {\bf x}^k$}.
        \EndIf
        
        \State{Take $\xi^k \in \partial f_2({\bf x}^k)$, and compute}
        \begin{equation}
            {\bf x}^{k+1} = \mathcal{D}_\gamma \circ \nabla h^* (\nabla h({\bf y}^k) - \lambda (\nabla f_1({\bf y}^k) -\xi^k)).
        \end{equation}
        \If{$\frac{\|{\bf x}_{k+1} - {\bf x}_{k}\|}{\|{\bf x}^k\|} < {\rm tol}$}
            break
        \EndIf
    \EndFor
  \end{algorithmic}
\end{algorithm}

The following is dedicated to establishing the convergence of PnP-iBPDCA (Algorithm \ref{alg:PnP-iBPDCA}). As mentioned in Theorem~\ref{thm:global_convergence} (Global convergence of iBPDCA), we do require objective function $\Psi_\lambda$ to be a KL function to guarantee global convergence of PnP-iBPDCA. Our main focus lies on validating the KL property of the newly proposed deep prior $g = \frac{1}{\lambda}\phi_\gamma$.

\begin{lemma} {\rm{({\bf Validation of KL property of $\phi_\gamma$})} }\label{prop:validation_real_analytic}
     Let $\mathcal{D}_\gamma $ be a GS denoiser defined in \eqref{gaussian_denoiser}. Then $\phi_\gamma$ which is associated with $\mathcal{D}_\gamma$ 
     given in \eqref{add2} 
    is subanalytic, and thus is a KL function.
\end{lemma}
    \begin{proof}
         Since $\mathcal{D}_\gamma = I_d - \nabla g_\gamma$ from \eqref{gaussian_denoiser} and $g_\gamma = \frac{1}{2} \Vert {\bf x} - \mathcal{N}_\gamma ({\bf x}) \Vert^2$ where $\mathcal{N}_\gamma$ is parameterized with a U-Net with softplus activation function, which is real analytic. From \cite{krantz2002primer}, we know that the sum and composition of real analytic functions are real analytic, leading us to the fact that $N_\gamma$ and $g_\gamma$ are real analytic functions.
    
        For ${\bf y} \in  {\rm int}{\rm dom}(h)$, the Jacobian of $\mathcal{D}_\gamma$ writes as
        \begin{align*}
            J_{\mathcal{D}_\gamma({\bf y})} &= \nabla(\nabla (\psi_\gamma \circ \nabla h^*) \circ \nabla h({\bf y})) \\
            &= \nabla^2 h ({\bf y}) \cdot \nabla^2 (\psi_\gamma \circ \nabla h^*)(\nabla h({\bf y})).
        \end{align*}
        By our assumption that $h$ is of Legendre type, mapping $\nabla h^*: {\rm int}{\rm dom}(h^*) \rightarrow  {\rm int}{\rm dom}(h)$ is bijective. 
        From Assumption~\ref{asm:assumption_on_h}, kernel function $h$ is strictly convex, and hence $\nabla^2 h({\bf y})$ is positive definite. For all ${\bf u} \in \mathbb{R}^n$, we have
        \begin{align*}
            \left \langle J_{\mathcal{D}_\gamma}({\bf y}) {\bf u}, {\bf u} \right \rangle &= \left \langle \nabla^2 h ({\bf y}) \cdot \nabla^2 (\psi_\gamma \circ \nabla h^*)(\nabla h({\bf y})) {\bf u}, {\bf u} \right \rangle \\
            &   > \left \langle \nabla^2 (\psi_\gamma \circ \nabla h^*)(\nabla h({\bf y})) {\bf u}, {\bf u} \right \rangle > 0,
        \end{align*}
        where the last line holds because $\psi_\gamma \circ \nabla h^*$ is strictly convex.
        Therefore $J_{\mathcal{D}_\gamma}$ is positive definite on $ {\rm int}{\rm dom}(h)$. By real analytic inverse function theorem \cite[Theorem 1.5.3]{krantz2002primer}, $\mathcal{D}_\gamma ^{-1}$ is real analytic on ${\rm Im}(\mathcal{D}_\gamma)$. 
       Furthermore, it follows from \cite[Proposition~2.2.3]{krantz2002primer} that the partial derivative of all orders, and the indefinite integral of the real analytic function are real analytic.
        Since $g_\gamma$ and $h$ are real analytic, $\nabla \psi_\gamma$ and hence $\psi_\gamma$ are also real analytic. Since the sum and composition of real analytic functions are real analytic \citep{krantz2002primer}, we finally have $\phi_\gamma$ is real analytic (also subanalytic), and thus is a KL function on its domain. This completes our proof.
\end{proof}    
In line with Assumption~\ref{asm:assumption_on_h}, we require $\phi_\gamma$ to be weakly convex to ensure the convergence of PnP-iBPDCA. Before proceeding, let us. consider an important implication derived from \cite{gribonval2020characterization}. 

\begin{lemma}\cite[Theorem~3]{gribonval2020characterization} \label{thm:convexity_of_deep_prior}
Suppose $\mathcal{Y} \subset \mathbb{R}^n$ be a non-empty set. Let $a:\mathcal{Y}\rightarrow \mathbb{R} \cup  \{ +\infty\}$, $b:\mathbb{R}^n\rightarrow \mathbb{R} \cup \{ + \infty\}$, and $A: \mathcal{Y} \rightarrow \mathbb{R}^n$ be arbitrary function, and $\varphi:\mathcal{Y}\rightarrow \mathbb{R}^n$. 
Suppose the following conditions hold: (i) there exists $\phi:\mathbb{R}^n \rightarrow \mathbb{R} \cup \{ +\infty \}$ such that $\varphi({\bf y}) \in \argmin_{{\bf x}\in \mathbb{R}^n} \{D({\bf x},{\bf y}) + \phi({\bf x})\}$ for each ${\bf y} \in \mathcal{Y}$, with $D({\bf x},{\bf y}) =a({\bf y}) -\left\langle {\bf x}, A({\bf y}) \right \rangle +b({\bf x})$; (ii) there exists a convex lower semi-continuous function $q:\mathbb{R}^n \rightarrow \mathbb{R} \cup \{ + \infty\}$ s.t. $ A(\varphi^{-1}({\bf x})) \subset \partial q({\bf x}), ~\forall \bf{x} \in {\rm Im}(\varphi)$, and let $C' \subset {\rm Im}(\varphi)$ be polygonally connected. Then, there exists a constant $K \in \mathbb{R}$ such that 
    \[
        q({\bf x}) = b({\bf x}) + \phi({\bf x}) + K, \quad \forall {\bf x} \in C'.
    \]
\end{lemma}
The above lemma demonstrates a valuable property that allows us to establish the weak convexity of \(\phi_\gamma\) for a generic Legendre kernel \(h\) in the following lemma.

\begin{lemma}{\rm{({\bf Validation of weak convexity of $\phi_\gamma$})} }\label{lemma:weakly_convex_phi}
Let $\mathcal{D}_\gamma $ be a GS denoiser defined in \eqref{gaussian_denoiser}. Then $\phi_\gamma$ which is associated with $\mathcal{D}_\gamma$ given in \eqref{add2}  is $\frac{\kappa L_{g_\gamma}}{1+L_{g_\gamma}}$-weakly convex on Im($\mathcal{D}_\gamma$), where $L_{g_\gamma}<1$ is the Lipschitz modulus of $\nabla g_\gamma$ in \eqref{gaussian_denoiser}.
\end{lemma}
\begin{proof}
    From Proposition~\ref{prop:bregman_denoiser_defined}, we noted that $\mathcal{D}_\gamma : {\rm int} {\rm dom}(h) \rightarrow \mathbb{R}$ and for each ${\bf y} \in {\rm int}{\rm dom}(h)$ $\mathcal{D}_\gamma ({\bf y}) \in \argmin_{{\bf x} \in \mathbb{R}^n}\{ D_h({\bf x}, {\bf y}) +\phi_\gamma ({\bf x})\}$. Setting
   $a({\bf y}) := \left \langle \nabla h({\bf y}), {\bf y}\right \rangle -h({\bf y}), $ 
    $A({\bf y}) := \nabla h({\bf y})$, and
\begin{align*} 
    b({\bf x}) &:= \begin{cases}
        h({\bf x}), \qquad\text{ if } {\bf x} \in {\rm dom}(h),\\
        +\infty, \qquad\text{ otherwise}.
    \end{cases}
\end{align*}
Let $q $ be a function such that $\nabla q({\bf x}) =\nabla h (\mathcal{D}_\gamma^{-1}({\bf x}))$.
 Since $\nabla h (\mathcal{D}_\gamma^{-1}({\bf x})) = \nabla h( \nabla \theta_\gamma({\bf x}) + {\bf x})$, where $\theta_\gamma$ is defined in \eqref{theta_gamma}, $\mathcal{D}_\gamma = {\rm prox}_{\theta_\gamma}({\bf x})$ and $\theta_\gamma$ is $\frac{L_{g_\gamma}}{1+L_{g_\gamma}}$-weakly convex. 
 This leads us to $\nabla^2 q({\bf x}) = \nabla^2 h( \nabla \theta_\gamma({\bf x}) + {\bf x}) (\nabla^2 \theta_\gamma({\bf x}) + I_d) \succeq \kappa \left (1 - \frac{L_{g_\gamma}}{1+L_{g_\gamma}} \right) I_d \succeq 0$. This verifies that $g$ is convex and lower semi-continuous.
Besides, as mentioned in \cite{hurault2022proximal}, ${\rm Im}(\mathcal{D}_\gamma)$ is polygonally connected. Applying Lemma~\ref{thm:convexity_of_deep_prior}, we obtain
\[
    q({\bf x}) = h({\bf x}) + \phi_\gamma({\bf x}) + K, \quad \forall {\bf x} \in {\rm Im}(\mathcal{D}_\gamma),
\]
where $\phi_\gamma$ is defined in \eqref{add2}.
Further, we have 
\begin{align*}
    \nabla^2 q({\bf x}) = \nabla^2 h({\bf x}) + \nabla^2 \phi_\gamma({\bf x}) \succeq \kappa \left (1 - \frac{L_{g_\gamma}}{1+L_{g_\gamma}} \right) I_d.
\end{align*}
Since $h$ is $\kappa$-strongly convex from Assumption~\ref{asm:assumption_on_h}, this leads to $\nabla^2 h ({\bf x}) - \kappa I_d \succeq 0$ and hence
$[\nabla^2 h({\bf x})  - \kappa I_d ] + [\nabla^2 \phi_\gamma({\bf x}) +  \frac{\kappa L_{g_\gamma}}{1 + L_{g_\gamma}}  I_d] \succeq 0.$
This implies that $ \nabla^2 \phi_\gamma({\bf x}) + \frac{\kappa L_{g_\gamma}}{1+L_{g_\gamma}}  I_d \succeq 0$ and thus $\phi_\gamma$ is $ \frac{\kappa L_{g_\gamma}}{1+L_{g_\gamma}}$-weakly convex on ${\rm Im} (\mathcal{D}_\gamma)$. This completes the proof.
\end{proof}

In the following, we present the convergence results of PnP-iBPDCA (Algorithm~\ref{alg:PnP-iBPDCA}).
\begin{theorem} {\rm{({\bf Convergence of PnP-iBPDCA})}}
    Suppose the kernel function $h$ satisfies Assumptions~\ref{asm:assumption_on_h} and~\ref{asm:essentially_smooth} with $g:=\frac{1}{\lambda}\phi_\gamma$. Let $g_\gamma:\mathbb{R}^n \rightarrow \mathbb{R} \cup \{+ \infty \}$ be proper and differentiable and $\mathcal{D}_\gamma({\bf y}): \operatorname{int} \operatorname{dom}(h)\rightarrow \mathbb{R}^n$ be defined as $\mathcal{D}_\gamma := {\bf y} - \nabla g_\gamma({\bf y})$. Assume $\operatorname{Im}(\mathcal{D}_\gamma )\subset \operatorname{int}\operatorname{dom}(h)$, the sequence $\left\{{\bf x}^k\right\}_{k=0}^\infty$ is generated by Algorithm~\ref{alg:PnP-iBPDCA}. Then the following statements hold:
    \begin{itemize}
    \item[(i)] $\{H_\delta({\bf x}^k,{\bf x}^{k-1})\}_{k=0}^\infty$ is non-increasing and convergent.
    \item[(ii)]  any cluster point ${\bf x}^*$ of sequence $\{{\bf x}^k\}_{k\geq 1}$ is a critical point of~\eqref{eq:DCA_with_deep_prior}, \ie\  it holds that $0\in \partial \Psi_\lambda({\bf x}^*)$.
    \item[(iii)] if $f_1$ and $f_2$ in Problem~\eqref{eq:DCA_with_deep_prior} are both KL functions, then the whole sequence $\{{\bf x}^k\}_{k\geq 1}$ generated by PnP-iBPDCA is convergent. 
    \end{itemize}
\end{theorem}
    \begin{proof}
        It follows from Lemma~\ref{lemma:weakly_convex_phi} that $\phi_\gamma$ is weakly-convex. Thus, Problem~\eqref{eq:DCA_with_deep_prior} can be seen as a special form of Problem~\eqref{eq:DCA_model} with $g= \frac{1}{\lambda}\phi_\gamma$. It follows from Lemma~\ref{lem:descending_property_auxiliary_function} that (i) holds. With Theorem~\ref{thm:subsquential_convergence}, (ii) holds as well. From Lemma~\ref{prop:validation_real_analytic} and $f_1$ and $f_2$ are both KL functions, we know that $\Psi_\lambda$ is a KL function. Finally, the conclusion (iii) can be derived from Theorem~\ref{thm:global_convergence}. This completes the proof.
    \end{proof}

\section{Numerical Experiments} 
\label{sec:experiments}
In this section, we discuss the effectiveness and robustness of our proposed schemes, iBPDCA (Algorithm~\ref{alg:iBPDCA}) and PnP-iBPDCA (Algorithm~\ref{alg:PnP-iBPDCA}). 
As mentioned in the previous section, iBPDCA accommodates generic weakly convex priors, while PnP-iBPDCA is a special instance of iBPDCA. Instead of handcrafted priors, PnP-iBPDCA features a data-driven deep prior that aligns with the theoretical analysis from Section~\ref{sec:convergence_iBPDCA} and provides theoretical convergence. 
It is well known that learned priors typically yield better results than handcrafted ones \citep{chen2022learning,wu2024extrapolated}. Therefore, we solely conduct experiments with PnP-iBPDCA. 
To validate the effectiveness of the proposed scheme, we undertake two imaging experiments: Rician noise removal on magnetic resonance images and phase retrieval. We compare our results with state-of-the-art PnP-based and non-PnP-based methods. Notably, Rician noise, which invariably affects MR images during the collection of measurements in k-space, can hinder tissue identification and clinical diagnosis. Conversely, phase retrieval focuses on recovering the signal phase from the measured amplitude.
All the experiments are conducted on an Nvidia Quadro RTX8000 GPU. The source code is available at \url{https://github.com/nicholechow/PnP-iBPDCA}.

\subsection{Application to Rician Noise Removal}
In this subsection, we conduct the Rician noise removal on MR images by the proposed framework. First, we introduce the model and experimental setting of our PnP-iBPDCA algorithm for solving the Rician noise removal problem. Then, we specify the model parameters and analyze the influence of the inertial parameter. Multiple efforts have been made to address Rician noise removal Problem~\eqref{eq:map_Rician_model} through DCA with variational prior such as \cite{getreuer2011variational, chen2015convex, kang2015nonconvex,chen2019variational,wu2022efficient}. Recently, \cite{wei2023nonconvex} tackled the problem with PnP-ADMM directly without addressing the fundamental DC structure of the problem. Naturally, we are interested in how our proposed method compares. As stated before, we split the objective function as Equation~\eqref{eq:Rician_obj_split}. More specifically, from Theorem 2.1 in \cite{baricz2007inequalities}, $I_0$ in Equation~\eqref{eq:Rician_obj_split} 
%
is strictly log convex. Hence $\log(I_0)$ is a proper and convex function, satisfying Assumption~\ref{asm:assumption_on_h}(iii). Then, take $h=\frac{1}{2}\Vert \cdot \Vert^2$, which is of Legendre type, and apply the proposed iBPDCA (Algorithm \ref{alg:iBPDCA}), we obtain
\begin{equation}\label{ibpdca_Rician}
\begin{aligned}
    {\bf x}^{k+1} &= \argmin_{x}{\mu \phi({\bf x})+\left\langle \frac{1}{\sigma^2}  \mathcal{A}^{\top} \mathcal{A} {\bf y}^k - \xi^k, {\bf x}-{\bf y}^k\right\rangle + \frac{1}{2 \lambda} \left\|{\bf x}-{\bf y}^k\right\|^2 } \\
    &= \argmin_{x} \mu \phi({\bf x}) + \frac{1}{2\lambda}\left\Vert {\bf x} -\left({\bf y}^k - \frac{\lambda}{\sigma^2}  \mathcal{A}^{\top} \mathcal{A} {\bf y}^k + \lambda \xi^k \right) \right\Vert^2  \\
    &= {\rm prox}_{\lambda \mu \phi}  \left({\bf y}^k - \frac{\lambda}{\sigma^2}  \mathcal{A}^{\top} \mathcal{A} {\bf y}^k + \lambda \xi^k \right),
\end{aligned}
\end{equation}
where
\begin{equation} \nonumber
        \xi^k =\mathcal{A}^{\top} \frac{f}{\sigma^2}\frac{I_1\left(\frac{f\mathcal{A}{\bf x}^k}{\sigma^2} \right)}{I_0\left(\frac{f\mathcal{A}{\bf x}^k}{\sigma^2}\right)},
\end{equation}
with $I_1$ being the modified Bessel function of the first kind with order one \citep{gray1895treatise}.

Now we specify the application of PnP-iBPDCA (Algorithm~\ref{alg:PnP-iBPDCA}) to solve the Rician noise removal problem. 
As stated in Proposition~\ref{prop:bregman_denoiser_defined}, we have to ensure $\psi_\gamma \circ \nabla h^*$ is convex on $ {\rm int}{\rm dom}\left(h^*\right)$ with $\nabla \psi_\gamma ({\bf y}) = {\bf y} -\nabla g_\gamma({\bf y})$.
By the setting of $h =\frac{1}{2} \Vert \cdot \Vert^2$, we recover $\psi_\gamma({\bf y})=\frac{1}{2} \Vert {\bf y} \Vert^2 -g_\gamma({\bf y})$. Then, it follows from Proposition 3.1(a) in \cite{hurault2022proximal} that $\psi_\gamma \circ \nabla h^*({\bf x}) =\frac{1}{2} \Vert x \Vert^2 -g_\gamma ({\bf x})$ is convex, satisfying assumption in Proposition~\ref{prop:bregman_denoiser_defined}. Hence, there exists a functional $\phi_\gamma$ such that $\mathcal{D}_\gamma={\rm prox}_{\phi_\gamma}^h$ with $h =\frac{1}{2} \Vert \cdot \Vert^2$, where $\mathcal{D}_\gamma$ is the GS denoiser in \eqref{gaussian_denoiser} and $\gamma$ is the noise level. Consequently, the PnP-iBPDCA is applicable for the Rician noise removal task and its iterative scheme from \eqref{ibpdca_Rician} can be read as 
\begin{equation} \label{eq:PnP-iBPDCA_Rician_denoiser_model}
    {\bf x}^{k+1}  = \mathcal{D}_\gamma \left ( {\bf y}^k -\frac{\lambda}{\sigma^2} \mathcal{A}^{\top} \mathcal{A} {\bf y}^k + \lambda \xi^k\right).
\end{equation}

Experiments are conducted on the simulated brain database BrainWeb\footnote{\url{https://brainweb.bic.mni.mcgill.ca/}}, which consists of three MR sequences: T1-weighted, T2-weighted, and proton-density-weighted (PD-weighted). 
For a thorough comparison, we incorporate all three views of MR images, \ie\ axial, sagittal, and coronal along with the three MR sequences with and without multiple sclerosis lesions. The imaging parameters were set as: $\mbox{RF}=0$, slice $\mbox{thickness}=1 mm$, phantom = normal. Since the BrainWeb dataset is under ongoing development and generates volumes with spontaneous variations, we store the generated slices for our experiment at \url{https://github.com/nicholechow/PnP-iBPDCA/tree/main/testsets/brainweb}. 
In our experiment, we focus on the scenario $\mathcal{A}=I$ in \eqref{eq:map_Rician_model} where images are corrupted with Rician noise, to showcase the effectiveness of the proposed technique. It should be noted that our algorithm can handle a vast variety of forward operators $\mathcal{A}$.

\subsubsection{Analysis of Parameters}
Before showcasing the results of the Rician noise removal, we first analyze the setting of algorithm parameters $L$, $\eta$, $\kappa$, and model parameters $\lambda$, $\gamma$.

As mentioned in Sections \ref{sec:alg} and \ref{sec:pnp_alg}, to guarantee the convergence of PnP-iBPDCA (Algorithm~\ref{alg:PnP-iBPDCA}), $(f_1,h)$ should be $L$-smad. Since $f_1({\bf x}):=\frac{1}{2\sigma^2}\Vert \mathcal{A}{\bf x} \Vert^2$, as stated in \eqref{eq:Rician_obj_split}, and $h=\frac{1}{2}\Vert \cdot \Vert^2$, we solely need to ensure the convexity of $Lh-f_1$ (\ie\  the NoLips condition) which is finding a $L >0$ such that 
\[
L \lambda_{\min} (\nabla^2 h({\bf x})) \geq  \lambda_{\max} (\nabla^2 f_1({\bf x})).
\]
Since $\lambda_{\min} (\nabla^2 h({\bf x})) = \lambda_{\min} (I_d) = 1 $ and $\lambda_{\max} (\nabla^2 f_1({\bf x})) = \lambda_{\max} \left(\frac{1}{\sigma^2}\mathcal{A}^{\top} \mathcal{A} \right) = \lambda_{\max} \left(\frac{1}{\sigma^2}I_d \right)=\frac{1}{\sigma^2}$ by the setting, we can then choose $L$ satisfying $L\geq \frac{1}{\sigma^2}$.
On the other hand, 
since $g:=\frac{1}{\lambda}\phi_\gamma$ is $\frac{L_{g_\gamma}}{\lambda(L_{g_\gamma}+1)}$-weakly convex, where $L_{g_\gamma}<1$ is the Lipschitz modulus of $\nabla g_\gamma$ in \eqref{gaussian_denoiser}, we can set $\eta=\frac{1}{2\lambda}$. 
Besides, since $h$ is 1-strongly convex, we can set $\kappa =1$.

Now, we can obtain a bound on $\lambda$ as
$
    \frac{1}{\lambda} >  \max \left\{\delta+ \frac{1}{2\lambda} , \frac{1}{\sigma^2} \right\}.
$
Then, we can choose $\lambda$ such that 
$\lambda < \min\left\{\frac{1}{2\delta},\sigma^2\right\}$. From Equations \eqref{ibpdca_Rician} and \eqref{eq:PnP-iBPDCA_Rician_denoiser_model} we see that denoiser strength $\gamma=\sqrt{\lambda\mu}$ is correlated to $\lambda$ and $\mu$ while $\lambda$ is correlated to noise level $\sigma$, we then can estimate $\gamma =\sqrt{\lambda\mu}$ and $\lambda = \min\left\{\frac{1}{2\delta},\sigma^2\right\} \lambda_c$, where $\lambda_c <1$.  
Following similar parameters selection procedure from \cite{wu2022efficient}, we choose $\lambda_c = \{0.0385, 0.102, 0.1462, 0.7312 \}$, and $\mu = \{1.9,1.6,1.3,1.3\}$ for Rician noise levels $\sigma = \{2.55,7.65,12.75,25.5\}$ respectively.

\subsubsection{Influence of Inertial}
After fixing the algorithm and model parameters in our proposed PnP-iBPDCA, the only remaining variable is the inertial parameter $\beta_k$. We then analyze the choice and influence of the $\beta_k$ in this subsection.

First of all, we analyze the selection of inertial parameter $\beta_k$. Considering our use of Euclidean geometry $\frac{1}{2} \Vert \cdot \Vert^2$, we can simplify the constraint on $\beta_k$ from \eqref{eq:bregman_distance_condition_on_extrapolation_Rician} to $0\leq\beta_k \leq \sqrt{\lambda (\delta-\epsilon)}< 1$, as mentioned in Remark~\ref{remark:range_of_beta_k}. 
By defining
\begin{equation}\label{pi(sigma)}
    \Pi(\lambda) := \sqrt{\lambda (\delta-\epsilon)},
\end{equation}
for each $\sigma = \{2.55, 7.65, 12.75, 25.5\}$, we have $\Pi(\lambda) = \{0.0949, 0.2381, 0.3055, 0.7071\}$.

\begin{figure}[t!] 
    \centering
    \begin{minipage}{0.25\linewidth}
        \centering
        \centerline{\includegraphics[width=2.3in]{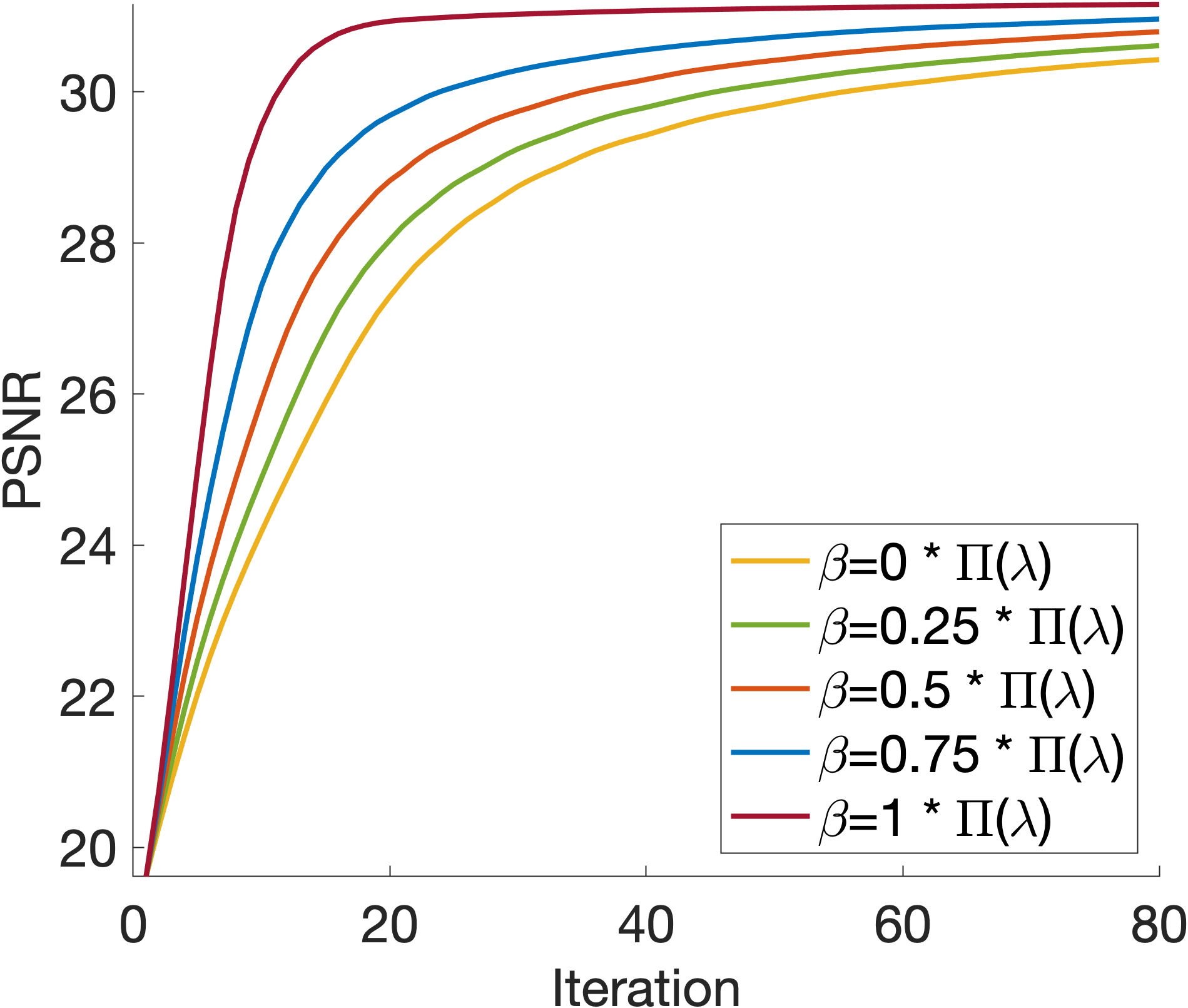}}
        \centerline{ {(a) PSNR }}
    \end{minipage} \hspace{1.3in}
    \begin{minipage}{0.25\linewidth}
        \centering
        \centerline{\includegraphics[width=2.3in]{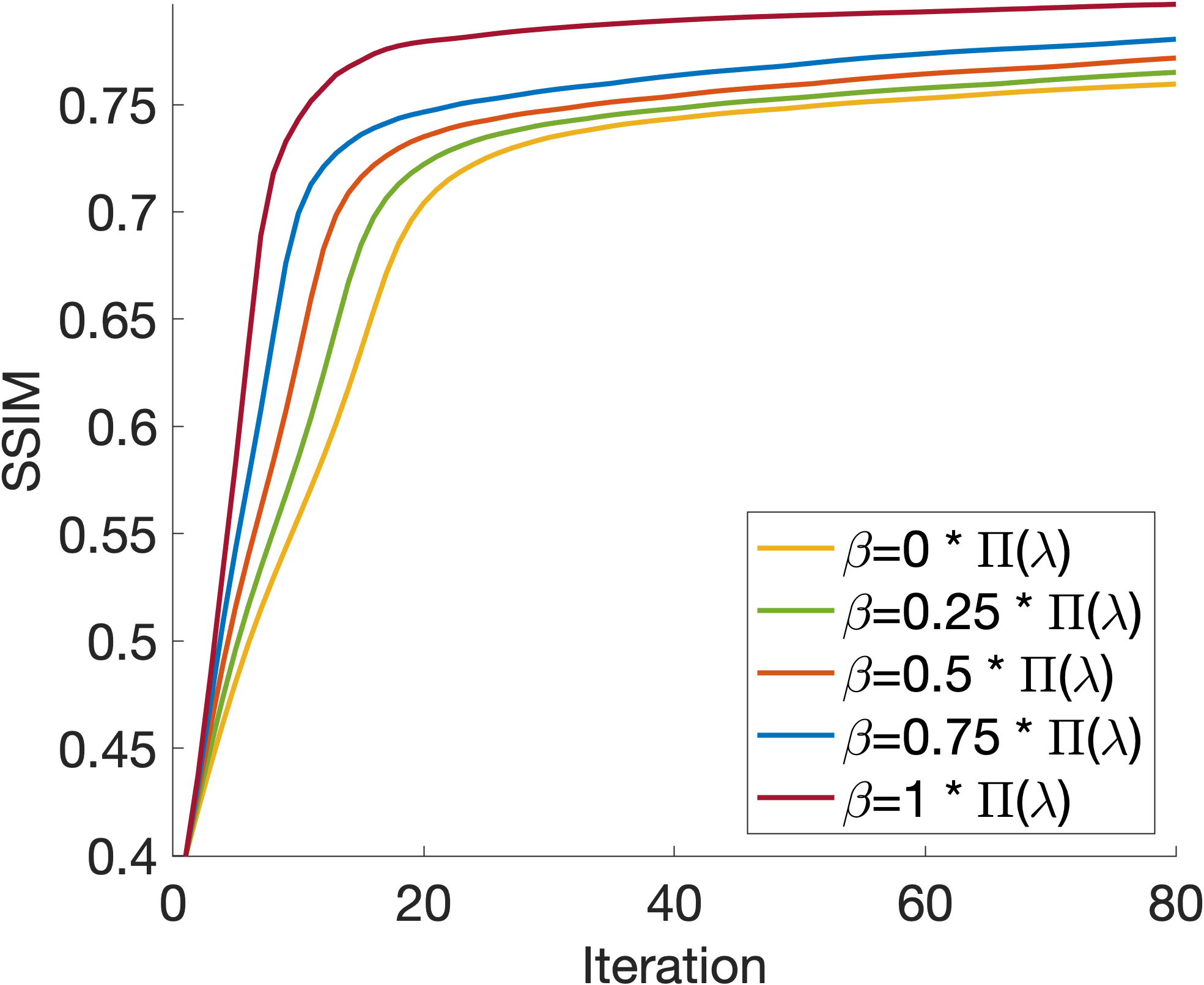}}
        \centerline{{(b) SSIM }}
    \end{minipage}\vspace{0.2in}

    \begin{minipage}{0.25\linewidth}
        \centering
        \centerline{\includegraphics[width=2.3in]{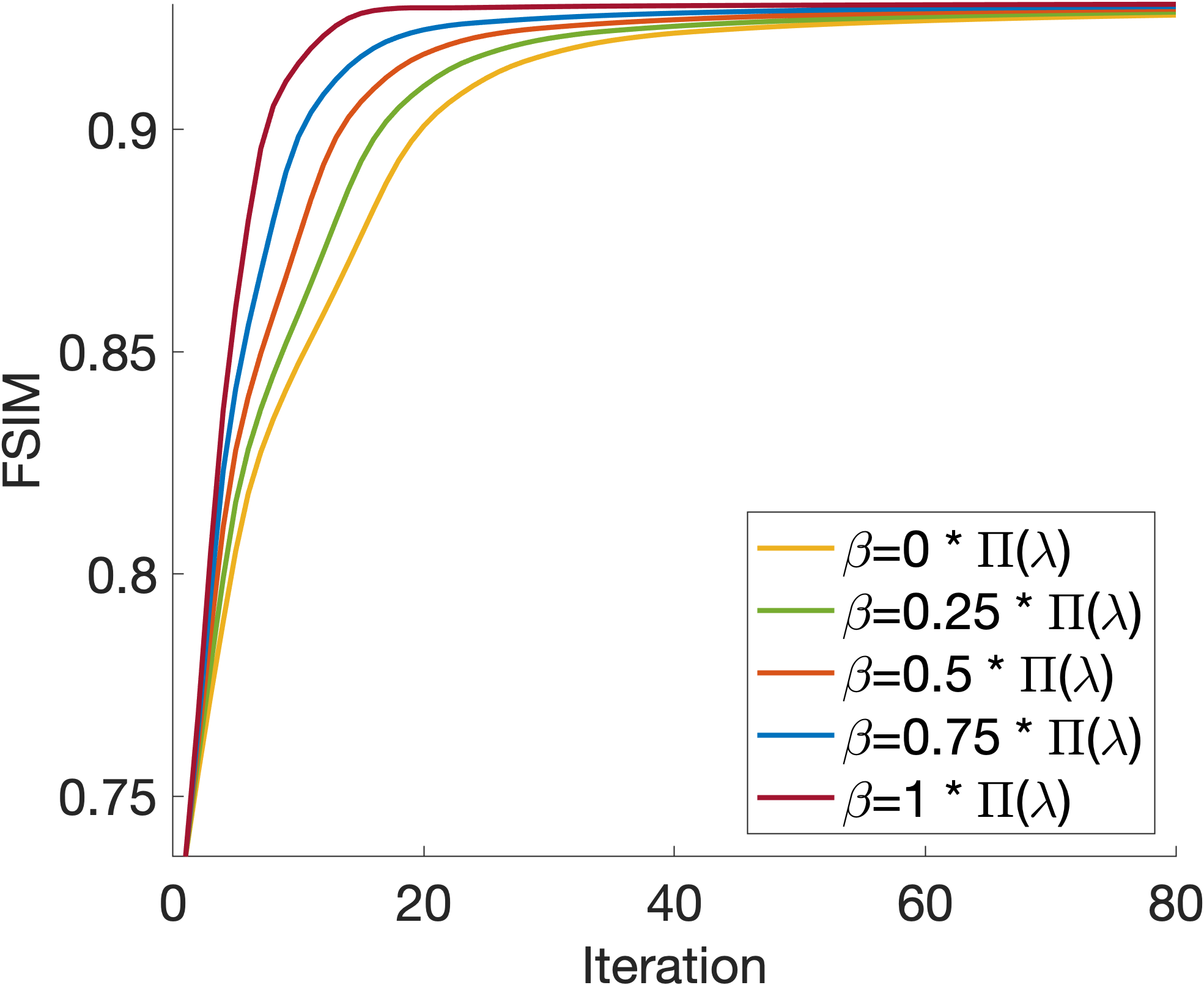}}
        \centerline{{(c) FSIM }}
    \end{minipage}
    \hspace{1.3in}
    \begin{minipage}{0.25\linewidth}
        \centerline{\includegraphics[width=2.3in]{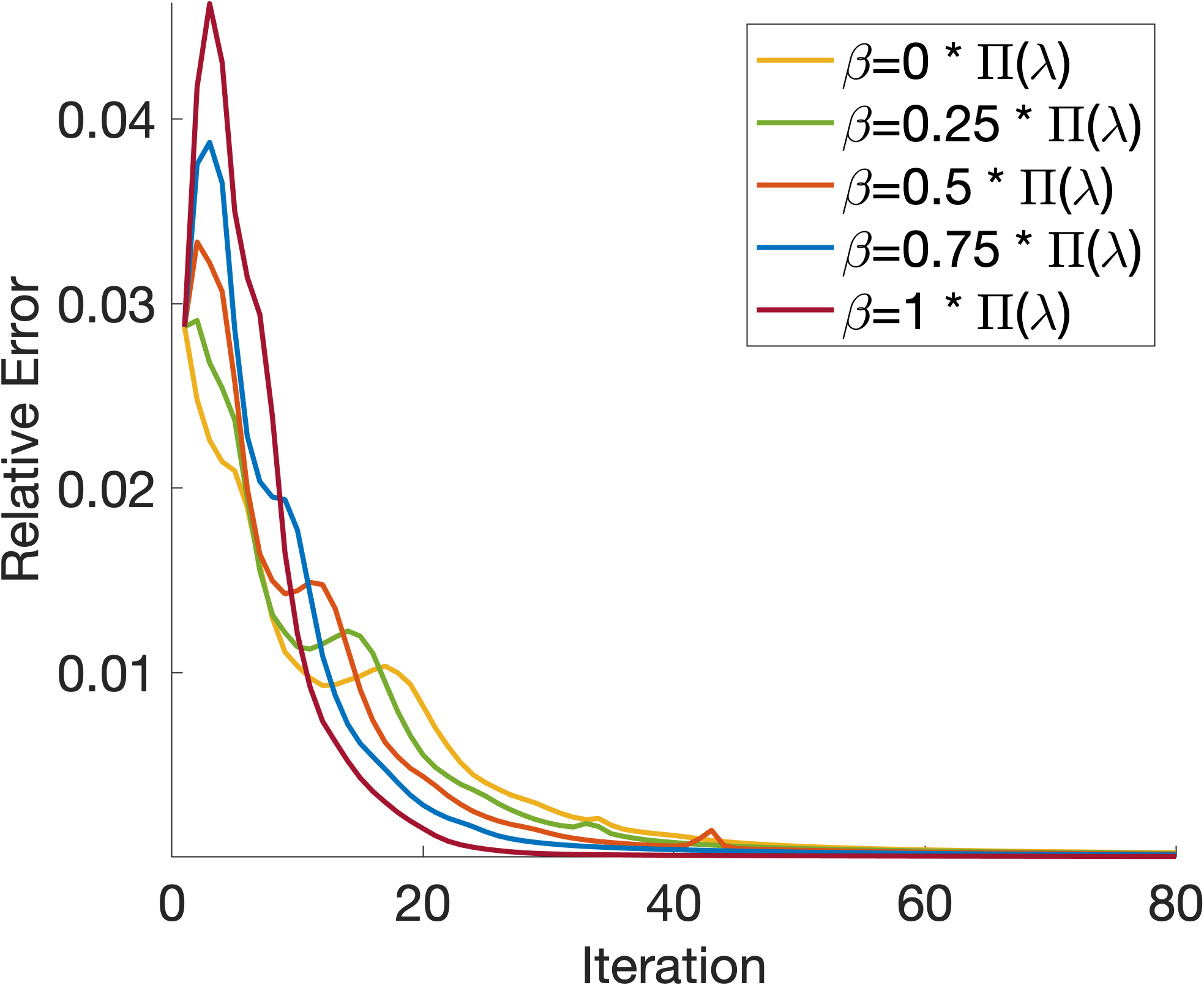}}
        \centerline{{(d) Relative error $ \frac{\Vert {\bf x}_{k+1} - {\bf x}_{k} \Vert }{{\Vert {\bf x}_k\Vert }}$}}
    \end{minipage}
    \caption{Effect of inertial parameter $\beta_k$ with noise level $\sigma = 25.5$. PSNR, SSIM, FSIM, and relative error with iteration are displayed, respectively, to show the influence of the inertial step. Overall, the results of the maximal inertial parameter converge the fastest.   
    }\label{fig:effect_of_inertial}
\end{figure}
\begin{figure}[!t] 
    \centering
    \begin{minipage}{0.3\linewidth}
        \centering
        \includegraphics[width=\linewidth]{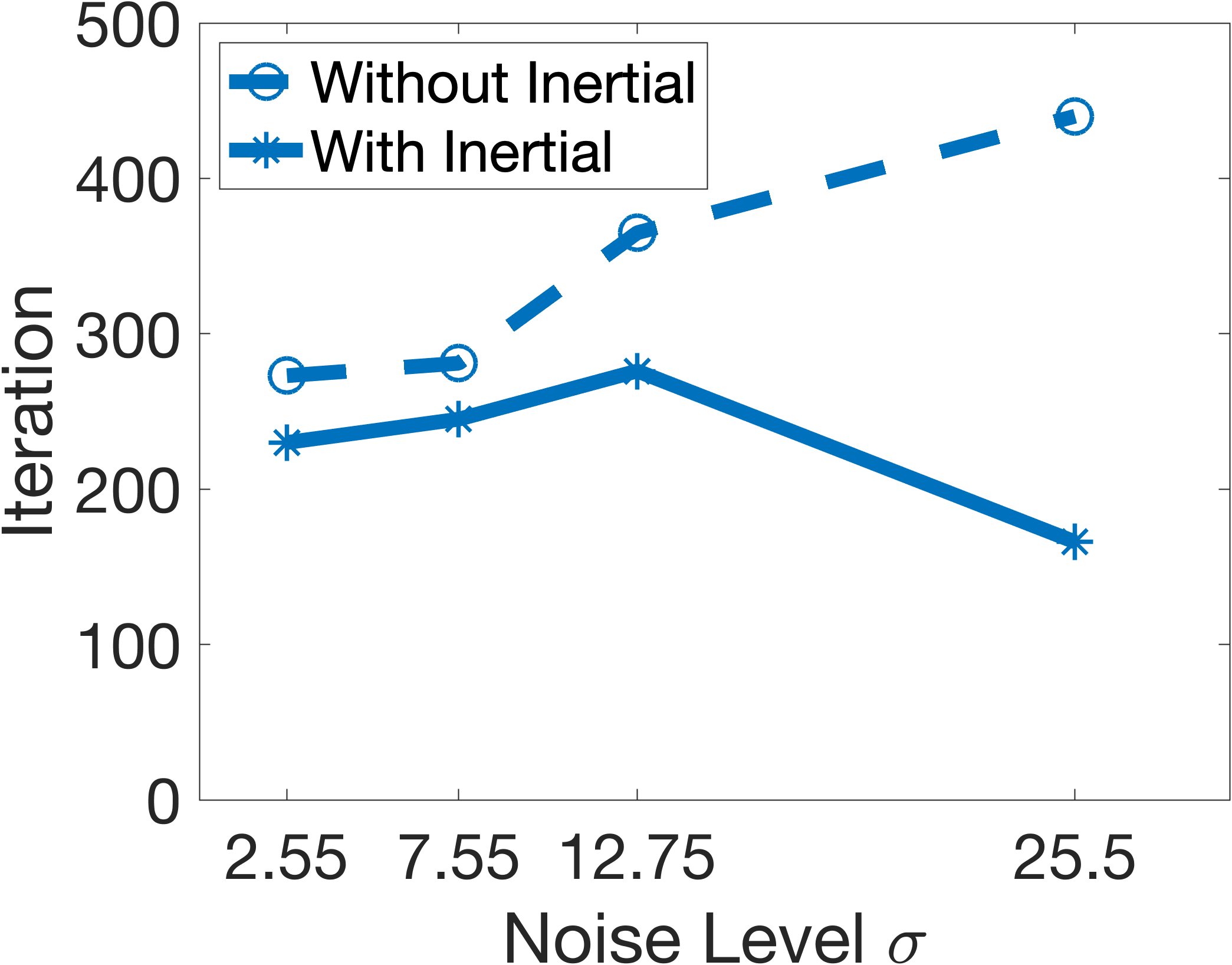}
        \centerline{(a) T1w sequence}
    \end{minipage}
    \begin{minipage}{0.3\linewidth}
        \centering
        \includegraphics[width=\linewidth]{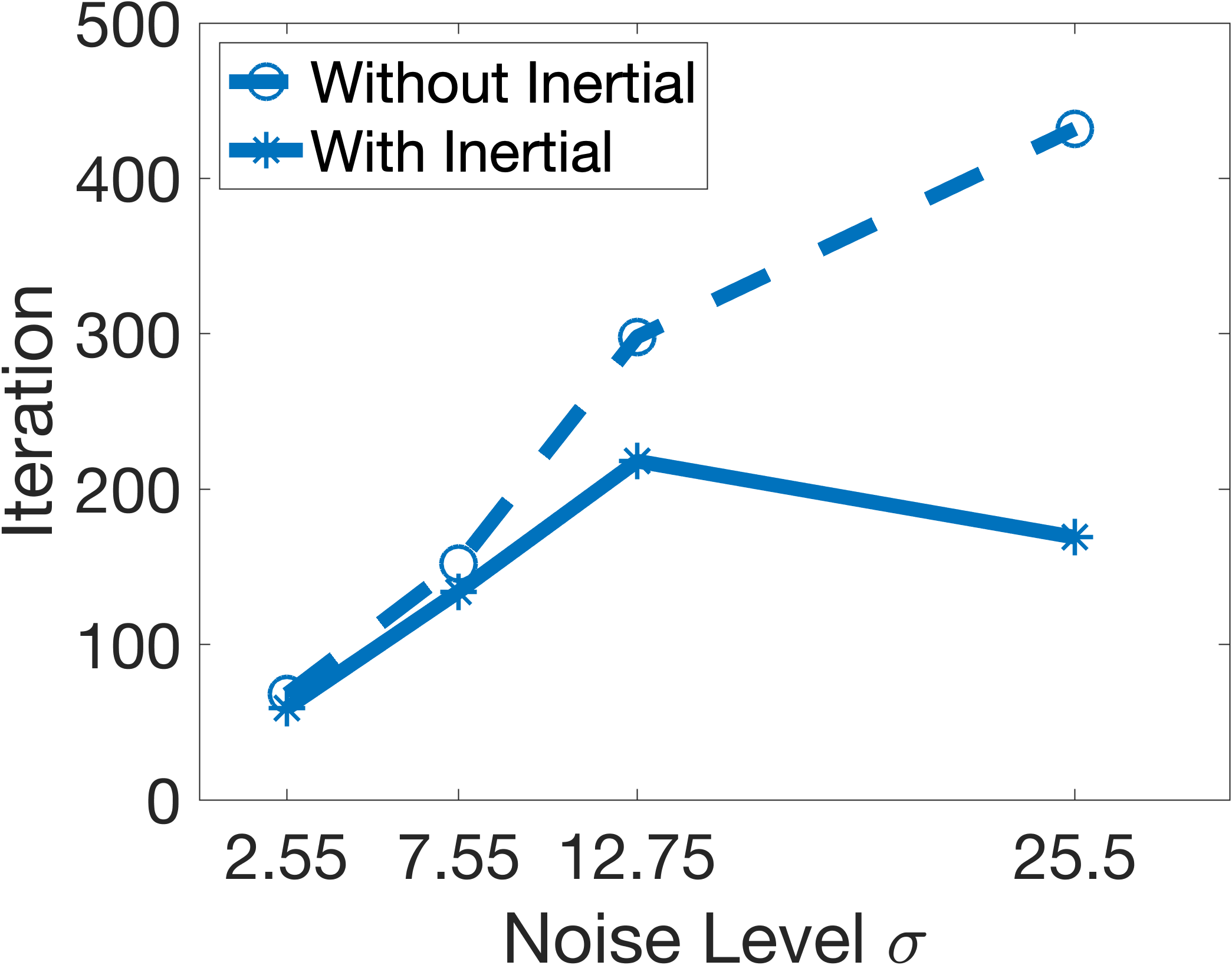}
        \centerline{{(b) T2w sequence}}
    \end{minipage}    
    \begin{minipage}{0.3\linewidth}
        \centering
        \includegraphics[width=\linewidth]{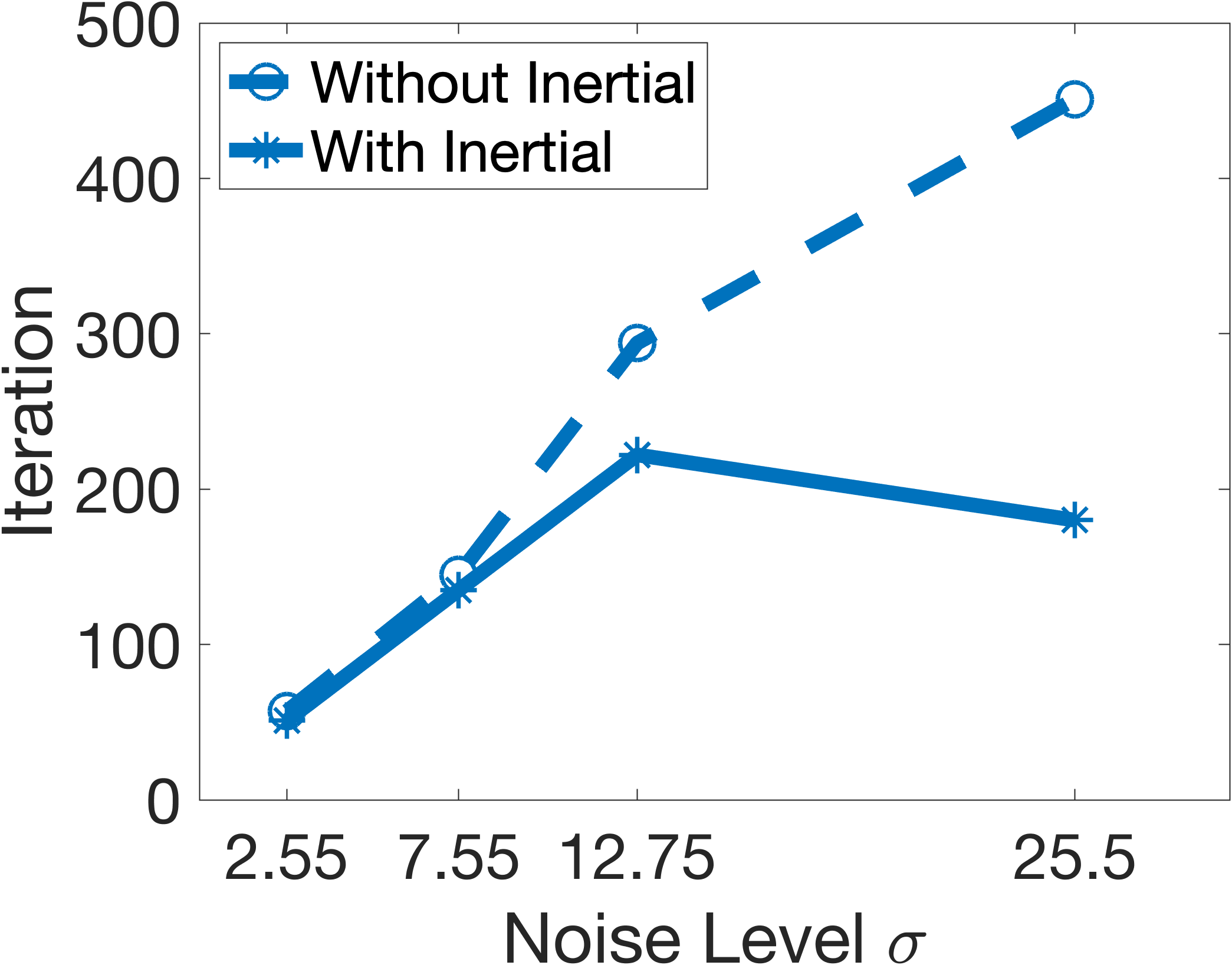}
        \centerline{{(c) PDw sequence}}
    \end{minipage} 
    \caption{
    Average number of convergence iterations on the BrainWeb dataset with and without inertial step. Brain images of three MR sequences T1w, T2w, and PDw with different Rician noise levels $\sigma=\{2.55, 7.55, 12.75, 25.5\}$ are tested. The average convergent iteration with different noise levels is displayed. Overall, the results of with inertial substantially saved the iteration step. 
    }\label{fig:avg_num_iter_brainweb}
\end{figure}

Following the setting of \eqref{pi(sigma)}, we set $\beta_k =\{0, 0.25, 0.50, 0.75, 1\} *\Pi(\lambda)$ to see the influence of inertial on Rician noise removal in Figure~\ref{fig:effect_of_inertial}. The result of PSNR, SSIM, FSIM, and relative error reveals that the acceleration of inertial is almost linear. A larger choice of $\beta_k$ accelerates the convergence rate and reduces computational costs while maintaining the visual quality. Therefore, we choose the maximal $\beta_k$ (\ie\  $\beta_k =\Pi(\lambda)$) in subsequent experiments.

\begin{table}[t!] 

\centering
\resizebox{1.0\hsize}{!}{
    \begin{tabular}{c|c|cccc|cccc|cccc}
    \hline
    \rowcolor[HTML]{EFEFEF} 
        \cellcolor[HTML]{EFEFEF}& Sequences & \multicolumn{4}{c|}{T1w}& \multicolumn{4}{c|}{T2w} & \multicolumn{4}{c}{PDw} \\ 
        \cline{2-14}
        \rowcolor[HTML]{EFEFEF} 
         \multirow{-2}{*}{\cellcolor[HTML]{EFEFEF}\textsc{Methods}}& Noise Level & 2.55 & 7.65 & 12.75 & 25.5 & 2.55 & 7.65 & 12.75 & 25.5 & 2.55 & 7.65 & 12.75 & 25.5  \\ \hline
        \multirow{3}*{Degraded}  & PSNR & 39.20 & 29.60 & 25.13 & 19.09 & 39.59 & 29.89 & 25.38 & 19.33 & 39.59 & 29.91 & 25.44 & 19.57  \\ 
        ~ & SSIM & 88.72 & 72.05 & 59.92 & 38.74 & 93.11 & 78.35 & 68.06 & 51.08 & 92.26 & 73.13 & 59.12 & 39.49  \\ 
        ~ & FSIM & 98.86 & 92.79 & 86.01 & 72.71 & 99.02 & 93.95 & 88.35 & 77.38 & 98.66 & 91.73 & 84.28 & 70.78  \\ \hline
        
        \multirowcell{3}{Prox-GS }  & PSNR & 35.90 & 30.77 & 27.50 & 22.18 & 34.04 & 29.83 & 27.08 & 22.36 & 35.13 & 30.56 & 27.67 & 22.90  \\ 
        ~ & SSIM & 87.71 & 76.76 & 71.81 & 62.67 & 92.05 & 82.63 & 78.08 & 70.86 & 90.68 & 79.48 & 73.59 & 65.20  \\ 
        ~ & FSIM & 96.51 & 93.10 & 90.35 & 84.64 & 96.99 & 94.28 & 92.15 & 88.19 & 95.88 & 91.53 & 88.29 & 83.61  \\ \hline
        
        \multirowcell{3}{WDNN }  & PSNR & 34.66 & 32.63 & 32.56 & 29.40 & 31.52 & 29.64 & 29.29 & 27.03 & 32.40 & 31.00 & 30.51 & 28.13  \\ 
        ~ & SSIM & 88.40 & 84.03 & 82.94 & 75.83 & 93.04 & 88.47 & 87.70 & 81.77 & 92.21 & 87.61 & 85.93 & 78.06  \\ 
        ~ & FSIM & 98.87 & 96.72 & 95.61 & 91.49 & 99.05 & 96.87 & 95.92 & 92.46 & 98.71 & 96.35 & 94.94 & 90.19  \\ \hline

        \multirowcell{3}{Deform2Self } & PSNR & 35.58 & 31.83 & 28.35 & 22.70 & 30.57 & 29.08 & 26.98 & 22.54 & 31.38 & 29.84 & 27.69 & 23.06 \\ 
        ~ & SSIM & 88.65 & 79.65 & 75.91 & 69.15 & 91.53 & 83.51 & 79.77 & 74.20 & 89.32 & 81.35 & 77.46 & 71.39 \\ 
        ~ & FSIM & 96.85 & 95.30 & 93.43 & 89.06 & 95.96 & 94.92 & 93.63 & 90.78 & 93.94 & 92.86 & 91.59 & 88.83 \\ \hline

        \multirowcell{3}{BDCA } & PSNR & 42.07 & 34.37 & 31.33 & 26.90 & \underline{41.27} & 32.44 & 29.77 & 25.21 & \underline{41.90} & 33.97 & 31.14 & 26.80  \\ 
        ~ & SSIM & 76.15 & 67.69 & 64.32 & 53.96 & 81.02 & 73.93 & 68.65 & 59.38 & 79.67 & 69.84 & 62.45 & 52.34  \\ 
        ~ & FSIM & \underline{99.21} & 94.49 & 93.49 & 86.01 & \cellcolor[HTML]{D4FFD3}\tBF{99.33} & 96.61 & 94.78 & 89.19 & \underline{99.11} & 95.42 & 93.52 & 87.17  \\ \hline
        
        \multirowcell{3}{CPnP } & PSNR & \underline{42.30} & \underline{37.04} & \underline{34.52} & \underline{30.61} & 41.25 & \underline{35.03} & \underline{32.28} & \underline{28.57} & 41.71 & \underline{36.15} & \underline{33.47} & \underline{29.87}  \\ 
        ~ & SSIM & \underline{96.21} & \underline{94.25} & \underline{91.52} & \cellcolor[HTML]{D4FFD3}\tBF{84.30} & \underline{96.87} & \underline{94.94} & \underline{93.17} & \cellcolor[HTML]{D4FFD3}\tBF{88.15} & \underline{96.47} & \underline{93.77} & \underline{90.95} & \underline{83.89}  \\ 
        ~ & FSIM & 99.20 & \underline{97.62} & \underline{96.08} & \underline{92.03} & \underline{99.22} & \underline{97.72} & \underline{96.38} & \underline{93.36} & 98.98 & \underline{97.12} & \underline{95.15} & \underline{90.85}  \\ \hline
        
        \multirowcell{3}{Ours \\ (Algorithm~\ref{alg:PnP-iBPDCA})} & PSNR & \cellcolor[HTML]{D4FFD3}\tBF{43.56} & \cellcolor[HTML]{D4FFD3}\tBF{37.52} & \cellcolor[HTML]{D4FFD3}\tBF{34.86} & \cellcolor[HTML]{D4FFD3}\tBF{30.71} & \cellcolor[HTML]{D4FFD3}\tBF{42.00} & \cellcolor[HTML]{D4FFD3}\tBF{35.28} & \cellcolor[HTML]{D4FFD3}\tBF{32.49} & \cellcolor[HTML]{D4FFD3}\tBF{28.72} & \cellcolor[HTML]{D4FFD3}\tBF{42.78} & \cellcolor[HTML]{D4FFD3}\tBF{36.62} & \cellcolor[HTML]{D4FFD3}\tBF{34.02} & \cellcolor[HTML]{D4FFD3}\tBF{30.45}  \\ 
        ~ & SSIM & \cellcolor[HTML]{D4FFD3}\tBF{98.68} & \cellcolor[HTML]{D4FFD3}\tBF{94.97} & \cellcolor[HTML]{D4FFD3}\tBF{93.50} & \underline{82.26} & \cellcolor[HTML]{D4FFD3}\tBF{98.58} & \cellcolor[HTML]{D4FFD3}\tBF{95.61} & \cellcolor[HTML]{D4FFD3}\tBF{93.88} & \underline{87.26} & \cellcolor[HTML]{D4FFD3}\tBF{98.32} & \cellcolor[HTML]{D4FFD3}\tBF{94.76} & \cellcolor[HTML]{D4FFD3}\tBF{92.45} & \cellcolor[HTML]{D4FFD3}\tBF{85.92}  \\ 
        ~ & FSIM &  \cellcolor[HTML]{D4FFD3}\tBF{99.40}& \cellcolor[HTML]{D4FFD3}\tBF{97.84} & \cellcolor[HTML]{D4FFD3}\tBF{96.24} & \cellcolor[HTML]{D4FFD3}\tBF{92.28} & \cellcolor[HTML]{D4FFD3}\tBF{99.33} & \cellcolor[HTML]{D4FFD3}\tBF{97.87} & \cellcolor[HTML]{D4FFD3}\tBF{96.52} & \cellcolor[HTML]{D4FFD3}\tBF{93.41} & \cellcolor[HTML]{D4FFD3}\tBF{99.18} & \cellcolor[HTML]{D4FFD3}\tBF{97.40} & \cellcolor[HTML]{D4FFD3}\tBF{95.79} & \cellcolor[HTML]{D4FFD3}\tBF{92.37} \\ \hline
    \end{tabular}}
\caption{Average Rician noise noise removal results with indexes PSNR (dB), SSIM (\%), and FSIM (\%). Three MR sequences T1w, T2w, and PDw are considered on the Brainweb dataset with different Rician noise levels. The top results are highlighted in \colorbox[HTML]{BBFFBB}{\tBF{green}}, while the second-best results are underlined.
}
\label{tab:brainweb_Rician_denoising}
\end{table}

In Figure~\ref{fig:avg_num_iter_brainweb}, we display the average number of iterations on T1-weighted (T1w), T2-weighted (T2w), and PD-weighted (PDw) sequences with different Rician noise levels. In each sequence, as the Rician noise level $\sigma$ increases, a more pronounced impact of inertial can be observed. This aligns with our expectations, as the noise level $\sigma$ increases, the upper bound on inertial parameters $\Pi(\lambda)$ increases as well, allowing a wider range of extrapolation and highlighting the efficacy of our inertial strategy. On the other hand, we know that even under the setting of a small Rician noise level $\sigma=2.55$, the inertial step increases the practical convergence by almost 100 iteration steps in the T1w sequence. Hence, from both Rician noise levels and MR sequences aspects, there is a significant need for the inertial step. 

\begin{figure}[t!]
\subfigure[\scriptsize{Original}]{
    \zoomincludgraphic{0.235\textwidth}{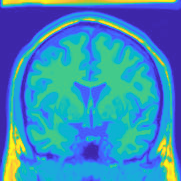}{0.75}{0.25}{0.9}{0.4}{2.7}{help_grid_off}{up_right}{line_connection_off}{3}{red}{1.5}{red}
}\hspace{-0.345in}
  \subfigure[\scriptsize{Observed (25.39/65.67)}]{
	\zoomincludgraphic{0.235\textwidth}{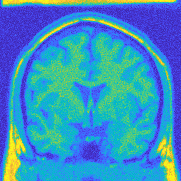}{0.75}{0.25}{0.9}{0.4}{2.7}{help_grid_off}{up_right}{line_connection_off}{3}{red}{1.5}{red}  
	}\hspace{-0.345in}
   \subfigure[\scriptsize{Prox-GS (28.19/79.14)}]{
	\zoomincludgraphic{0.235\textwidth}{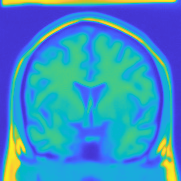}{0.75}{0.25}{0.9}{0.4}{2.7}{help_grid_off}{up_right}{line_connection_off}{3}{red}{1.5}{red} 
	}\hspace{-0.345in}
    \subfigure[\scriptsize{WDNN (32.62/86.98)}]{
	\zoomincludgraphic{0.235\textwidth}{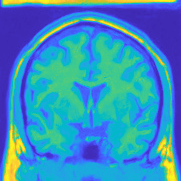}{0.75}{0.25}{0.9}{0.4}{2.7}{help_grid_off}{up_right}{line_connection_off}{3}{red}{1.5}{red}  
	}

  \subfigure[\scriptsize{Deform2Self (29.93/84.74)}]{
	\zoomincludgraphic{0.235\textwidth}{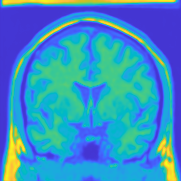}{0.75}{0.25}{0.9}{0.4}{2.7}{help_grid_off}{up_right}{line_connection_off}{3}{red}{1.5}{red} 
	}\hspace{-0.345in}
 \subfigure[\scriptsize{BDCA (30.71/84.36)}]{
	\zoomincludgraphic{0.235\textwidth}{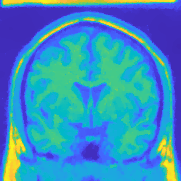}{0.75}{0.25}{0.9}{0.4}{2.7}{help_grid_off}{up_right}{line_connection_off}{3}{red}{1.5}{red}  
	}\hspace{-0.345in}
 \subfigure[\scriptsize{CPnP (33.75/91.89)}]{
	\zoomincludgraphic{0.235\textwidth}{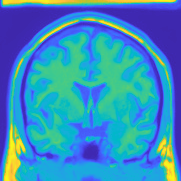}{0.75}{0.25}{0.9}{0.4}{2.7}{help_grid_off}{up_right}{line_connection_off}{3}{red}{1.5}{red}  
	}\hspace{-0.345in}
 \subfigure[\scriptsize{Ours (34.28/94.55)}]{
	\zoomincludgraphic{0.235\textwidth}{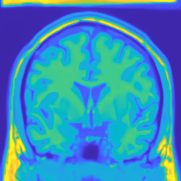}{0.75}{0.25}{0.9}{0.4}{2.7}{help_grid_off}{up_right}{line_connection_o }{3}{red}{1.5}{red} 
	}
 \caption{Rician noise removal results (PSNR(dB)/SSIM($\%$)) with noise level 12.75 on the Coronal viewpoint of an MR image from the T1w sequences of the BrainWeb dataset. Visualization comparison of our scheme and some state-of-the-art Rician noise removal methods: (c) Prox-GS \citep{hurault2022proximal}, (d) BDCA \citep{wu2022efficient} (e) WDNN \citep{you2019denoising} (f) Deform2Self \citep{xu2021deformed2self} (g) CPnP \citep{wei2023nonconvex}, and (h) Our PnP-iBPDCA.}
    \label{fig:Rician_denoising_t1w}
\end{figure}

\subsubsection{Comparison with State-of-the-art Methods}
To validate the effectiveness of our proposed method in handling MR images corrupted with various levels of Rician noise, we comprehensively compare our proposed PnP-iBPDCA algorithm with several state-of-the-art techniques. Specifically, WDNN \citep{you2019denoising}, Deform2Self \citep{xu2021deformed2self}, BDCA \citep{wu2022efficient}, and CPnP \citep{wei2023nonconvex} are compared. All comparison codes were either sourced from their officially published versions or kindly provided by the respective authors.

\begin{figure}[t!]
 \subfigure[\scriptsize{Original}]{
	\zoomincludgraphic{0.235\textwidth}{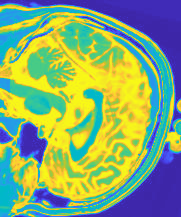}{0.38}{0.62}{0.53}{0.75}{2.7}{help_grid_off}{bottom_right}{line_connection_off}{3}{red}{1.5}{red} 
	}\hspace{-0.345in}
  \subfigure[\scriptsize{Observed (25.76/75.99)}]{
	\zoomincludgraphic{0.235\textwidth}{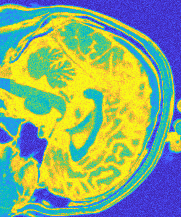}{0.38}{0.62}{0.53}{0.75}{2.7}{help_grid_off}{bottom_right}{line_connection_off}{3}{red}{1.5}{red} 
	}\hspace{-0.345in}
   \subfigure[\scriptsize{Prox-GS (27.54/82.89)}]{
	\zoomincludgraphic{0.235\textwidth}{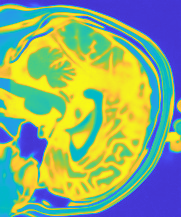}{0.38}{0.62}{0.53}{0.75}{2.7}{help_grid_off}{bottom_right}{line_connection_off}{3}{red}{1.5}{red} 
	}\hspace{-0.345in}
    \subfigure[\scriptsize{WDNN (30.64/89.86)}]{
	\zoomincludgraphic{0.235\textwidth}{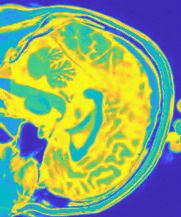}{0.38}{0.62}{0.53}{0.75}{2.7}{help_grid_off}{bottom_right}{line_connection_off}{3}{red}{1.5}{red} 
	}

 \subfigure[\scriptsize{Deform2Self (26.83/83.94)}]{
	\zoomincludgraphic{0.235\textwidth}{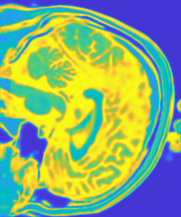}{0.38}{0.62}{0.53}{0.75}{2.7}{help_grid_off}{bottom_right}{line_connection_off}{3}{red}{1.5}{red} 
	}\hspace{-0.345in}
 \subfigure[\scriptsize{BDCA (29.62/88.65)}]{
	\zoomincludgraphic{0.235\textwidth}{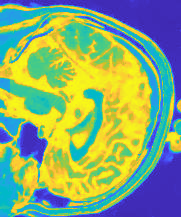}{0.38}{0.62}{0.53}{0.75}{2.7}{help_grid_off}{bottom_right}{line_connection_off}{3}{red}{1.5}{red} 
	}\hspace{-0.345in}
 \subfigure[\scriptsize{CPnP (31.78/93.25)}]{
	\zoomincludgraphic{0.235\textwidth}{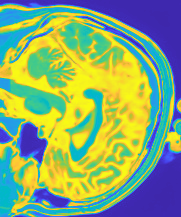}{0.38}{0.62}{0.53}{0.75}{2.7}{help_grid_off}{bottom_right}{line_connection_off}{3}{red}{1.5}{red} 
	}\hspace{-0.345in}
 \subfigure[\scriptsize{Ours (32.14/93.97)}]{
	\zoomincludgraphic{0.235\textwidth}{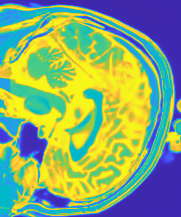}{0.38}{0.62}{0.53}{0.75}{2.7}{help_grid_off}{bottom_right}{line_connection_off}{3}{red}{1.5}{red} 
	}
 \caption{Rician noise removal quantitative results (PSNR(dB)/SSIM($\%$)) with noise level 12.75 on Sagittal viewpoint of an MR image from the T2w sequences of the BrainWeb dataset. Visualization comparison of our scheme and some state-of-the-art Rician noise removal methods: (c) Prox-GS \citep{hurault2022proximal}, (d) BDCA \citep{wu2022efficient} (e) WDNN \citep{you2019denoising} (f) Deform2Self \citep{xu2021deformed2self} (g) CPnP \citep{wei2023nonconvex}, and (h) Our PnP-iBPDCA.}
    \label{fig:Rician_denoising_2}
 \end{figure}
To provide a comprehensive overview of state-of-the-art methods in Ricain noise removal, we summarize the aforementioned approaches as follows: BDCA, an iterative approach with a total-variation prior; WDNN, a convolutional neural network trained end-to-end; Deform2Self, a self-supervised spatial transformer; and CPnP, a PnP method that incorporates RealSN-DnCNN \citep{ryu2019plug}. Additionally, in our subsequent experiment, the Prox-GS denoiser is also included, which is originally designed for Gaussian noise removal \citep{hurault2022proximal}, to establish a baseline.

Typically, MR scans are three-dimensional, with each slide exhibiting significant similarity. Since we are processing the volume as two-dimensional images, we gathered $142$ representative slides from the BrainWeb dataset. 
We present the average Rician noise removal results on three MR sequences of the Brainweb dataset in Table \ref{tab:brainweb_Rician_denoising}. For different Rician noise levels, the PSNR, SSIM, and FSIM indexes are used to measure the performance of these Rician noise removal methods. 
From the quantitative results, our PnP-iBPDCA achieves the best average overall performance among all the methods evaluated. Notably, the difference between our PnP-iBPDCA and other state-of-the-art methods is particularly pronounced in T2w and PDw sequences, indicating that our proposed framework generalizes effectively across various pulse sequences. Furthermore, end-to-end trained networks, such as WDNN\ citep{you2019denoising} and Deform2self\citep{xu2021deformed2self}, face challenges in handling low levels of Rician noise (\ie\ $\sigma = 2.55$). In contrast, our proposed framework demonstrates robust generalization across a wide range of noise levels, from low ($\sigma = 2.55$) to high ($\sigma = 25.5$).

\begin{figure}[t!]
\subfigure[\scriptsize{Original}]{
	\zoomincludgraphic{0.235\textwidth}{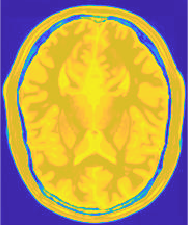}{0.5}{0.65}{0.65}{0.8}{2.7}{help_grid_off}{bottom_left}{line_connection_off}{3}{red}{1.5}{red} 
	}\hspace{-0.345in}
  \subfigure[\scriptsize{Observed (19.62/37.97)}]{
	\zoomincludgraphic{0.235\textwidth}{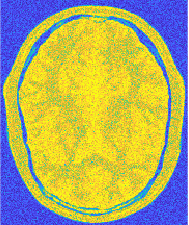}{0.5}{0.65}{0.65}{0.8}{2.7}{help_grid_off}{bottom_left}{line_connection_off}{3}{red}{1.5}{red}  
	}\hspace{-0.345in}
   \subfigure[\scriptsize{Prox-GS (23.13/69.88)}]{
	\zoomincludgraphic{0.235\textwidth}{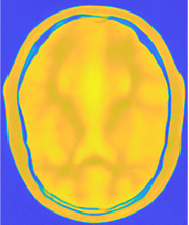}{0.5}{0.65}{0.65}{0.8}{2.7}{help_grid_off}{bottom_left}{line_connection_off}{3}{red}{1.5}{red} 
	}\hspace{-0.345in}
    \subfigure[\scriptsize{WDNN (29.29/80.70) }]{
	\zoomincludgraphic{0.235\textwidth}{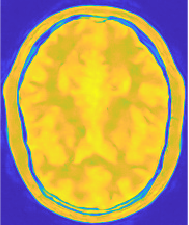}{0.5}{0.65}{0.65}{0.8}{2.7}{help_grid_off}{bottom_left}{line_connection_off}{3}{red}{1.5}{red}  
	}

  \subfigure[\scriptsize{Deform2Self (23.41/75.41)}]{
	\zoomincludgraphic{0.235\textwidth}{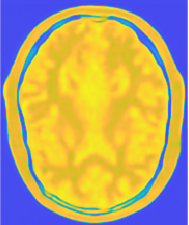}{0.5}{0.65}{0.65}{0.8}{2.7}{help_grid_off}{bottom_left}{line_connection_off}{3}{red}{1.5}{red} 
	}\hspace{-0.345in}
 \subfigure[\scriptsize{BDCA (26.79/76.64)}]{
	\zoomincludgraphic{0.235\textwidth}{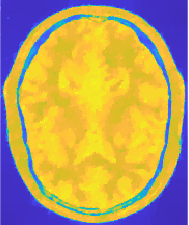}{0.5}{0.65}{0.65}{0.8}{2.7}{help_grid_off}{bottom_left}{line_connection_off}{3}{red}{1.5}{red}  
	}\hspace{-0.345in}
 \subfigure[\scriptsize{CPnP (30.09/85.52)}]{
	\zoomincludgraphic{0.235\textwidth}{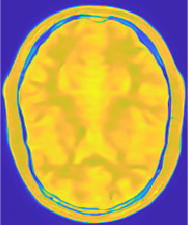}{0.5}{0.65}{0.65}{0.8}{2.7}{help_grid_off}{bottom_left}{line_connection_off}{3}{red}{1.5}{red}  
	}\hspace{-0.345in}
 \subfigure[\scriptsize{Ours (30.94/88.98)}]{
	\zoomincludgraphic{0.235\textwidth}{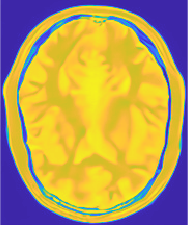}{0.5}{0.65}{0.65}{0.8}{2.7}{help_grid_off}{bottom_left}{line_connection_off}{3}{red}{1.5}{red} 
	}
 \caption{Rician noise removal results (PSNR(dB)/SSIM($\%$)) with noise level 25.5 on the Axial viewpoint of an MR image from the PDw sequences of the BrainWeb dataset. Visualization comparison of our scheme and some state-of-the-art Rician noise removal methods: (c) Prox-GS \citep{hurault2022proximal}, (d) BDCA \citep{wu2022efficient} (e) WDNN \citep{you2019denoising} (f) Deform2Self \citep{xu2021deformed2self} (g) CPnP \citep{wei2023nonconvex}, and (h) Our PnP-iBPDCA.}
    \label{fig:Rician_denoising}
\end{figure}

To better present the effectiveness of the proposed algorithm, we display the visual results of Rician noise removal on three MR sequences with a ``parula'' color map. We present Rician noise removal results on a single T1w MR image in Figure \ref{fig:Rician_denoising_t1w} with noise level $\sigma =12.75$. Compared to state-of-the-art techniques, our proposed method preserves fine details. Even though CPnP \citep{wei2023nonconvex} has a very similar quantitative result to our proposed framework, we observe that CPnP created artifacts and distorted the shape of lateral ventricles from the zoomed-in part while our proposed scheme best preserves its shape among all state-of-the-art methods.

In the T2w MR sequence, we present the visual results in Figure \ref{fig:Rician_denoising_2} with Rician noise $\sigma =12.75$. We observe that end-to-end trained networks specially for the task of Rician noise removal often over-smooth and lose fine details. In contrast, our proposed PnP-iBPDCA method effectively manages Rician noise while preserving the intricate structures of tissues and brain anatomy. In particular, our proposed method preserves the fine line of tentorium, shown in the zoomed-in area, and best resembles the ground truth. A clear line of the tentorium is critical in measuring the tentorial angle for the diagnosis of achondroplasia.

Finally, we present the visual results of the PDw MR sequence with heavy Rician noise ($\sigma=25.5$) in Figure \ref{fig:Rician_denoising}. As claimed before, our proposed PnP-iBPDCA can better preserve the detailed information while removing the Rician noise. Under the image corrupted by heavy noise, our results can still recover a clear structure. 
In the zoomed-in area, our proposed method shows a much clearer margin of the irregular structure, known as an infarct, compared to state-of-the-art methods. This tissue is essential in clinical settings for detecting old cerebral infarction and cerebral atrophy.

Overall, the findings from both the quantitative and visual results reveal that our proposed method, PnP-iBPDCA, outperforms other techniques in preserving fine brain tissue details and maintaining their original coloration while effectively removing Rician noise. In contrast, existing iterative-based, deep learning-based, and PnP methods exhibit artifacts, discoloration, and distortion, while some retain residual noise.

\subsection{Application to Phase Retrieval}
In this subsection, we devote to applying our proposed method to solve the phase retrieval problem \eqref{eq:phase_retrieval}. Note that the additive noise $\omega$ in \eqref{eq:phase_retrieval_model} can be modeled as either shot noise or additive white Gaussian noise (AWGN). Both types of noises will be experimented with in this section.
As mentioned in \eqref{decompostion}, we can reformulate the model \eqref{eq:phase_retrieval} into the form of Problem \eqref{eq:DCA_model} with 
\begin{equation}\label{decomposition2}
    f_1({\bf x}):= \frac{1}{4} \left\Vert |\mathcal{K}{\bf x}|^2\right \Vert^2 + \frac{1}{4} \Vert {\bf d} \Vert^2,~
    f_2({\bf x}):= \frac{1}{2} \left\langle {\bf d},|\mathcal{K}{\bf x}|^2\right\rangle, ~{\rm and}~ 
    g({\bf x}):= \varsigma \vartheta({\bf x}). 
\end{equation}

Indeed, $\vartheta(\cdot)$ in \eqref{decomposition2} represents the generic regularization term, common choice for handcrafted priors are translation-invariant Haar pyramid (TIHP) tight frame \citep{shi2015sparse} and total variation \citep{chang2016phase,gaur2015sparsity,tillmann2016dolphin}. On the other hand, deep prior was also widely used to solve the phase retrieval problem under the plug-and-play (PnP) framework \citep{katkovnik2017phase,metzler2018prdeep,wei2022tfpnp}. However, they shared a common problem which is the lack of theoretical convergence guarantee. To address this challenge, we aim to apply our proposed convergent framework.

Neither $f_1 - f_2$ nor $f_1$ in the above setting possesses a global Lipschitz gradient with respect to Euclidean geometry. Observing that, \cite{bolte2018first} was the first to propose tackling it with Bregman variant algorithms. Inheriting their work, we adopt their kernel function $h = \frac{1}{4} \Vert \cdot \Vert^4 + \frac{1}{2} \Vert \cdot \Vert^2$ and apply our proposed PnP-iBPDCA to tackle the phase retrieval problem. 
Before that, we first present some properties to make sure our kernel function $h = \frac{1}{4} \Vert \cdot \Vert^4 + \frac{1}{2} \Vert \cdot \Vert^2$ satisfies the assumptions of the proposed framework.

\begin{proposition}
    Let $h = \frac{1}{4} \Vert \cdot \Vert^4 + \frac{1}{2} \Vert \cdot \Vert^2$, then the kernel function $h$ enjoys the following properties:
    \begin{itemize}
        \item[(i)] $h\in \mathcal{C}^2$, is 1-strongly convex and of Legendre type.
        \item[(ii)] $\nabla h$ is Lipschitz is bounded on any bounded subset of $\mathbb{R}^n$.
        \item[(iii)]  $\nabla h$ is a bijection from $\mathbb{R}^n$ to $\mathbb{R}^n$, and its inverse $(\nabla h)^{-1} = \nabla h^*$.
    \end{itemize}
\end{proposition}
To apply the proposed iBPDCA and PnP-iBPDCA, we have to ensure Assumptions~\ref{asm:assumption_on_h} holds.  
\begin{proposition} 
    Let $f_1, f_2, g $ be splited according to \eqref{decomposition2}. Then, we have
    \begin{itemize}
        \item[(i)] $f_1, f_2$ are proper and convex.
        \item[(ii)] $(f_1,h)$ is $L$-smad (see Definition~\ref{lem:full_extended_descent_lemma}) on $\mathbb{R}^n$ for any  $L\geq 3 \sum_{r=1}^m \Vert \mathcal{K}_r \Vert^2$.
    \end{itemize}
\end{proposition}
    \begin{proof}  
(i) By the setting of $f_1$ in \eqref{decomposition2}, we can express it element-wise ${\bf d}[r] = |\mathcal{K}_r {\bf x}| + \omega[r]$.
            $ \nabla f_1({\bf x}) = \sum^m_{r=1}  |\mathcal{K}_r {\bf x}|^2 \mathcal{K}_r^{\dagger} \mathcal{K}_r {\bf x} $ and 
            $     \nabla^2 f_1({\bf x}) = 3\sum^m_{r=1}  |\mathcal{K}_r {\bf x}|^2 \mathcal{K}_r^{\dagger} \mathcal{K}_r 
             \succeq 0.$ 
            This completes the convexity of $f_2$.  
            Similarly, we have $\nabla f_2({\bf x}) = \sum_{r=1}^m  {\bf d}[r] \mathcal{K}_r^{\dagger} \mathcal{K}_r {\bf x}$ and $
                \nabla^2 f_2({\bf x}) = \sum_{r=1}^m  {\bf d}[r] \mathcal{K}_r^{\dagger} \mathcal{K}_r 
                 \succeq 0.$ This implies the convexity of $f_1$.  
            (ii) A tighter bound on $L$ is desirable such that a larger step size $\lambda$ can be chosen leading to a slower but more stable convergence when applying the PnP framework. Different from \citep{bolte2018first,takahashi2022new}, we have adopted a similar approach as \cite{godeme2023provable}. 
            Because both $f_1$ and $h$ are convex, we only need to ensure the noLips condition (\ie\  $Lh-f_1$ is convex).
            For all $ {\bf x,u} \in \mathbb{R}^n$, we have
            \begin{align*}
                \left \langle {\bf u}, \nabla^2 f_1({\bf x}) {\bf u} \right \rangle &= 
                3\sum^m_{r=1}  |\mathcal{K}_r {\bf x}|^2 |\mathcal{K}_r {\bf u}|^2\leq 3 \Vert {\bf x} \Vert^2 \Vert {\bf u} \Vert^2 \sum_{r=1}^m \Vert \mathcal{K}_r \Vert^2, 
            \end{align*}
            and
            \begin{align*}
                \left \langle {\bf u} , \nabla ^2 h({\bf x}) {\bf u} \right \rangle &=\left( \Vert {\bf x} \Vert^2 + 1 \right) \Vert {\bf u} \Vert^2 + 2| \left \langle {\bf x,u} \right \rangle | ^2  \geq \Vert {\bf x} \Vert^2 \Vert {\bf u} \Vert^2.
            \end{align*}
            Thus for all $L\geq 3 \sum_{r=1}^m \Vert \mathcal{K}_r \Vert^2$, we have 
            $L \nabla^2 h({\bf x}) - \nabla^2 f_1({\bf x}) \succeq 0$,
            which means $Lh-f_1$ is convex. 
            This completes the proof. 
    \end{proof}
 \begin{remark}
            The above method has the advantage of avoiding eigenvalue computation, where classically \cite{bolte2018first} proposed to ensure convexity of $Lh-f_1$ through finding a $L>0$ such that $L\lambda_{\min}(\nabla^2 h({\bf x})) \geq \lambda_{\max}(\nabla^2 f_1({\bf x}))$. In our experiment, we use the built-in function{\verb| torch.linAlgorithmeigvals()|} to deduce the smallest possible value of $L$.
 \end{remark}

Taking first-order derivative of $f_1$ and $f_2$, we have
\[
    \nabla f_1({\bf x}) = \Re\left(\mathcal{K}^{\dagger}[\mathcal{K}{\bf u}\odot|\mathcal{K} {\bf x}|^2]\right) \quad \text{and} \quad \nabla f_2({\bf x}) = \Re \left( \mathcal{K}^{\dagger} [\mathcal{K}{\bf x} \odot {\bf d} ] \right).
\]     
Now we can apply Algorithm~\ref{alg:PnP-iBPDCA} and derive the following iterative scheme for solving \eqref{eq:phase_retrieval} as follows:
\begin{align*}
    \xi^k &= \Re \left( \mathcal{K}^{\dagger} [\mathcal{K} {\bf u} \odot {\bf d} ] \right), \\
    {\bf x}^{k+1} &= \mathcal{D}_\gamma \circ \nabla h^* (\nabla h({\bf x}^k) - \lambda (\nabla f_1({\bf x}^k) -\xi^k)),
\end{align*}
where $\nabla h^*({\bf x}) = t^* {\bf x}$, $t^*$ represents the \textit{unique positive real root} of the polynomial $t^3 \Vert {\bf x} \Vert^2 + t + 1 = 0$ \cite[Proposition~5.1]{bolte2018first}.

In Proposition~\ref{prop:bregman_denoiser_defined}, it has been assumed that $\psi_\gamma \circ \nabla h^*$ is strictly convex. However, with the newly proposed construction of $\psi_\gamma$ along with $h=\frac{1}{4}\Vert \cdot \Vert^4 +\frac{1}{2}\Vert \cdot \Vert^2$, the convexity of $\psi_\gamma \circ \nabla h^*$ remains unknown. In the following, we aim to illustrate this issue.

\begin{lemma} {\rm{({\bf Validation the convexity of $\psi_\gamma \circ \nabla h^*$})}}
     Let $h  = \frac{1}{4} \|\cdot\|^4 +\frac{1}{2} \|\cdot\|^2$ and $\mathcal{D}_\gamma$ is the GS denoiser defined in \eqref{gaussian_denoiser}.  Assume ${\rm Im}(\mathcal{D}_\gamma )\subset  {\rm int}{\rm dom}(h)$. Then, $\psi_\gamma \circ \nabla h^*$ is convex on $ {\rm int}{\rm dom}(h^*)$, where $\psi$ is given in Proposition~\ref{prop:bregman_denoiser_defined}.
\end{lemma}
 \begin{proof}
    Define $\eta_\gamma:= \psi_\gamma \circ \nabla h^*$. By simple computation, we have
$ \nabla \eta_\gamma ({\bf x}) = \nabla^2 h^*({\bf x}) \cdot \nabla \psi_\gamma (\nabla h^*({\bf x}))$ and
    \begin{align*}
        \nabla^2 \eta_\gamma({\bf x}) &= (\nabla^2 h^*({\bf x}))^2 \cdot \nabla^2 \psi_\gamma (\nabla h^*({\bf x})) + \nabla^3 h^*({\bf x}) \cdot \nabla \psi_\gamma (\nabla h^*({\bf x}))\\
        &= (\nabla^2 h^*({\bf x}))^2 \cdot \nabla^2 \psi_\gamma (\nabla h^*({\bf x})).
    \end{align*}
    It follows from the definition of $h$ that
    \begin{alignat*}{3}
        \nabla h({\bf y}) = (\Vert {\bf y}\Vert^2 + 1) {\bf y}, ~
        \nabla^2 h({\bf y}) = (\Vert {\bf y}\Vert^2 + 1) I_d + 2{\bf y}{\bf y}^{\top}, ~{\rm and}~
         \nabla^3 h({\bf y}) = 2(I_d \otimes {\bf y} + {\bf y} \otimes I_d).
       \end{alignat*} 
       Besides, it follows from Proposition 5.1 in \cite{bolte2018first} that
       \begin{alignat*}{3}  
        \nabla h^*({\bf x}) = t^* {\bf x}, \quad 
        &&\nabla^2 h^*({\bf x}) = t^* I_d, 
    \end{alignat*}
    Let ${\bf y} =\nabla h^{-1}({\bf x})$, then ${\bf x} =\nabla h({\bf y})$ and $       \nabla^2 \eta_\gamma(\nabla h({\bf y})) = (\nabla^2 h^*(\nabla h({\bf y})))^2 \cdot \nabla^2 \psi_\gamma({\bf y}) = (t^* I_d)^2 \cdot \nabla^2 \psi_\gamma({\bf y}).$ 
    Combining the definition of $\nabla \psi_\gamma$ in Proposition~\ref{prop:bregman_denoiser_defined}, we have
    \begin{align*}
        \nabla \psi_\gamma({\bf y}) &= \nabla^2 h({\bf y}) \cdot ({\bf y} - \nabla g_\gamma ({\bf y})) 
        = ((\Vert {\bf y} \Vert^2 + 1) I_d + 2{\bf y}{\bf y}^{\top}) \cdot \mathcal{D}_\gamma ({\bf y}),
        \end{align*}
        and
        \begin{align*}
        \nabla^2 \psi_\gamma ({\bf y}) &= \nabla^3 h({\bf y}) \cdot ({\bf y} - \nabla g_\gamma ({\bf y})) + \nabla^2 h({\bf y}) \cdot J_{\mathcal{D}_\gamma({\bf y})} \\
        &= 2(I_d \otimes {\bf y} + {\bf y} \otimes I_d) \cdot \mathcal{D}_\gamma({\bf y}) + ((\Vert {\bf y}\Vert^2 + 1) I_d + 2{\bf y}{\bf y}^{\top}) \cdot J_{\mathcal{D}_\gamma({\bf y})},
    \end{align*}
    where $\otimes$ represents Kronecker product. Since ${\bf x}=\nabla h({\bf y})$ and hence ${\bf y} = \nabla h^*({\bf x})$ according to $\nabla h^*=\nabla h^{-1}$, we have
$$       \nabla^2 \eta_\gamma({\bf x}) =  (t^* I_d)^2 \cdot \left\{( 2(I_d \otimes {\bf y} + {\bf y} \otimes I_d) \cdot \mathcal{D}_\gamma({\bf y}) + ((\Vert {\bf y}\Vert^2 + 1) I_d + 2{\bf y}{\bf y}^{\top}) \cdot J_{\mathcal{D}_\gamma({\bf y})} \right \}.
$$
Consequently, for all ${\bf d},{\bf y}\in \mathbb{R}^n_+$ in image manifold, we have $\left \langle \nabla^2 \eta_\gamma ({\bf y}) {\bf d} ,{\bf d} \right \rangle > 0$, which comes from the fact that $J_{\mathcal{D}_\gamma}$ is positive definite \citep{hurault2022proximal}.
We have now verified the strict convexity in the image manifold.
\end{proof}

\begin{remark}
    Note that the convexity of $\psi_\gamma \circ \nabla h^*$ can be guaranteed with commonly used kernel functions such as the Hellinger $h({\bf x}) = -\sqrt{1-{\bf x}^2}$ \citep{bauschke2017descent} and all polynomial kernel functions \citep{ding2023nonconvex}. These kernel functions cover a wide range of applications including non-negative matrix factorization, low-rank minimization, and phase retrieval.
\end{remark}

\subsubsection{Comparison with State-of-the-art methods}
Following the aforementioned analysis, our algorithm can be used to handle the phase retrieval problem with Poisson noise. In the experiment, we consider coded diffraction pattern (CDP) with $m=4$ intensity-only measurements \citep{candes2015phase}. 

\begin{table}[!t]
    \centering
    \setlength{\tabcolsep}{0.05in} 
    \resizebox{0.98\columnwidth}{!}{%
    \begin{tabular}{c|c|>{\centering\arraybackslash}p{1.9cm}|>{\centering\arraybackslash}p{1.8cm}|>{\centering\arraybackslash}p{1.8cm}|>{\centering\arraybackslash}p{1.8cm}|>{\centering\arraybackslash}p{1.8cm}|>{\centering\arraybackslash}p{1.8cm}|>{\centering\arraybackslash}p{1.8cm}|>{\centering\arraybackslash}p{1.8cm}}
    \hline
        \cellcolor[HTML]{EFEFEF}& \cellcolor[HTML]{EFEFEF}& \cellcolor[HTML]{EFEFEF} & \multicolumn{3}{c|}{\cellcolor[HTML]{EFEFEF}Traditional}   & \multicolumn{2}{c|}{\cellcolor[HTML]{EFEFEF}Supervised}  & \multicolumn{2}{c}{\cellcolor[HTML]{EFEFEF}Plug-and-Play} \\  
        \hhline{>{\arrayrulecolor[HTML]{EFEFEF}}->{\arrayrulecolor{black}}|>{\arrayrulecolor[HTML]{EFEFEF}}->{\arrayrulecolor{black}}|>{\arrayrulecolor[HTML]{EFEFEF}}->{\arrayrulecolor{black}}|-------|}
         \cellcolor[HTML]{EFEFEF}\multirowcell{-2}{ Noise Type} &\cellcolor[HTML]{EFEFEF} \multirowcell{-2}{Noise Level} & \cellcolor[HTML]{EFEFEF}\multirow{-2}{*}{\diagbox[height=2.\line]{\raisebox{0.04in}{\scriptsize Index~}}{\scriptsize \hspace{-0.1in}Method}} &\cellcolor[HTML]{EFEFEF} WF &\cellcolor[HTML]{EFEFEF} DOLPHIn &\cellcolor[HTML]{EFEFEF} AmpFlow &\cellcolor[HTML]{EFEFEF}  TFPnP &\cellcolor[HTML]{EFEFEF} TFPnP$^\ast$  &\cellcolor[HTML]{EFEFEF} prDeep &\cellcolor[HTML]{EFEFEF} Ours \\ \hline
        \multirowcell{9}{Gaussian}  & \multirowcell{3}{SNR = 10} & PSNR  & 18.57 & 24.81 & 17.52 & 28.79 & \underline{29.99} & 27.65 & \cellcolor[HTML]{D4FFD3}\tBF{30.47}  \\ 
        ~ & ~ & SSIM & 35.81 & 60.02 & 32.07 & 84.40 & \underline{87.27}  & 78.98 & \cellcolor[HTML]{D4FFD3}\tBF{87.86} \\ 
        ~ & ~ & Time & 0.88 & 8.47 & 1.45 & \underline{0.50} & \cellcolor[HTML]{D4FFD3}\tBF{0.04} & 10.33 & 0.74 \\ \cline{2-10}\hhline{~~~~~~~-~~}
        ~ & \multirowcell{3}{SNR = 15} & PSNR  & 24.85 & 27.59 & 22.77 & 30.54 & \underline{32.56}  & 29.67 & \cellcolor[HTML]{D4FFD3}\tBF{33.11} \\ 
        ~ & ~ & SSIM  & 60.06 & 73.05 & 52.22 & 87.90 & \underline{92.10} & 82.90 & \cellcolor[HTML]{D4FFD3}\tBF{92.71}   \\ 
        ~ & ~ & Time & 0.68 & 8.46 & 0.66 & \underline{0.38} & \cellcolor[HTML]{D4FFD3}\tBF{0.05} & 10.58 & 0.88 \\ \cline{2-10}\hhline{~~~~~~~-~~}
        ~ & \multirowcell{3}{SNR = 20} & PSNR  & 30.27 & 28.96 & 27.53 & 33.69 & \underline{35.18} & 32.53 & \cellcolor[HTML]{D4FFD3}\tBF{35.58} \\ 
        ~ & ~ & SSIM  & 78.62 & 79.43 & 69.88 & 93.60 & \underline{95.39}  & 89.80 & \cellcolor[HTML]{D4FFD3}\tBF{95.28}\\ 
        ~ & ~ & Time & 0.54 & 8.69 & 0.52 & \underline{0.14} & \cellcolor[HTML]{D4FFD3}\tBF{0.05} & 8.25 & 0.81  \\ \hline
        \multirowcell{9}{Poisson }  & \multirowcell{3}{$\alpha=9$} & PSNR  & 34.28 & 28.46 & 36.02 & 36.66 & \cellcolor[HTML]{D4FFD3}\tBF{40.55} & \underline{39.60} & 39.02  \\ 
        ~ & ~ & SSIM  & 88.94 & 76.88 & 91.85 & 95.86 & \cellcolor[HTML]{D4FFD3}\tBF{98.39} & 97.15 & \underline{97.73} \\ 
        ~ & ~ & Time & 0.53 & 7.91 & \cellcolor[HTML]{D4FFD3}\tBF{0.39} & \underline{0.15} & 0.40 & 9.76 & 1.14  \\ \cline{2-10}\hhline{~~~~~~~-~~}
        ~ & \multirowcell{3}{$\alpha = 27$} & PSNR & 24.34 & 23.26 & 25.40 & 29.97 & \cellcolor[HTML]{D4FFD3}\tBF{33.98} & \underline{33.46} & 32.75 \\ 
        ~ & ~ & SSIM & 57.41 & 52.57 & 61.52 & 86.29 & \cellcolor[HTML]{D4FFD3}\tBF{94.33} & \underline{92.98} & 92.67 \\ 
        ~ & ~ & Time & 0.66 & 7.84 & 0.57 & \cellcolor[HTML]{D4FFD3}\tBF{0.13} & \underline{0.15} & 8.37 & 0.73\\ \cline{2-10}\hhline{~~~~~~~-~~}
        ~ & \multirowcell{3}{$\alpha = 81$} & PSNR& 13.04 & 14.85 & 13.18 & 26.55 & \cellcolor[HTML]{D4FFD3}\tBF{28.25} & \underline{26.89} & 26.81  \\ 
        ~ & ~ & SSIM & 15.86 & 19.33 & 16.03 & 78.42 & \cellcolor[HTML]{D4FFD3}\tBF{83.41} & \underline{80.81} & 78.49\\ 
        ~ & ~ & Time & 1.43 & 8.05 & 1.94 & \underline{0.13} & \cellcolor[HTML]{D4FFD3}\tBF{0.07} & 5.54 & 0.63 \\ \hline
    \end{tabular}}
    \caption{Quantitative comparison of state-of-the-art phase retrieval algorithms with Average PSNR (dB), SSIM (\%), and interference time (second) on PrDeep12 dataset. Degradation with CDP measurement with $m=4$ and Gaussian noise levels SNR = \{10, 15, 20\}, Poisson noise levels $\alpha =\{9,27,81\}$. The best results are in \colorbox[HTML]{BBFFBB}{\tBF{green}} whereas the second-best results are underlined.}    \label{tab:quantitative_comparison_phase_retrieval}

\end{table}
The CDP measurement model uses a spatial light modulator (SLM) to spread a target's frequency information, hoping it will be easier to construct the image from the observed image.
More specifically, we set the number of measurements to four (\ie\  $\mathcal{K}=[(\mathcal{FD}_1)^{\top}, (\mathcal{FD}_2)^{\top}, \ldots (\mathcal{FD}_4)^{\top})]$). 
Then, we sample and reconstruct $12$ commonly used test images in the phase retrieval task, including $6$ ``natural'' and $6$ ``unnatural'' images used in \cite{metzler2018prdeep} and \cite{wei2022tfpnp}. In the following, we are comparing our proposed method with three classical approaches\footnote{Implementation of classical approaches WF, AmpFlow are tested through PhasePack (\url{https://github.com/tomgoldstein/phasepack-matlab}).}, which are WF \citep{candes2015phase}, AmplitudeFlow (abbrev. as AmpFlow) \citep{wang2017solving}; dictionary learning-based method DOLPHIn \citep{tillmann2016dolphin};
and two PnP approaches prDeep \citep{metzler2018prdeep} and TFPnP \citep{wei2022tfpnp}. As mentioned in \cite{wei2022tfpnp}, only a single policy network is trained for phase retrieval. However, we observed a significant decrease in the quantitative result when the training mask and the testing mask were mismatched, which is also reported by \cite{liu2023prista}. We retrained the TFPnP model to our testing mask (denoted by TFPnP$^\ast$) to reproduce the result reported in \cite{wei2022tfpnp}.

\begin{figure}[!t]
\subfigure[\scriptsize{Original}]{
    \zoomincludgraphic{0.235\textwidth}{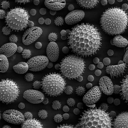}{0.34}{0.25}{0.62}{0.53}{1.5}{help_grid_off}{up_left}{line_connection_off}{3}{green}{1.5}{blue}
}\hspace{-0.34in}
\subfigure[\scriptsize{WF (27.34/85.47)}]{
    \zoomincludgraphic{0.235\textwidth}{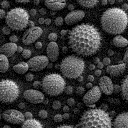}{0.34}{0.25}{0.62}{0.53}{1.5}{help_grid_off}{up_left}{line_connection_off}{3}{green}{1.5}{blue}
}\hspace{-0.34in}
\subfigure[\scriptsize{AmpFlow (25.52/80.03)}]{
	\zoomincludgraphic{0.235\textwidth}{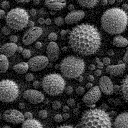}{0.34}{0.25}{0.62}{0.53}{1.5}{help_grid_off}{up_left}{line_connection_off}{3}{green}{1.5}{blue}
}\hspace{-0.34in}
\subfigure[\scriptsize{DOLPHIn (24.88/77.65)}]{
	\zoomincludgraphic{0.235\textwidth}{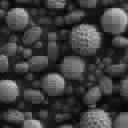}{0.34}{0.25}{0.62}{0.53}{1.5}{help_grid_off}{up_left}{line_connection_off}{3}{green}{1.5}{blue}
}
\subfigure[\scriptsize{TFPnP (27.26/85.14)}]{
    \zoomincludgraphic{0.235\textwidth}{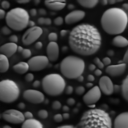}{0.34}{0.25}{0.62}{0.53}{1.5}{help_grid_off}{up_left}{line_connection_off}{3}{green}{1.5}{blue}
}\hspace{-0.34in}
\subfigure[\scriptsize{TFPnP$^*$ (30.77/93.90)}]{
    \zoomincludgraphic{0.235\textwidth}{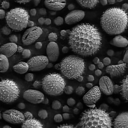}{0.34}{0.25}{0.62}{0.53}{1.5}{help_grid_off}{up_left}{line_connection_off}{3}{green}{1.5}{blue}
}\hspace{-0.34in}
\subfigure[\scriptsize{prDeep (28.39/88.81)}]{
    \zoomincludgraphic{0.235\textwidth}{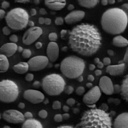}{0.34}{0.25}{0.62}{0.53}{1.5}{help_grid_off}{up_left}{line_connection_off}{3}{green}{1.5}{blue}
}\hspace{-0.34in}
\subfigure[\scriptsize{Ours (32.22/95.43)}]{
    \zoomincludgraphic{0.235\textwidth}{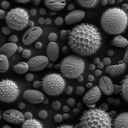}{0.34}{0.25}{0.62}{0.53}{1.5}{help_grid_off}{up_left}{line_connection_off}{3}{green}{1.5}{blue}
}
\caption{Reconstruction results (PSNR(dB)/SSIM(\%)) of $128\times128$ image from four noisy intensity-only CDP measurements (Gaussian SNR=15). Visualization comparison of our scheme and some state-of-the-art PR algorithms: (b)  WF \citep{candes2015phase}, (c) AmplitudeFlow \citep{wang2017solving}, (d) DOLPHIn \citep{tillmann2016dolphin}, (e) TFPnP \citep{wei2022tfpnp}, (f) ${\rm TFPnP}^\ast$, (g) prDeep \citep{metzler2018prdeep}, and (h) Our PnP-iBPDCA.}
\label{fig:pr_result} 
\end{figure}

Following the discussion of the phase retrieval in the beginning, both shot noise with $\omega \sim \mathcal{N}\left(0,\alpha^2 |\mathcal{K} {\bf }|^2\right)$ and additive white Gaussian noise $\omega \sim \mathcal{N}(0, 10^{-\frac{{\rm SNR}}{10}})$ are considered. Hence, we conduct the phase retrieval on Gaussian noise levels $\alpha =\{10, 15, 20\}$ and Poisson noise levels $\alpha = \{9, 27, 81\}$ and to validate the effectiveness of the proposed scheme on the experimental aspect. The average numerical results with PSNR, SSIM, and Time are listed in Table \ref{tab:quantitative_comparison_phase_retrieval}. For Gaussian noise, our method has the best visual result out of the state-of-the-art methods while maintaining interference time similar to supervised methods. However, a slight dent in the result is observed for Poisson noise. We have expected this outcome since our model emphasizes theoretical guarantee while following the model from \cite{metzler2018prdeep} means a model mismatching between the original noise model and the modeling.

\begin{figure}[!t]
\subfigure[\scriptsize{Original}]{
    \zoomincludgraphic{0.235\textwidth}{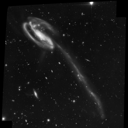}{0.22}{0.6}{0.42}{0.83}{2.2}{help_grid_off}{bottom_left}{line_connection_off}{3}{green}{1.5}{blue}
}\hspace{-0.34in}
\subfigure[\scriptsize{WF (23.38/33.11)}]{
    \zoomincludgraphic{0.235\textwidth}{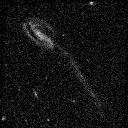}{0.22}{0.6}{0.42}{0.83}{2.2}{help_grid_off}{bottom_left}{line_connection_off}{3}{green}{1.5}{blue}
}\hspace{-0.34in}
\subfigure[\scriptsize{AmpFlow (23.36/34.26)}]{
    \zoomincludgraphic{0.235\textwidth}{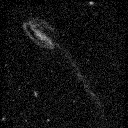}{0.22}{0.6}{0.42}{0.83}{2.2}{help_grid_off}{bottom_left}{line_connection_off}{3}{green}{1.5}{blue}
}\hspace{-0.34in}
\subfigure[\scriptsize{DOLPHIn (21.28/17.50)}]{
	\zoomincludgraphic{0.235\textwidth}{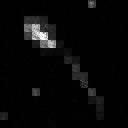}{0.22}{0.6}{0.42}{0.83}{2.2}{help_grid_off}{bottom_left}{line_connection_off}{3}{green}{1.5}{blue}
}
\subfigure[\scriptsize{TFPnP (31.88/83.35)}]{
    \zoomincludgraphic{0.235\textwidth}{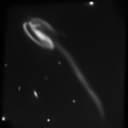}{0.22}{0.6}{0.42}{0.83}{2.2}{help_grid_off}{bottom_left}{line_connection_off}{3}{green}{1.5}{blue}
}\hspace{-0.34in}
\subfigure[\scriptsize{TFPnP$^*$ (36.89/91.73)}]{
    \zoomincludgraphic{0.235\textwidth}{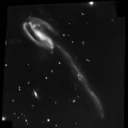}{0.22}{0.6}{0.42}{0.83}{2.2}{help_grid_off}{bottom_left}{line_connection_off}{3}{green}{1.5}{blue}
}\hspace{-0.34in}
\subfigure[\scriptsize{prDeep (35.45/90.35)}]{
    \zoomincludgraphic{0.235\textwidth}{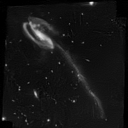}{0.22}{0.6}{0.42}{0.83}{2.2}{help_grid_off}{bottom_left}{line_connection_off}{3}{green}{1.5}{blue}
}\hspace{-0.34in}
\subfigure[\scriptsize{Ours (36.14/90.89)}]{
    \zoomincludgraphic{0.235\textwidth}{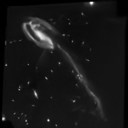}{0.22}{0.6}{0.42}{0.83}{2.2}{help_grid_off}{bottom_left}{line_connection_off}{3}{green}{1.5}{blue}
}
\caption{Reconstruction results (PSNR(dB)/SSIM(\%)) of $128\times128$ image from four noisy intensity-only CDP measurements (Poisson $\alpha=27$). Visualization comparison of our scheme and some state-of-the-art PR algorithms: (b)  WF \citep{candes2015phase}, (c) AmplitudeFlow \citep{wang2017solving}, (d) DOLPHIn \citep{tillmann2016dolphin}, (e) TFPnP \citep{wei2022tfpnp}, (f) ${\rm TFPnP}^\ast$, (g) prDeep \citep{metzler2018prdeep}, and (h) Our PnP-iBPDCA.}
\label{fig:pr_result_2} 
\end{figure}

Furthermore, we report the corresponding visual results in Figure \ref{fig:pr_result} and Figure \ref{fig:pr_result_2}. More specifically, the phase retrieval results under the degradation of four CDP measurements and Gaussian noise level $\mbox{SNR}=15$ are displayed in Figure \ref{fig:pr_result}. The traditional methods, like WF and AmpFlow, fail to remove the noise and the deep learning-based methods ignore the detailed information in the Pollen image. The oversmoothing also occurs in the results of TFPnP$^\ast$. Whereas our method unifies the traditional approach with a deep prior under the PnP framework and generates the best phase retrieval results. From the visualization performance, our methods remove Gaussian noise while preserving the intrinsic detailed structure, the superiority result is more pronounced in the zoomed-in part. 

For the phase retrieval with Poisson noise, not only the numerical results are reported but also the visual results with Poisson noise level $\alpha=27$ are presented in Figure \ref{fig:pr_result_2}. 
With medium-level Poisson noise, although the traditional methods recover the phase from the CDP measurement, they still fail to remove the noise. As to the deep-learning-based approaches, while the noise-free image is obtained, the typical oversmoothing happens. 
Our result is noise-free and better approaching the original phase in detail.

Overall, based on the phase retrieval results for both noise types, the proposed scheme achieves the best overall performance in terms of theoretical, numerical, and interference speed. Although our method exhibits a slight decrease in performance compared to the TFPnP and prDeep, their framework lacks theoretical convergence guarantees. Besides, prDeep takes the longest time to process an image. 
Furthermore, although TFPnP employs pre-trained denoisers for PnP, then they rely on reinforcement learning to adjust internal parameters. 
It takes around 4.5 hours to obtain TFPnP$^\ast$ for our testing mask for each noise type while the pretrained model (TFPnP) provided by the authors has significantly worse visual performance than its retrained counterpart (TFPnP$^\ast$).
End-to-end learning approaches are often time-consuming, worst-case performance on unfamiliar inputs deteriorates significantly \cite{chen2022learning}, and are prone to gradient explosion. In contrast, our method only requires minor parameter tuning to guarantee performance regardless of the inputs such as different mask numbers, noise type, and noise strength.

\section{Conclusions}
\label{sec:conclusions}

This paper explored an inertial difference-of-convex algorithm to minimize a difference-of-convex function with a weakly convex function. Our proposed method extends the existing DCA approach and introduces inertial techniques to accelerate convergence. The convergence of our proposed method is established using the Kurdyka-{\L}ojasiewicz property. By incorporating Plug-and-Play (PnP) with a gradient step denoiser, we leveraged the benefits of deep priors, further enhancing the performance of our algorithm in image restoration tasks. The convergence of this PnP variant is also guaranteed due to the weak convexity of the deep prior. 
We have also conducted extensive experiments on image restoration, evaluating the performance of our proposed algorithms in Rician noise removal and phase retrieval. Compared to state-of-the-art methods, our results were superior or comparable both visually and quantitatively, demonstrating the effectiveness of our method and its accelerated practical convergence.

For future research, we will consider variants of the proposed method, including different acceleration techniques such as dynamically adapting parameter choices based on noise level estimation at each step, and implementing backtracking line search with various stopping criteria to enhance descent. Additionally, we will explore generalizing our proposed Bregman denoiser to generic kernel functions.

\acks{We would like to express our gratitude to the authors of \cite{wu2022efficient,wei2023nonconvex} for graciously providing us with the source code. This work was supported by the National Natural Science Foundation of China grants 12471291, 12001286, and
the China Postdoctoral Science Foundation grants 2022M711672, NSFC/RGC N CUHK 415/19, ITF ITS/173/22FP,
RGC 14300219, 14302920, 14301121, and CUHK Direct Grant for Research.}



\appendix
\section{Preliminaries of Nonconvex Nonsmooth Optimization}
\label{app:preliminaries}

\subsection{Subdifferentials} 
\begin{definition}{\rm ({\bf Subdifferentials}) \citep{attouch2013convergence, bolte2014proximal}} \label{def:subdifferentials} 
For a proper and lower semicontinuous function  $f:\mathbb {R}^{n}\rightarrow (-\infty,+\infty]$,
\begin{itemize}
\item[(i)] given ${\bf x}\in {\rm dom}(f)$, the Fr\'{e}chet subdifferential of $f$ at ${\bf x}$, expressed as $\widehat{\partial}f({\bf x})$, is the set of all vectors ${\bf u}\in \mathbb{R}^n$ satisfying\vspace{-0.05in}
$$\liminf_{{\bf y}\neq {\bf x}, {\bf y}\rightarrow {\bf x}}\frac{f({\bf y})-f({\bf x})-\langle {\bf u},{\bf y}-{\bf x}\rangle}{\|{\bf y}-{\bf x}\|}\geq0,\vspace{-0.05in}$$
and we set $\widehat{\partial}f({\bf x}) = \emptyset$ when ${\bf x}\notin {\rm dom}(f)$.
\vspace{0.2cm}
\item[(ii)] (limiting-)subdifferential of $f$ at ${\bf x}$, written by $\partial f({\bf x})$, is defined by 
\begin{equation}\label{pf}
\partial f({\bf x}):=\{{\bf u}\in\mathbb{R}^n\; | \; \exists ~ {\bf x}^k\rightarrow {\bf x}, ~{\rm s.t.}~f({\bf x}^k)\rightarrow f({\bf x})
 ~{\rm and}~ \widehat{\partial}f({\bf x}^k) \ni {\bf u}^k \rightarrow {\bf u} \}.  \end{equation}

\item[(iii)] a point ${\bf x}^*$ is called (limiting-)critical point or stationary point of $f$ if it satisfies $0\in\partial f({\bf x}^*)$, and the set of critical points of $f$ is denoted by ${\rm crit} f$.
\end{itemize}
\end{definition}

Note that Definition~\ref{def:subdifferentials} implies that the property $\widehat{\partial}f({\bf x})\subseteq \partial f({\bf x})$ immediately holds, and $\widehat{\partial}f({\bf x})$ is closed and convex, meanwhile $\partial f({\bf x})$ is closed \cite[Theorem 8.6]{rockafellar2009variational}. Also, the subdifferential \eqref{pf} reduces to the gradient of $f$ denoted by $\nabla f$ if $f$ is continuously differentiable. Moreover, as mentioned in \cite{rockafellar2009variational}, if $g$ is a continuously differentiable function, it holds that $\partial(f+g)=\partial f+\nabla g$.

\subsection{Kurdyka-{\L}ojasiewicz (KL) Property} 

\begin{definition}{\rm ({\bf KL property and KL function})}  \citep{attouch2010proximal, bolte2014proximal}) \label{def:KL_property}
Let $f:\mathbb{R}^{n}\rightarrow (-\infty,+\infty]$ be a proper and lower semicontinuous function.
\begin{itemize}
\item[$(a)$]$f$ is said to have KL property at ${\bf x}^*\in{\rm dom}(\partial f)$ if there exist $\tau\in(0,+\infty]$, a neighborhood $U$ of ${\bf x}^*$ and a continuous and concave function $\varrho:[0,\tau)\rightarrow \mathbb{R}_+$ such that
\begin{itemize}
\item[\rm(i)] $\varrho(0)=0$ and $\varrho$ is continuously differentiable on $(0,\tau)$ with $\varrho'>0$;

\item[\rm(ii)] $\forall {\bf x}\in U\cap\{{\bf z} \in \mathbb{R}^n\; | \; f({\bf x}^*) < f({\bf z}) < f({\bf x}^*)+\tau\}$, the following KL inequality holds: 
\begin{equation} \nonumber
\varrho'(f({\bf x})-f({\bf x}^*)) \cdot {\rm dist}(0,\partial f({\bf x})) \geq 1. 
\end{equation}
\end{itemize}

\item[$(b)$] If $f$ satisfies the KL property at each point of dom$(\partial f)$, then $f$ is called a KL function.
\end{itemize}
\end{definition}

Let $\Phi_\tau$ denote the set of function $\varrho$ that satisfies the condition in Definition~\ref{def:KL_property}(a). Then, we can establish a uniformized KL property established in \cite{bolte2014proximal}.

\begin{lemma}\label{thm:uniformized_KL} {\rm ({\bf Uniformized KL property}) \citep{bolte2014proximal}} 
    Let  $f:\mathbb{R}^{n}\rightarrow(-\infty,+\infty]$ be a proper and lower semicontinuous function and $\Gamma$ be a compact set. Assume that $f$ is a constant on $\Gamma$ and satisfies the KL property at each point of $\Gamma$. Then, there exist $\vartheta>0,~\tau>0$ and $\varrho\in \Phi_{\tau}$ such that  
    \begin{equation} \nonumber
        \varrho'(f({\bf x})-f(\bar {\bf x})) \cdot {\rm dist}(0,\partial f({\bf x}))\geq 1, 
    \end{equation}
     for all $\bar {\bf x}\in \Gamma$ and each ${\bf x}$ satisfying
    ${\rm dist}({\bf x},\Gamma)<\vartheta$ and $f(\bar {\bf x}) < f({\bf x}) < f(\bar {\bf x})+\tau.$
\end{lemma}

\section{Missing Proofs in Section~\ref{sec:alg}} \label{app:proof}

\subsection{Proof of Lemma~\ref{lem:descending_property_auxiliary_function} (Sufficient decrease property)}\label{proof:descending_property_auxiliary_function}

By first-order optimality condition for Subproblem~\eqref{subproblem:main_subproblem_of_iBPDCA}, we have
\begin{equation} \label{eq:first_order_optimality_descend}
        0\in\partial g({\bf x}^{k+1}) + \nabla f_1({\bf y}^k) -\xi^k +\frac{1}{\lambda} \left(\nabla h({\bf x}^{k+1}) - \nabla h({\bf y}^k)\right).
\end{equation}
Combining $\eta$-weak convexity of $g$ and \eqref{eq:first_order_optimality_descend}, we obtain
\begin{align*}
    g({\bf x}^k) - g({\bf x}^{k+1}) &\geq \left \langle - \nabla f_1({\bf y}^k) +\xi^k -\frac{1}{\lambda} \left(\nabla h({\bf x}^{k+1})  - \nabla h({\bf y}^k)\right), {\bf x}^k - {\bf x}^{k+1} \right \rangle \\ 
    &\qquad -
    \frac{\eta}{2} \Vert {\bf x}^k - {\bf x}^{k+1} \Vert ^2.
\end{align*}
Then, it follows from Lemma~\ref{lem:three-point-identity} (Three-point identity) that
\begin{align*}
    \left \langle \nabla h({\bf x}^{k+1}) - \nabla h({\bf y}^k), {\bf x}^k - {\bf x}^{k+1} \right \rangle = D_h({\bf x}^k,{\bf y}^k) - D_h({\bf x}^k,{\bf x}^{k+1}) - D_h({\bf x}^{k+1},{\bf y}^k),
\end{align*}
which leads to
\begin{equation} \label{eq:g_inequality_descend}
    \begin{aligned}
    g({\bf x}^k) - g({\bf x}^{k+1}) &\geq \left \langle - \nabla f_1({\bf y}^k) +\xi^k , {\bf x}^k - {\bf x}^{k+1} \right \rangle -\frac{\eta}{2} \Vert {\bf x}^k - {\bf x}^{k+1} \Vert ^2\\
    & \qquad - \frac{1}{\lambda} \left(D_h({\bf x}^k,{\bf y}^k) - D_h({\bf x}^k,{\bf x}^{k+1}) - D_h({\bf x}^{k+1},{\bf y}^k)\right).
    \end{aligned}
\end{equation}
By the definition of subgradient of $f_2$ (\ie\ $f_2({\bf x}^{k+1}) -f_2({\bf x}^k) \geq \left \langle \xi^k ,{\bf x}^{k+1} -{\bf x}^k\right \rangle$, where $\xi^k \in \partial f_2({\bf x}^k)$), we can reformualte \eqref{eq:g_inequality_descend} as 
\begin{equation} \label{eq:-f2+g_descend}
\begin{aligned}
    & - f_2({\bf x}^k) + g({\bf x}^k) - \left( - f_2({\bf x}^{k+1}) + g({\bf x}^{k+1})\right) \\
     &\geq -\left \langle  \nabla f_1({\bf y}^k), {\bf x}^k - {\bf x}^{k+1} \right \rangle -\frac{\eta}{2} \Vert {\bf x}^k - {\bf x}^{k+1} \Vert ^2 \\
     & \qquad - \frac{1}{\lambda} \left(D_h({\bf x}^k,{\bf y}^k) - D_h({\bf x}^k,{\bf x}^{k+1}) - D_h({\bf x}^{k+1},{\bf y}^k)\right).
\end{aligned}
\end{equation}
On the other hand, it follows from the convexity of $f_1$ and Definition~\ref{lem:full_extended_descent_lemma} (Restricted $L$-smooth adaptable on $\mathcal{X}$) that
\begin{equation} \label{eq:f_1_descend}
\begin{aligned} 
    &f_1({\bf x}^k) - f_1({\bf x}^{k+1}) - \left \langle \nabla f_1({\bf y}^k), {\bf x}^k - {\bf x}^{k+1} \right \rangle \\
    &\quad = f_1({\bf x}^k) -f_1({\bf y}^k) - \left \langle \nabla f_1({\bf y}^k), {\bf x}^k - {\bf y}^k\right \rangle -f_1({\bf x}^{k+1}) +f_1({\bf y}^k) +\left \langle \nabla f_1({\bf y}^k), {\bf x}^{k+1} -{\bf y}^k \right \rangle \\
    &\quad \geq -LD_h({\bf x}^{k+1},{\bf y}^k).
\end{aligned}
\end{equation}
Combining \eqref{eq:-f2+g_descend} and \eqref{eq:f_1_descend}, we obtain
\begin{equation} \label{eq:obj_step_descend}
\begin{aligned}
    &\Psi({\bf x}^k) - \Psi ({\bf x}^{k+1}) \\
    &\quad \geq -LD_h({\bf x}^{k+1},{\bf y}^k) -\frac{\eta}{2} \Vert {\bf x}^k - {\bf x}^{k+1} \Vert ^2 - \frac{1}{\lambda}\left(D_h({\bf x}^k,{\bf y}^k) - D_h({\bf x}^k,{\bf x}^{k+1}) - D_h({\bf x}^{k+1},{\bf y}^k)\right) \\
    &\quad \geq \frac{1}{\lambda} D_h({\bf x}^k,{\bf x}^{k+1}) + \left(\frac{1}{\lambda} -L \right) D_h({\bf x}^{k+1},{\bf y}^k) -\frac{1}{\lambda} D_h({\bf x}^k,{\bf y}^k) -\frac{\eta}{2} \Vert {\bf x}^k - {\bf x}^{k+1} \Vert ^2.
\end{aligned}
\end{equation}
Involving Definition~\ref{def:breg_dist} and the fact that $h$ is $\kappa$-strongly convex, we can deduce an lower bound for $-\Vert {\bf x}^k - {\bf x}^{k+1} \Vert^2$, which is
\begin{equation} \label{eq:Dh_descend}
\begin{aligned}
    - D_h({\bf x}^k,{\bf x}^{k+1}) &= -\left(h({\bf x}^k) - h({\bf x}^{k+1}) -\left \langle \nabla h({\bf x}^{k+1}) ,{\bf x}^k - {\bf x}^{k+1} \right \rangle \right)
    \leq -\frac{\kappa}{2} \Vert {\bf x}^k -{\bf x}^{k+1} \Vert^2.
\end{aligned}
\end{equation}
Plugging \eqref{eq:Dh_descend} into \eqref{eq:obj_step_descend}, we have
\begin{align*}
    \Psi({\bf x}^k) - \Psi ({\bf x}^{k+1}) \geq \left(\frac{1}{\lambda} - \frac{\eta}{\kappa}\right) D_h({\bf x}^k,{\bf x}^{k+1}) + \left(\frac{1}{\lambda}-L \right) D_h({\bf x}^{k+1},{\bf y}^k) -\frac{1}{\lambda} D_h({\bf x}^k,{\bf y}^k).
\end{align*}
Involving the definition of $H_\delta$ in \eqref{eq:definition_Hk} , we have
\begin{align*}
    H_\delta ({\bf x}^k,{\bf x}^{k-1}) &\geq H_\delta({\bf x}^{k+1},{\bf x}^k) + \left(\frac{1}{\lambda} -\frac{\eta}{\kappa}-\delta \right)D_h({\bf x}^k,{\bf x}^{k+1}) \\
    &\quad + \left(\frac{1}{\lambda} -L \right) D_h({\bf x}^{k+1},{\bf y}^k) -\frac{1}{\lambda} D_h({\bf x}^k,{\bf y}^k) + \delta D_h({\bf x}^{k-1},{\bf x}^k).
\end{align*}        
Utilizing \eqref{bregman_distance_condition_on_extrapolation}, we further obtain
\begin{align*}
    H_\delta ({\bf x}^k,{\bf x}^{k-1}) &\geq H_\delta({\bf x}^{k+1},{\bf x}^k) + \left(\frac{1}{\lambda} -\frac{\eta}{\kappa}-\delta\right)D_h({\bf x}^k,{\bf x}^{k+1})+ \left(\frac{1}{\lambda} -L \right) D_h({\bf x}^{k+1},{\bf y}^k) \\
    &\quad + \epsilon D_h({\bf x}^{k-1},{\bf x}^k).
\end{align*}
Since $\frac{1}{\lambda} >  \max \left\{\delta+\frac{\eta}{\kappa}, L \right \}$ and $\epsilon>0$, $\{H_\delta\}_{k=0}^\infty $ is non-increasing. This completes the proof.

\subsection{Proof of Proposition~\ref{prop:convergence_property_of_Dh}} \label{proof:convergence_property_of_Dh}
(i) Rearranging \eqref{inequality_for_auxiliary_function}, we have
    \begin{align*}
        H_\delta ({\bf x}^k,{\bf x}^{k-1})- H_\delta({\bf x}^{k+1},{\bf x}^k) &\geq \left[\frac{1}{\lambda} - \left(\frac{\eta}{\kappa}+\delta \right)\right]D_h({\bf x}^k,{\bf x}^{k+1}) \\
        &\qquad+ \left(\frac{1}{\lambda} -L \right) D_h({\bf x}^{k+1},{\bf y}^k)
        + \epsilon D_h({\bf x}^{k-1},{\bf x}^k) \\
        &\geq \left[\frac{1}{\lambda} - \left(\frac{\eta}{\kappa}+\delta \right)\right] D_h({\bf x}^k,{\bf x}^{k+1}) + \epsilon D_h({\bf x}^{k-1},{\bf x}^k),
    \end{align*}
    where the last inequality comes from 
    $\left( \frac{1}{\lambda} - L\right) D_h({\bf x}^{k+1}, {\bf y}^k)\geq 0$. Multiplying both sides by $\lambda$, and since $\left(1 -\lambda \left(\frac{\eta}{\kappa}+\delta)\right)D_h({\bf x}^k,{\bf x}^{k+1}\right) \geq 0$, we have
    \begin{equation}\label{inequ_}
    \begin{aligned}
        \lambda[H_\delta ({\bf x}^k,{\bf x}^{k-1})- H_\delta({\bf x}^{k+1},{\bf x}^k)] &\geq \left[1 -\lambda \left(\frac{\eta}{\kappa}+\delta\right)\right]D_h({\bf x}^k,{\bf x}^{k+1}) + \epsilon \lambda D_h({\bf x}^{k-1},{\bf x}^k)\\
        &\geq \epsilon \lambda  D_h({\bf x}^{k-1},{\bf x}^k), 
    \end{aligned}
    \end{equation}
    %
    From Equation~\eqref{v(p)}, we have $\Psi^* > -\infty$. Summing \eqref{inequ_} from $k=0$ to $n$, we have
    \begin{equation}\label{ineqDh}
        \begin{aligned}
        \epsilon \lambda \sum_{k=0}^n  D_h({\bf x}^{k-1},{\bf x}^k) &\leq \lambda \sum_{k=0}^n \left(H_\delta ({\bf x}^k,{\bf x}^{k-1})- H_\delta({\bf x}^{k+1},{\bf x}^k)\right). \\
        \text{Dividing both sides by $\epsilon\lambda,$}\\
        \sum_{k=0}^n  D_h({\bf x}^{k-1},{\bf x}^k) &\leq \frac{1}{\epsilon} \left[\Psi({\bf x}^0) +D_h({\bf x}^{-1},{\bf x}^0)- \left(\Psi({\bf x}^n)+D_h({\bf x}^n,{\bf x}^{n+1}) \right)\right]  \\
        &\leq \frac{1}{\epsilon} (\Psi({\bf x}^0) - \Psi({\bf x}^n)) 
        \leq \frac{1}{\epsilon} (\Psi({\bf x}^0) - \Psi^*),
    \end{aligned}
    \end{equation}
    where the second inequality arises from the initialization condition ${\bf x}^{-1}={\bf x}^0 \in {\rm intdom}(h)$, implying $D_h({\bf x}^{-1},{\bf x}^0)=0$, and the convexity of function $h$ ensures $D_h({\bf x}^n,{\bf x}^{n+1})\geq0,\forall n \in \mathbb{N}$. The last inequality follows from Definition~\ref{v(p)} that $ \Psi^* = \inf\{\Psi({\bf x}) | {\bf x} \in \mathcal{X}\} \leq \Psi({\bf x}^n), \forall n \in\mathbb{N}$.

   With Assumption~\ref{asm:proximal_mapping_subset_C}, we have ${\bf x}^{n+1} \in {\rm intdom}(h)$, leads us to ${\bf x}^{k} \in {\rm intdom}(h), \forall k \in \mathbb{N}$ by induction. We now can take the limit as $n \rightarrow \infty$, and establish the first part of the assertion. The second part of the assertion is as follows.

(ii) Following from \eqref{ineqDh}, we have 
\[n \min_{1\leq k \leq n}D_h({\bf x}^{k-1},{\bf x}^k)\leq \sum^n_{k=1} D_h({\bf x}^{k-1},{\bf x}^k) \leq \frac{1}{\epsilon} \left(\Psi({\bf x}^0) - \Psi^* \right) ,\]
dividing both sides by $n$ yields the desired outcome. This completes the proof.

\subsection{Proof of Theorem~\ref{thm:subsquential_convergence} ({\bf Subsequential convergence of iBPDCA})} \label{proof:thm:subsquential_convergence}
(i) From Proposition~\ref{prop:convergence_property_of_Dh}, the sequence $\{H_\delta({\bf x}^k,{\bf x}^{k-1})\}^\infty_{k=0}$ is non-increasing. Consequently, $H_\delta({\bf x}^k,{\bf x}^{k-1})) \leq H_\delta({\bf x}^0,{\bf x}^{-1}), \forall k \in \mathbb{N}$. Also, $D_h({\bf x}^{-1}, {\bf x}^0)=0$ since 
we set ${\bf x}^{-1}={\bf x}^0$ in the proposed Algorithm~\ref{alg:iBPDCA}. We can then prove the boundedness of objective function $\Psi({\bf x}^k)$ accordingly:
\[\Psi({\bf x}^k)\leq \Psi({\bf x}^k) + \delta D_h({\bf x}^{k-1},{\bf x}^k) = H_\delta({\bf x}^k,{\bf x}^{k-1}) \leq H_\delta({\bf x}^0,{\bf x}^{-1}) =\Psi({\bf x}^0).\]
The boundedness of $\{{\bf x}^k\}^\infty_{k=0}$ automatically fulfills due to the level boundedness of $\Psi$ by Assumption~\ref{asm:assumption_on_h}(iv).

(ii) Since $f_2$ is convex on $\mathbb{R}^n$, hence continuous.  By the boundedness of $\{{\bf x}^k\}^\infty_{k=0}$ from (i), $\{\xi^k\}^\infty_{k=0}$ is bounded as well.

(iii) Since $h$ is convex, we have $D_h({\bf x},{\bf y}) \geq 0, \forall {\bf x} \in  {\rm dom}(h), ~ {\bf y} \in  {\rm int}{\rm dom}(h)$. Taking advantage of $\kappa$-strong convexity of the function $h$ and $\frac{1}{\lambda} >  \max \left \{\delta + \frac{\eta}{\kappa}, L \right\}$. We obtain the following by rearranging \eqref{inequality_for_auxiliary_function} in Lemma~\ref{lem:descending_property_auxiliary_function} ({Sufficient decrease property of $H_\delta$}), we have 
\begin{equation} \label{eq:H_delta_relation_to_l2}
    \begin{aligned}
    H_\delta ({\bf x}^k,{\bf x}^{k-1}) -  H_\delta({\bf x}^{k+1},{\bf x}^k) &\geq \left(\frac{1}{\lambda} -\frac{\eta}{\kappa}-\delta \right)D_h({\bf x}^k,{\bf x}^{k+1}) \\
    &\quad + \left(\frac{1}{\lambda} -L\right) D_h({\bf x}^{k+1},{\bf y}^k) + \epsilon D_h({\bf x}^{k-1},{\bf x}^k) \\
    &\geq \left(\frac{1}{\lambda} -L\right) D_h({\bf x}^{k+1},{\bf y}^k) \\
    &\geq \frac{\kappa}{2} \left(\frac{1}{\lambda} -L\right)  \Vert {\bf x}^{k+1} -{\bf y}^k \Vert^2 \\
    &\geq \frac{\kappa}{2}\left(\frac{1}{\lambda} -L \right)\left(\Vert {\bf x}^{k+1} - {\bf x}^k\Vert^2 - \beta_k^2\Vert {\bf x}^k -{\bf x}^{k-1}\Vert^2\right),
    \end{aligned}
\end{equation}  
where the last equality comes from extrapolation step \eqref{extrapolation_step} and the reverse triangle inequality.

By summing \eqref{eq:H_delta_relation_to_l2} from $k=0$ to $\infty$, we obtain
\begin{align*}
    &\quad \frac{\kappa}{2}\left(\frac{1}{\lambda} -L\right) \sum_{k=0}^\infty (1-\beta_{k+1}^2)\Vert {\bf x}^{k+1} - {\bf x}^k\Vert^2 - \beta_0^2 \Vert {\bf x}^0 - {\bf x}^{-1}\Vert^2\\
    & =\frac{\kappa}{2} \left(\frac{1}{\lambda} -L \right) \sum_{k=0}^\infty (1-\beta_{k+1}^2)\Vert {\bf x}^{k+1} - {\bf x}^k\Vert^2 \\
    &\leq H_\delta({\bf x}^0,{\bf x}^{-1}) - \liminf_{n \to \infty} H_\delta({\bf x}^{n+1},{\bf x}^{n})\\
    &=\Psi({\bf x}^0) - \liminf_{n \to \infty} \left(\Psi({\bf x}^{n+1}) +\delta D_h({\bf x}^n,{\bf x}^{n+1}) \right)\\
    &\leq \Psi({\bf x}^0) - \Psi^*  < \infty,
\end{align*}
where the last inequality comes from Definition~\ref{v(p)}.
This shows that $\lim_{k\rightarrow\infty}\Vert {\bf x}^{k+1}-{\bf x}^k \Vert =0$ since $\frac{1}{\lambda} >  \max\{\delta+\frac{\eta}{\kappa}, L\}$ and $1-\beta_{k+1}^2 >0$.

(iv) Let ${\bf x}^*$ be an accumulation point of $\left\{{\bf x}^k \right\}_{k=0}^\infty$, and $\left\{{\bf x}^{k_j}\right\}_{j=0}^\infty$ be its subsequence of $\left\{{\bf x}^k\right\}_{k=0}^\infty$ such that $\lim_{j\rightarrow\infty}{\bf x}^{k_j}={\bf x}^*$. By the first-order optimality condition of Subproblem~\eqref{subproblem:main_subproblem_of_iBPDCA}, we have 
\begin{align*}
    0 &\in \partial g({\bf x}^{k_j+1}) + \nabla f_1({\bf y}^{k_j}) -\xi^{k_j} +\frac{1}{\lambda} \left(\nabla h({\bf x}^{k_j+1}) - \nabla h({\bf y}^{k_j}) \right).
\end{align*}
Rearranging the above inequality and adding $\nabla f_1({\bf x}^{k_j+1})$ on both sides, we obtain
\begin{equation} \label{thm:subsquential_convergence_subproblem}
    \begin{aligned}
        &\nabla f_1({\bf x}^{k_j+1}) -\nabla f_1({\bf y}^{k_j}) +\xi^{k_j} + \frac{1}{\lambda} \left( \nabla h({\bf y}^{k_j}) - \nabla h({\bf x}^{k_j+1})\right) \\
        &\qquad \in \partial g({\bf x}^{k_j+1}) + \nabla f_1({\bf x}^{k_j+1}).
    \end{aligned}
\end{equation}

As the sequence $\{{\bf x}^k\}^\infty_{k=0}$ is bounded according to (i), it follows that its subsequence $\{{\bf x}^{k_j}\}^\infty_{j=0}$ is also bounded. Furthermore, considering the Lipschitzness of $\nabla f_1$ and $\nabla h$ as stated in Assumption~\ref{asm:assumption_on_h}, we can conclude that there exists a constant $C_0 > 0$ such that
\begin{align*}
    \left\Vert \nabla f_1({\bf x}^{k_j+1}) -\nabla f_1({\bf y}^{k_j})  + \frac{1}{\lambda} \left(\nabla h({\bf y}^{k_j}) - \nabla h({\bf x}^{k_j+1}) \right) \right\Vert \leq C_0 \left\Vert {\bf x}^{k_j+1} -{\bf y}^{k_j}\right\Vert.
\end{align*}
As $j\rightarrow \infty$, we can to establish that $\Vert {\bf x}^{k_j+1} - {\bf x}^{k_j} \Vert \rightarrow 0$, and $\Vert {\bf x}^{k_j} - {\bf x}^{k_j-1} \Vert \rightarrow 0$ from (iii). Then, we have $\Vert {\bf x}^{k_j+1} -{\bf y}^{k_j}\Vert \leq \Vert {\bf x}^{k_j+1} - {\bf x}^{k_j} \Vert + \beta_k^2 \Vert {\bf x}^{k_j} - {\bf x}^{k_j-1} \Vert \rightarrow 0$ with \eqref{extrapolation_step} and triangle inequality. Hence
\begin{equation} \label{limiting_subproblem}
    \nabla f_1({\bf x}^{k_j+1}) -\nabla f_1({\bf y}^{k_j})  + \frac{1}{\lambda} \left(\nabla h({\bf y}^{k_j}) - \nabla h({\bf x}^{k_j+1}) \right)  \rightarrow 0.
\end{equation}

Similarly, from (ii), we know that $\{{\bf x}^{k_j}\}$ and $\{\xi^{k_j}\}$ are bounded. 
Without loss of generality we can assume that $\lim_{j\rightarrow \infty} \xi^{k_j} = \xi^*$ exists, which belongs to $\partial f_2({\bf x}^*)$ because of the closedness of $\partial f_2$. Taking the limit as $j \rightarrow \infty$ to \eqref{thm:subsquential_convergence_subproblem} and utilizing \eqref{limiting_subproblem},  with the continuity of $g$ and $\nabla f_1$, we obtain
$    \xi^* \in \partial g({\bf x}^*) + \nabla f_1({\bf x}^*)$
and 
$    0 \in \partial g({\bf x}^*) + \nabla f_1({\bf x}^*) -\partial f_2({\bf x}^*).$
Therefore, ${\bf x}^*$ is a limiting-critical point of Problem \eqref{eq:DCA_model}.
This completes the proof. 

\subsection{Proof of Proposition~\ref{prop:proporty_accumulation_point}}\label{proof:proporty_accumulation_point}
(i) By Proposition~\ref{prop:convergence_property_of_Dh}, we have $\lim_{k\rightarrow\infty}D_h({\bf x}^{k-1},{\bf x}^k)=0$. With Equation~\eqref{v(p)} and Lemma~\ref{lem:descending_property_auxiliary_function} (Sufficient decrease property of $H_\delta$), we know that the sequence $\{H_\delta({\bf x}^k,{\bf x}^{k-1})\}^\infty_{k=0}$ is bounded from below and non-increasing as well. Therefore, we have
\begin{align*}
\zeta := \lim_{k\rightarrow \infty} H_\delta({\bf x}^k,{\bf x}^{k-1})= \lim_{k\rightarrow \infty}  \Psi({\bf x}^k)+ D_h({\bf x}^{k-1},{\bf x}^k)=\lim_{k\rightarrow\infty}\Psi({\bf x}^k).
\end{align*}
Hence, $\zeta : =  \lim_{k\rightarrow \infty}H_\delta({\bf x}^k,{\bf x}^{k-1}) = \lim_{k\rightarrow\infty} \Psi({\bf x}^k) $ exists.

(ii) From Theorem~\ref{thm:subsquential_convergence} (Subsequential convergence of iBPDCA), we have $\emptyset \neq \Omega \subseteq {\rm crit} \Psi$, where ${\rm crit} \Psi$ is the set of critical points of $\Psi$. Take any ${\bf x}^* \in \Omega$, by the definition of accumulation point, there exists a convergence subsequence $\left\{{\bf x}^{k_j}\right\}_{j=0}^\infty$ such that $\lim_{j\rightarrow\infty}{\bf x}^{k_j}={\bf x}^*$. From the first-order optimality condition of Subproblem~\eqref{subproblem:main_subproblem_of_iBPDCA}, we have 
\begin{align*}
    &g({\bf x}^{k+1} ) +\left \langle \nabla f_1({\bf y}^k)-\xi^k, {\bf x}^{k+1} -{\bf y}^k \right \rangle +\frac{1}{\lambda} D_h({\bf x}^{k+1} ,{\bf y}^k) \\
    &\qquad \leq g({\bf x}^*) +\left \langle \nabla f_1({\bf y}^k)-\xi^k, {\bf x}^{*} -{\bf y}^k \right \rangle +\frac{1}{\lambda} D_h({\bf x}^{*} ,{\bf y}^k).
\end{align*}
Rearranging the above,
\begin{align*}
    g({\bf x}^{k+1}) &\leq g({\bf x}^*) +\left \langle \nabla f_1({\bf y}^k)-\xi^k, {\bf x}^{*} -{\bf y}^k \right \rangle +\frac{1}{\lambda}D_h({\bf x}^{*} ,{\bf y}^k) \\
    &\quad - \left \langle \nabla f_1({\bf y}^k)-\xi^k, {\bf x}^{k+1} -{\bf y}^k \right \rangle - \frac{1}{\lambda} D_h({\bf x}^{k+1} ,{\bf y}^k) \\
    &= g({\bf x}^*) +\left \langle \nabla f_1({\bf y}^k)-\xi^k, {\bf x}^{*} -{\bf x}^{k+1} \right \rangle +\frac{1}{\lambda}D_h({\bf x}^{*} ,{\bf y}^k) - \frac{1}{\lambda} D_h({\bf x}^{k+1} ,{\bf y}^k).
\end{align*}
Adding $f_1({\bf x}^{k+1})$ to both sides, we have
\begin{equation} \label{accumlation_point_subproblem}
    \begin{aligned}
        g({\bf x}^{k+1}) +f_1({\bf x}^{k+1}) &\leq g({\bf x}^*)+ f_1({\bf x}^{k+1})  +\left \langle \nabla f_1({\bf y}^k)-\xi^k, {\bf x}^{*} -{\bf x}^{k+1} \right \rangle \\
        & \qquad +\frac{1}{\lambda}D_h({\bf x}^{*} ,{\bf y}^k) - \frac{1}{\lambda} D_h({\bf x}^{k+1} ,{\bf y}^k) \\
        &\leq g({\bf x}^*)+ f_1({\bf x}^{*})  +\left \langle \nabla f_1({\bf y}^k)-\xi^k, {\bf x}^{*} -{\bf x}^{k+1} \right \rangle \\
        & \qquad +\frac{1}{\lambda}D_h({\bf x}^{*} ,{\bf y}^k) - \frac{1}{\lambda} D_h({\bf x}^{k+1} ,{\bf y}^k) -\left \langle \nabla f_1({\bf x}^{k+1}),{\bf x}^* - {\bf x}^{k+1} \right \rangle\\ 
        &\leq g({\bf x}^*)+ f_1({\bf x}^{*})  +\left \langle \nabla f_1({\bf y}^k)-\xi^k, {\bf x}^{*} -{\bf x}^{k+1} \right \rangle \\
        & \qquad +\frac{1}{\lambda}D_h({\bf x}^{*} ,{\bf y}^k) + \frac{1}{\lambda} D_h({\bf y}^k,{\bf x}^{*}) -\left \langle \nabla f_1({\bf x}^{k+1}),{\bf x}^* - {\bf x}^{k+1} \right \rangle,
    \end{aligned}
\end{equation}
where the second inequality follows from the convexity of $f_1$, and the third inequality comes from the convexity of $h$, which is $D_h({\bf x},{\bf y})\geq0, \forall {\bf x}\in  {\rm dom}(h), {\bf y}\in  {\rm int}{\rm dom}(h)$.

Based on Assumption~\ref{asm:assumption_on_h}(iii) that $\nabla h$ is Lipschitz continuous, we can establish 
\begin{align*}
    &\lim_{j\rightarrow\infty} \left (D_h({\bf x}^{*} ,{\bf y}^{k_j}) + D_h({\bf y}^{k_j},{\bf x}^{*}) \right) 
    =\lim_{j\rightarrow\infty} \left \langle \nabla h({\bf x}^*) -\nabla h({\bf y}^{k_j}), {\bf x}^* - {\bf y}^{k_j}\right \rangle \\
    \leq &\lim_{j\rightarrow\infty} \Vert \nabla h({\bf x}^*) -\nabla h({\bf y}^{k_j}) \Vert \Vert  {\bf x}^* - {\bf y}^{k_j} \Vert 
    \leq \lim_{j\rightarrow\infty} L_h \Vert  {\bf x}^* - {\bf y}^{k_j} \Vert^2 =0,
\end{align*}
where the first equality is derived from 
Lemma~\ref{lem:three-point-identity} (Three-point identity),
and the first inequality is deduced using the Cauchy-Schwarz inequality. 
From Theorem~\ref{thm:subsquential_convergence}(ii),
$\left\{\xi^{k_j}\right\}_{j=0}^\infty$ is bounded. By substituting $k=k_j$ back into \eqref{accumlation_point_subproblem} and taking $j \rightarrow \infty$, we can deduce 
\begin{align*}
    \zeta &= \lim_{j\rightarrow\infty} f_1({\bf x}^{k_j+1}) - f_2({\bf x}^{k_j+1}) + g({\bf x}^{k_j+1}) \\
    & \leq \lim_{j\rightarrow\infty} g({\bf x}^*)+ f_1({\bf x}^{*}) +\left \langle \nabla f_1({\bf y}^{k_j})-\xi^{k_j}, {\bf x}^{*} -{\bf x}^{k_j+1} \right \rangle \\
        & \qquad +\frac{1}{\lambda}D_h({\bf x}^{*} ,{\bf y}^{k_j}) + \frac{1}{\lambda} D_h({\bf y}^{k_j},{\bf x}^{*}) -\left \langle \nabla f_1({\bf x}^{k_j+1}),{\bf x}^* - {\bf x}^{k_j+1} \right \rangle -f_2({\bf x}^{k_j+1})\\
    & \leq \limsup_{j\rightarrow\infty} f_1({\bf x}^{*}) -f_2({\bf x}^{k_j+1}) +g({\bf x}^*) 
     \leq \Psi({\bf x}^*),
\end{align*}
where the last line follows from the continuity of $-f_2$.
From the lower semicontinuity of $\Psi$, we have 
\begin{align*}
    \Psi({\bf x}^*) \leq \liminf_{j\rightarrow\infty} \Psi({\bf x}^{k_j+1}) = \lim_{j\rightarrow \infty} \Psi({\bf x}^{k_j+1}) = \zeta.
\end{align*}
 Therefore, we can deduce that $\Psi({\bf x}^*)=\lim_{j\rightarrow \infty} \Psi ({\bf x}^{k_j +1})=\zeta$, leading to the conclusion that $\Psi \equiv \zeta$ on $\Omega$, since the selection of ${\bf x}^* \in \Omega$ is arbitrary. This completes the proof. 

\subsection{Proof of Theorem~\ref{thm:global_convergence} (Global convergence of iBPDCA)}\label{proof:global_convergence}

(i) From Theorem~\ref{thm:subsquential_convergence}(i), we see that $\{{\bf x}^k\}_{k=0}^\infty$ is bounded. With the definition of $\Omega$ from Proposition~\ref{prop:proporty_accumulation_point}, this implies that $\lim_{k\rightarrow \infty} {\rm dist}({\bf x}^k, \Omega) =0$. Recall from Theorem~\ref{thm:subsquential_convergence}(iv) that $\Omega \subseteq {\rm crit} \Psi$. Thus, for any $\mu>0$, there exits $K_0>0$ such that ${\rm dist}({\bf x}^k,\Omega)< \mu$ and ${\bf x}^k \in \mathcal{N}_0$ whenever $k\geq K_0$, where $\mathcal{N}_0$ is the open set as defined in Assumption~\ref{asm:f2_local_continuity}. Moreover, since $\Omega$ is compact due to the boundedness of $\{{\bf x}^k\}_{k=0}^\infty$, by shrinking $\mu$ is necessary, we may assume without the loss of generality that $\nabla f_2$ is globally Lipschitz continuous on the bounded set $\mathcal{N}:=\{{\bf x}\in \mathcal{N}_0\mid {\rm dist}({\bf x},\Omega) < \mu \}$.

Next, we consider the subdifferential of the auxiliary function $H_\delta$ in \eqref{eq:definition_Hk} at the point $({\bf x}^k,{\bf x}^{k-1})$ for $k \geq K_0$, we have
\begin{equation} \nonumber
    \begin{aligned}
    \partial H_\delta (\mathbf{x}^k, \mathbf{x}^{k-1}) &= \left ( \partial_x H_\delta (\mathbf{x}^k, \mathbf{x}^{k-1}) , \partial_y H_\delta (\mathbf{x}^k, \mathbf{x}^{k-1}) \right).
    \end{aligned}
\end{equation}
Since $f_2$ is continuously differentiable in $\mathcal{N}$ and that ${\bf x}^k \in \mathcal{N}$ for $k\geq K_0$, we have 
\begin{equation} \label{eq:partial_H_delta}
    \begin{aligned}
        \partial_x H_\delta (\mathbf{x}^k, \mathbf{x}^{k-1}) & = \nabla f_1(\mathbf{x}^k) - \nabla f_2(\mathbf{x}^k) + \partial g(\mathbf{x}^k) - \delta \left \langle \nabla^2 h(\mathbf{x}), \mathbf{x}^{k-1} - \mathbf{x}^k \right \rangle,  \\
        \partial_y H_\delta (\mathbf{x}^k, \mathbf{x}^{k-1}) & = \delta \left( \nabla h(\mathbf{x}^{k-1}) - \nabla h(\mathbf{x}^k) \right).
    \end{aligned}
\end{equation}

On the other hand, with the first-order optimality condition of Subproblem~\eqref{subproblem:main_subproblem_of_iBPDCA}, for $k\geq K_0 +1$ we have 
\begin{equation} \nonumber
    -\nabla f_1({\bf y}^{k-1}) + \nabla f_2({\bf x}^{k-1}) -\frac{1}{\lambda} \left(\nabla h({\bf x}^k) -\nabla h({\bf y}^{k-1})\right) \in \partial g({\bf x}^{k}),
\end{equation}
since $f_2$ is continuously differentiable in $\mathcal{N}$ and ${\bf x}^{k-1} \in \mathcal{N}$ whenever $k\geq K_0 +1$. Combining this with \eqref{eq:partial_H_delta}, we have
\begin{equation} \nonumber
    \begin{aligned}
    &\nabla f_1(\mathbf{x}^k) -\nabla f_1({\bf y}^{k-1})
    + \nabla f_2({\bf x}^{k-1}) -\nabla f_2(\mathbf{x}^k) -\frac{1}{\lambda} \left(\nabla h({\bf x}^k) -\nabla h({\bf y}^{k-1})\right) \\ 
    & \quad - \delta \left \langle \nabla^2 h(\mathbf{x}), \mathbf{x}^{k-1} - \mathbf{x}^k \right \rangle \in \partial_x H_\delta (\mathbf{x}^k, \mathbf{x}^{k-1}).
    \end{aligned}
\end{equation}
Using this, the definition of ${\bf y}^{k-1}$ and global Lipschitz continuity of $\nabla f_1, \nabla f_2$ and $\nabla h$ on $\mathcal{N}_0$, we see that there exists a $C_1>0$ such that 
\begin{equation}\label{distance_partialH_and_zero}
    \lim_{k\rightarrow\infty} {\rm dist}\left(({\bf 0},{\bf 0}),\partial H_\delta({\bf x}^k,{\bf x}^{k-1})\right) \leq C_1 \left( \Vert {\bf x}^k -{\bf x}^{k-1}\Vert + \Vert {\bf x}^{k-1} - {\bf x}^{k-2}\Vert  \right),
\end{equation}
whenever $k \geq K_0 +1$. According to Theorem~\ref{thm:subsquential_convergence}(iii), $\Vert {\bf x}^{k+1} - {\bf x}^k\Vert \rightarrow 0$, we conclude that 
\begin{equation} \nonumber
    \lim_{k\rightarrow 0} {\rm dist}\left({\bf (0,0)},\partial H_\delta({\bf x}^k,{\bf x}^{k-1})\right) =0.
\end{equation}

(ii) According to Theorem~\ref{thm:subsquential_convergence}(iii), it can be concluded that $\Vert {\bf x}^k - {\bf x}^{k-1} \Vert \rightarrow 0$. Consequently, both ${\bf x}^k$ and ${\bf x}^{k-1}$ converge to ${\bf x}^*$. Let the set of accumulation points for the sequence $\{({\bf x}^k,{\bf x}^{k-1})\}_{k=0}^\infty$ is denoted by $\Upsilon$.
Furthermore, by utilizing Proposition~\ref{prop:convergence_property_of_Dh} and Proposition~\ref{prop:proporty_accumulation_point} (i), we can establish
\begin{align*}
    \lim_{k\rightarrow \infty} H_\delta ({\bf x}^k,{\bf x}^{k-1}) = \lim_{k\rightarrow \infty} \Psi({\bf x}^k) +\delta \lim_{k\rightarrow \infty} D_h({\bf x}^{k-1},{\bf x}^k) = \zeta.
\end{align*}
From Proposition~\ref{prop:proporty_accumulation_point}, we have $\forall ({\bf x}^*,{\bf x}^*) \in \Upsilon \text{ s.t. } {\bf x}^* \in \Omega$, $H_\delta({\bf x}^*,{\bf x}^*)=\Psi({\bf x}^ *)=\zeta$. Since ${\bf x}^*$ is arbitrary, we can conclude that $H_\delta \equiv \zeta$ on $\Upsilon$.

(iii) From Theorem~\ref{thm:subsquential_convergence}(iv), it is known that any accumulation point of $\left\{{\bf x}^k\right\}_{k=0}^\infty$ is a limiting point of Problem \eqref{eq:DCA_model}. Therefore, it is sufficient to prove that $\left\{{\bf x}^k\right\}_{k=0}^\infty$ is convergent. 
First, we consider the case that there exists $k>0$ such that $H_\delta({\bf x}^k,{\bf x}^{k-1})=\zeta$. In accordance with Proposition~\ref{prop:convergence_property_of_Dh} and Proposition~\ref{prop:proporty_accumulation_point}(i), it is known that the sequence $\left\{H_\delta({\bf x}^k,{\bf x}^{k-1})\right\}^\infty_{k=0}$ is non-increasing, and converge to $\zeta$. Hence, we have $H_\delta({\bf x}^{k+\hat{k}},{\bf x}^{k+\hat{k}-1})=\zeta$ for any $\hat{k}\geq 0$. Using \eqref{inequality_for_auxiliary_function}, we know that there exists some $C_2>0$ such that for all $k \in \mathbb{N}$  
\begin{equation} \label{difference_between_H}
    \begin{aligned}
    H_\delta ({\bf x}^k,{\bf x}^{k-1}) &\geq H_\delta({\bf x}^{k+1},{\bf x}^k) + \left(\frac{1}{\lambda} -\frac{\eta}{\kappa}-\delta \right)D_h({\bf x}^k,{\bf x}^{k+1}) \\
    &\quad + \left(\frac{1}{\lambda} -L\right) D_h({\bf x}^{k+1},{\bf y}^k) + \epsilon D_h({\bf x}^{k-1},{\bf x}^k) \\
    &\geq \epsilon D_h({\bf x}^{k-1},{\bf x}^k)
    \geq \frac{\kappa \epsilon}{2} \Vert {\bf x}^k - {\bf x}^{k-1} \Vert^2 
    \geq C_2 \Vert {\bf x}^k - {\bf x}^{k-1} \Vert ^2,
    \end{aligned}
\end{equation}
where the second-to-last inequality is derived from the $\kappa$-strong convexity property of the function $h$. Based on the above, we can conclude that ${\bf x}^k = {\bf x}^{k+\hat{k}}$ holds for all $\hat{k}\geq 0$, which means $\left\{{\bf x}^k \right\}^\infty_{k=0}$ is finitely convergent.

Second, we consider the case that $H_\delta({\bf x}^k,{\bf x}^{k-1})>\zeta$ for all $k\geq 0$. Since $H_\delta({\bf x}^k,{\bf x}^{k-1})$ is a KL function, $\Upsilon$ is a compact subset of ${\rm dom} (\partial H_\delta)$ and $H_\delta \equiv \zeta$ on $\Upsilon$ from (ii), by Lemma~\ref{thm:uniformized_KL} (Uniformized KL property), there exist $\vartheta>0,~\tau>0$ and a concave function $\varrho\in \Phi_{\tau}$ such that
\begin{equation} \label{concave_function_dist}
    \varrho'\left(H_\delta({\bf x},{\bf y})-\zeta\right) \cdot {\rm dist}\left({\bf (0,0)},\partial H_\delta({\bf x},{\bf y})\right) \geq 1, \quad \forall ({\bf x},{\bf y}) \in U,
\end{equation}
where $U= \{({\bf x},{\bf y})\in \mathbb{R}^n \times \mathbb{R}^n \mid {\rm dist}(({\bf x},{\bf y}),\Upsilon)< \vartheta\} \cap \{({\bf x},{\bf y})\in \mathbb{R}^n \times \mathbb{R}^n \mid \zeta < H_\delta({\bf x},{\bf y}))<\zeta+\tau\}$.

Since $\Upsilon$ is the set of accumulation points of $\left\{({\bf x}^k, {\bf x}^{k-1})\right\}_{k=0}^\infty$ as mentioned in (ii), and $\{{\bf x}^k\}_{k=0}^\infty$ is bounded due to Theorem~\ref{thm:subsquential_convergence}(i), we have
\begin{equation} \nonumber
    \lim_{k\rightarrow 0} {\rm dist}(({\bf x}^k,{\bf x}^{k-1}),\Upsilon) =0.
\end{equation}
Hence, there exists $K_1>0$ such that ${\rm dist}(({\bf x}^k,{\bf x}^{k-1}),\Upsilon)<\vartheta, \forall k\geq K_1$. As mentioned earlier, the sequence $\left\{H_\delta({\bf x}^k,{\bf x}^{k-1})\right\}^\infty_{k=0}$ is non-increasing and converge to $\zeta$. Consequently, there exists $K_2>0$ such that $\zeta<H_\delta({\bf x}^k,{\bf x}^{k-1})<\zeta+\tau, \forall k\geq K_2$. By setting $\overline{K}=\max\{K_0+1,K_1,K_2\}$, where $K_0$ is defined in Proof of Theorem~\ref{thm:global_convergence}(i), it follows that the sequence $\{{\bf x}^k\}_{k\geq\overline{K}}\in U$. By referencing \eqref{concave_function_dist}, we obtain
\begin{equation} \label{diff_phi_greater_than_one}
    \varrho'\left(H_\delta({\bf x}^k,{\bf x}^{k-1})-\zeta\right) \cdot {\rm dist}\left({\bf (0,0)},\partial H_\delta({\bf x}^k,{\bf x}^{k-1})\right) \geq 1, \quad \forall k\geq\overline{K}.
\end{equation}
Due to the concavity of $\varrho$, we have that $\forall k \geq \overline{K}$, 
\begin{align*}
    &\left [\varrho\left(H_\delta({\bf x}^k,{\bf x}^{k-1}) -\zeta \right) -\varrho(H_\delta({\bf x}^{k+1},{\bf x}^{k}) -\zeta ) \right]\cdot {\rm dist} \left({\bf (0,0)},\partial H_\delta({\bf x}^k,{\bf x}^{k-1})\right)\\
    &\geq \varrho'\left(H_\delta({\bf x}^k,{\bf x}^{k-1}) -\zeta \right) \cdot {\rm dist}\left({\bf (0,0)},\partial H_\delta({\bf x}^k,{\bf x}^{k-1})\right) 
     \cdot \left(H_\delta({\bf x}^k,{\bf x}^{k-1})-H_\delta({\bf x}^{k+1},{\bf x}^k)\right)\\
    &\geq H_\delta({\bf x}^k,{\bf x}^{k-1})-H_\delta({\bf x}^{k+1},{\bf x}^k)
    \geq C_2 \Vert {\bf x}^k - {\bf x}^{k-1} \Vert^2,
\end{align*}
where last inequality comes from \eqref{difference_between_H}, and the second-to-last inequality holds due to \eqref{diff_phi_greater_than_one} and that the sequence $\{H_\delta({\bf x}^k,{\bf x}^{k-1})\}^\infty_{k=0}$ is non-increasing. 
Utilizing the above, together with \eqref{distance_partialH_and_zero}, we have $\forall k \geq \overline{K}$,
\begin{align*}
    \Vert {\bf x}^k - {\bf x}^{k-1} \Vert^2 &\leq \frac{C_1}{C_2} \left [\varrho\left(H_\delta({\bf x}^k,{\bf x}^{k-1}) -\zeta \right) -\varrho \left(H_\delta({\bf x}^{k+1},{\bf x}^{k}) -\zeta \right) \right]\\
    &\quad \cdot \left(\Vert {\bf x}^k - {\bf x}^{k-1} \Vert + \Vert {\bf x}^{k-1} -{\bf x}^{k-2}\Vert \right).
\end{align*}
By taking the square root of both sides and applying the inequality of arithmetic and geometric means, we obtain
\begin{align*}
    \Vert {\bf x}^k - {\bf x}^{k-1} \Vert &\leq \sqrt{\frac{2 C_1}{C_2} \left [\varrho\left(H_\delta({\bf x}^k,{\bf x}^{k-1}) -\zeta \right) -\varrho\left(H_\delta({\bf x}^{k+1},{\bf x}^{k}) -\zeta \right) \right]}\\
    &\quad \cdot \sqrt{\frac{\Vert {\bf x}^k - {\bf x}^{k-1} \Vert + \Vert {\bf x}^{k-1} -{\bf x}^{k-2}\Vert}{2}} \\
    &\leq \frac{C_1}{C_2} \left [\varrho\left(H_\delta({\bf x}^k,{\bf x}^{k-1}) -\zeta \right) -\varrho\left(H_\delta({\bf x}^{k+1},{\bf x}^{k}) -\zeta \right) \right]\\
    &\quad + \frac{\Vert {\bf x}^k - {\bf x}^{k-1} \Vert + \Vert {\bf x}^{k-1} -{\bf x}^{k-2}\Vert}{4}.
\end{align*}
Then we have
\begin{align*}
    \frac{1}{2}\Vert {\bf x}^k - {\bf x}^{k-1} \Vert &\leq \frac{C_1}{C_2} \left [\varrho\left(H_\delta({\bf x}^k,{\bf x}^{k-1}) -\zeta \right) -\varrho\left(H_\delta({\bf x}^{k+1},{\bf x}^{k}) -\zeta \right) \right]\\
    &\qquad + \frac{1}{4} ( \Vert {\bf x}^{k-1} -{\bf x}^{k-2}\Vert-\Vert {\bf x}^k - {\bf x}^{k-1}\Vert).
\end{align*}
Summing the above from $k=\overline{K}$ to $\infty$,
\begin{align*}
    \sum_{k=\overline{K}}^\infty \Vert {\bf x}^k - {\bf x}^{k-1} \Vert \leq \frac{2C_1}{C_2}\varrho\left(H_\delta({\bf x}^{\overline{K}},{\bf x}^{\overline{K}-1}) -\zeta\right) +\frac{1}{2} \Vert {\bf x}^{\overline{K}-1} -{\bf x}^{\overline{K}-2} \Vert <\infty,
\end{align*}
which implies the global convergence of $\{{\bf x}^k\}$ and summability of $\{\Vert{\bf x}^{k+1}-{\bf x}^k\Vert\}_{k\geq 0}$. This completes the proof.

\bibliography{references}
\end{document}